\documentclass[11pt,reqno]{amsart}
\usepackage{amsmath, amsfonts, amsthm, amssymb, mathrsfs, amscd, graphicx, cite, color}
\usepackage{latexsym,hyperref}

\title[Vanishing dissipation limit for 3D Navier-Stokes equations]
{Vanishing dissipation limit to the planar rarefaction wave for the three-dimensional compressible Navier-Stokes-Fourier equations}

\author[L.-A. Li]{Lin-An Li}
\address{Research Center for Mathematics and Mathematics Education, Beijing Normal University at Zhuhai, 519087, China;
  and Laboratory of Mathematics and Complex Systems (Ministry of Education), School of Mathematical Sciences, Beijing Normal University, Beijing 100875,  P. R. China.}
\email{linanli@amss.ac.cn}

\author[D. Wang]{Dehua Wang}
\address{Department of Mathematics, University of Pittsburgh, Pittsburgh, PA 15260, USA.}
\email{dwang@math.pitt.edu}

\author[Y. Wang]{Yi Wang}
\address{
Institute of Applied Mathematics, AMSS, Chinese Academy of Sciences, Beijing 100190, P. R. China;
 and School of Mathematical Sciences, University of Chinese Academy of Sciences, Beijing 100049, P. R. China.}
\email{wangyi@amss.ac.cn}

\newtheorem{theorem}{Theorem}[section]
\newtheorem{lemma}{Lemma}[section]

\newtheorem{proposition}{Proposition}[section]

\theoremstyle{definition}

\theoremstyle{remark}
\newtheorem{remark}{Remark}[section]

\newcommand{\bbr}{\mathbb R}

\newcommand{\bbt}{\mathbb T}

\newcommand{\bv}{\mbox{\boldmath $v$}}
\newcommand{\bq}{\mbox{\boldmath $q$}}

\renewcommand{\div}{{\rm div}}

\def\charf {\mbox{{\text 1}\kern-.30em {\text l}}}

\def\lam{\lambda}  

\def\di{\displaystyle}

\begin{document}

\date{\today}

\begin{abstract}

We study the vanishing dissipation limit of the three-dimensional (3D) compressible Navier-Stokes-Fourier equations to the corresponding 3D full Euler equations. Our results are twofold. First,  we prove that the 3D compressible Navier-Stokes-Fourier equations admit a family of smooth solutions  that converge  to the planar rarefaction wave solution of the 3D compressible Euler equations with arbitrary strength. Second,  we obtain a uniform convergence rate  in terms of the viscosity and heat-conductivity coefficients. For this multi-dimensional problem, we first need to introduce the hyperbolic wave to recover the physical dissipations of the  inviscid rarefaction wave profile as in our previous work \cite{LWW2} on the two-dimensional (2D) case. However, due to the 3D setting that makes the analysis significantly more challenging  than the 2D problem, the hyperbolic scaled variables for the space and time could not be  used to normalize the dissipation coefficients as in the 2D case. Instead, the analysis of the 3D case is carried out in the original non-scaled variables, and consequently  the dissipation terms are more singular compared with the 2D scaled case. Novel ideas and techniques are developed to establish the uniform estimates. In particular, more accurate {\it a priori} assumptions (see \eqref{PA}) with respect to the dissipation coefficients are crucially needed for the stability analysis, and some new observations on the cancellations of the physical structures for the flux terms are essentially used to justify the 3D limit. Moreover, we find that the decay rate with respect to the dissipation coefficients is determined by the nonlinear flux terms in the original variables for the 3D limit in this paper, but fully determined by the error terms in the scaled variables for the 2D case in \cite{LWW2}.

\end{abstract}

\date{\today}

\maketitle 
%
%
\section{Introduction}  
\setcounter{equation}{0}
We are concerned with the vanishing dissipation limit of the three-dimensional Navier-Stokes-Fourier equations describing the motion of compressible, viscous and heat-conducting flows:
\begin{equation}  \label{NS}
\begin{cases}
\displaystyle \rho_t + \div(\rho \bv) = 0,    \\ 
\displaystyle (\rho \bv)_t + \div(\rho \bv \otimes \bv) + \nabla p = \div\mathbb{S}, \\
\displaystyle (\rho E)_t+\div (\rho E \bv + p \bv+ \bq)=\div(\bv\mathbb{S}),
\end{cases}
\end{equation}
where $\rho=\rho(t,x)$ denotes the density, $\bv=\bv(t,x)=(v_1,v_2,v_3)(t,x)\in \mathbb{R}^3$ the velocity, and $p=p(t,x)$ the pressure, with
the temporal variable  $t\in\mathbb{R}^+=(0,\infty)$ 
and the spatial variable $x=(x_1,x_2,x_3)\in \mathbb{R}^3$.
Moreover, 
$E=e+\frac12|\bv|^2$ is the specific total energy with the specific internal energy $e=e(t,x)$, and $\mathbb{S}$ is the viscous stress tensor defined by
\begin{equation}
\mathbb{S}= 2\mu_1 \mathbb{D}(\bv) + \lambda_1\div\bv\mathbb{I},
\end{equation}
where $\mathbb{D}(\bv)=\frac{\nabla \bv + (\nabla \bv)^\top}{2}$ is the deformation tensor with $(\nabla \bv)^\top$ denoting the transpose of the matrix $\nabla \bv$,  $\mathbb{I}$ represents the $3\times3$ identity matrix, and the two constants $\mu_1$ and $\lambda_1$ denote the shear  and the bulk viscosity coefficients   
satisfying the physical restrictions:
\begin{equation} \label{VIS}
\mu_1>0,\quad 2\mu_1 + 3\lambda_1 \geq0.
\end{equation}
By the Fourier laws, the heat flux $\bq$ is given by
\begin{equation}
\bq=-\kappa_1\nabla \theta,
\end{equation}
where $\theta=\theta(t,x)$ denotes  the   absolute temperature, and    the constant $\kappa_1>0$ is the heat-conductivity coefficient. In the present paper we take
\begin{equation}\label{1.4}
\mu_1=\mu\varepsilon, \quad\lambda_1=\lambda\varepsilon,\quad \kappa_1=\kappa\varepsilon,
\end{equation}
 where $\varepsilon>0$ is the vanishing dissipation parameter, and $\mu, \lambda$ and $\kappa$ are the given uniform-in-$\varepsilon$ constants with $\mu, \lambda$ still satisfying the physical condition \eqref{VIS} and $\kappa>0$. The equations \eqref{NS}  describe  the mass continuity, the balance of momentum, and the conservation of total energy for the viscous heat-conducting flow. For the ideal polytropic flow  the pressure $p$ and the internal energy $e$ satisfy the following constitutive relations:
\begin{equation}
p=R\rho\theta=A\rho^{\gamma}\exp\Big(\frac{\gamma-1}{R}S\Big),\quad e=\frac{R}{\gamma-1}\theta,
\end{equation}
where $S$ denotes the entropy,  and $\gamma>1$, $A$ and $R$ are all positive fluid constants.

 In this paper, we consider the viscous system \eqref{NS}  in
 the domain $\Omega=\mathbb{R}\times\mathbb{T}^2$ with 
$\mathbb{T}^2= (\mathbb{R}/ \mathbb{Z})^2$ denoting a two-dimensional unit flat torus, subject to 
the following initial condition: 
\begin{equation} \label{NSI}
(\rho, \bv,\theta)(0, x)=(\rho_0, \bv_0, \theta_0)(x),
\end{equation}
where $\rho_0, \theta_0>0$ and $\bv_0 := (v_{10}, v_{20}, v_{30})$ satisfy the periodic boundary condition for $(x_2,x_3)\in\mathbb{T}^2$.
We shall study the vanishing dissipation limit  in the case of the planar rarefaction wave for the dissipative system \eqref{NS},  with the following far field condition of   solutions  imposed in the $x_1$-direction:
\begin{equation}\label{ff}
(\rho, \bv, \theta)(t, x) \to (\rho_\pm, \bv_\pm, \theta_\pm), \quad {\rm as}\quad x_1\to \pm\infty,~ t>0,
\end{equation}
where $\bv_\pm :=(v_{1\pm},0,0)$ and $\rho_\pm, ~\theta_\pm>0, ~v_{1\pm}$ are prescribed constant states, 
such that the two end states $(\rho_\pm, v_{1\pm}, \theta_\pm)$ are connected by the rarefaction wave of the following Riemann problem for the one-dimensional (1D)    Euler equations:
\begin{equation}  \label{ES}
\begin{cases}
\displaystyle \rho_t + (\rho v_1)_{x_1} = 0,               \quad\qquad \qquad x_1 \in \bbr, ~t > 0, \\
\displaystyle (\rho v_1)_t + (\rho v_1^2 + p)_{x_1} = 0, \\
\displaystyle (\rho E)_t + (\rho E v_1 + pv_1)_{x_1} = 0,
\end{cases}
\end{equation}
with  the   initial data:
\begin{equation} \label{ESI}
(\rho, v_1, \theta)(0, x_1) = (\rho_0^r, v_{10}^r, \theta_0^r)(x_1) = \begin{cases}
(\rho_-, v_{1-}, \theta_-),     \quad x_1 < 0,\\
(\rho_+, v_{1+}, \theta_+),     \quad x_1 > 0.
\end{cases}
\end{equation}
When $\varepsilon\to 0$, the solutions of the 3D compressible Navier-Stokes-Fourier equations \eqref{NS}-\eqref{NSI} with the far field condition \eqref{ff} connected by the rarefaction wave in \eqref{ES}-\eqref{ESI} formally converge to the   solutions of the 3D   Euler equations:
\begin{equation}  \label{2ES}
\begin{cases}
\displaystyle \rho_t + \div(\rho \bv) = 0,               \quad\quad\qquad \qquad (x_1, x_2, x_3) \in \Omega, ~t > 0, \\
\displaystyle (\rho \bv)_t + \div(\rho \bv \otimes \bv) + \nabla p = 0, \\
\displaystyle (\rho E)_t + \div (\rho E \bv + p \bv) = 0,
\end{cases}
\end{equation}
with the following planar Riemann initial data:
\begin{equation} \label{2ESI}
(\rho, \bv, \theta)(0, x) = (\rho_0^r, \bv_0^r, \theta_0^r)(x_1) = \begin{cases}
(\rho_-, \bv_-, \theta_-),     \quad x_1 < 0, \\
(\rho_+, \bv_+, \theta_+),     \quad x_1 > 0.
\end{cases}
\end{equation}

It should be noted that, even though the tangential components $v_2, v_3$ are continuous on the both sides of the plane $\{x_1 = 0\}$ and the jump discontinuity for the velocity $\bv$ lies only in the normal component $v_1$ as in \eqref{2ESI},
the multi-dimensional Riemann problem \eqref{2ES}-\eqref{2ESI} has some essential difference from the one-dimensional   Riemann problem \eqref{ES}-\eqref{ESI}. For instance, 
Chiodaroli, De Lellis and Kreml \cite{CDK} and Chiodaroli and Kreml \cite{CK} first proved that there exist infinitely many bounded admissible weak solutions to the two-dimensional planar shock Riemann problem \eqref{2ES}-\eqref{2ESI} satisfying the natural entropy condition 
and their construction of multi-dimensional weak solutions is formulated by the convex integration method  (c.f. De Lellis and Szekelyhidi \cite{DS}), which may not be applied directly to the 1D Riemann problem \eqref{ES}-\eqref{ESI}.
Their results were extended to the 2D Riemann initial data \eqref{2ESI} with planar shock connected with contact discontinuity or rarefaction wave 
in \cite{KM,BCK}, and the global instability of the multi-dimensional planar shocks for the compressible isothermal Euler flow was shown in \cite{LXY}. 
However, the uniqueness of the uniformly bounded admissible weak solution was obtained in \cite{CC,FK,FKV} 
for the corresponding multi-dimensional Riemann solution to \eqref{2ES}-\eqref{2ESI} containing only planar rarefaction wave.
 Hence the uniqueness of planar rarefaction wave  for the multi-dimensional Riemann problem is exactly the same as the one-dimensional case and is in sharp contrast with the non-uniqueness of the 2D planar shock.
We remark that the non-uniqueness   for the 2D planar shock Riemann solution confirms that the usual entropy inequality could not serve as a criterion to select a unique solution in the case of systems of conservation laws in more than one spatial dimension. Therefore one possible alternative selection criterion is to study the vanishing dissipation limit in the multi-dimensional setting especially for the compressible Navier-Stokes-Fourier equations \eqref{NS} with physical dissipations, even though it is still a big challenging open problem in the general setting due to the complicated and ambiguous structure of the multi-dimensional Euler equations \eqref{2ES}.

On the other hand, compared with the multi-dimensional planar shock wave, the multi-dimension planar rarefaction wave seems more stable for the multi-dimensional compressible Euler equations by the uniqueness results in \cite{CC, FK, FKV}, which motivates us to justify mathematically the vanishing physical dissipation limit of the 3D compressible Navier-Stokes-Fourier equations \eqref{NS}-\eqref{NSI} towards the planar rarefaction wave solution of  the 3D Euler equations \eqref{2ES}-\eqref{2ESI} in the present paper.

 The vanishing viscosity limit  in the one-dimensional case has been studied extensively in literature, including the basic wave patterns.
 Here we briefly mention the works of \cite{GX,Y,BB} for the 1D hyperbolic systems of conservation laws   with  artificial viscosity,
 \cite{HL,X-1,H-L-W,JNS,KV,KV1,W-2,M,HWY-1,HJW,HWY-2,HWWY,CP} for the 1D compressible Navier-Stokes equations with physical viscosities,   \cite{CG18a, CLQ18, CEIV, DN2018, Kato84, BTW18, Masmoudi2007}  for other relevant results; and we refer   to our previous paper \cite{LWW2} for more detailed survey on the 1D problem.
However,   the  multi-dimensional vanishing viscosity limit is a challenging problem and has been less developed. In particular, for the compressible viscous system \eqref{NS} there are only a few results to date on the vanishing dissipation limit towards the planar basic wave patterns. Very recently, we \cite{LWW2} first proved the vanishing viscosity limit to the planar rarefaction wave for the two-dimensional compressible isentropic Navier-Stokes equations in $\bbr\times\bbt$ by introducing the hyperbolic wave and using the hyperbolic scaling variables $\frac x\varepsilon$ and $\frac t\varepsilon$.
The time-asymptotic stability of the planar rarefaction wave to the 2D/3D compressible Navier-Stokes equations with fixed viscosity and heat-conductivity coefficients was proved  in \cite{LW,LWW}. We remark that there are substantial differences between the  asymptotic stability   and the vanishing viscosity limit of planar rarefaction wave in the multi-dimensional setting,  since the error terms for the inviscid rarefaction wave profile have sufficient time-decay rate for the  asymptotic stability, but do not  have enough decay with respect to the dissipation coefficients for the vanishing dissipation limit in the multi-dimensional case.
The goal of this  paper  is to justify mathematically the vanishing dissipation limit to the planar rarefaction wave for the 3D compressible Navier-Stokes-Fourier system \eqref{NS} under the physical constraints \eqref{VIS} and \eqref{1.4}.

In contrast to the 1D vanishing dissipation limits in \cite{X-1, JNS, H-L-W}, the new difficulties here lie in the spatially higher dimensionality for the wave propagation along the transverse $x_2, x_3$-directions and their interactions with the rarefaction wave in the $x_1$-direction. Accordingly, a new hyperbolic wave is crucially introduced to recover the physical dissipation of the full Navier-Stokes-Fourier system for the inviscid rarefaction wave profile, which is partially inspired by  \cite{HWY-2, LWW2}.
Without this hyperbolic wave but with only the rarefaction wave itself as the solution profile, the $H^2$-norm of the  perturbation of the 3D compressible Navier-Stokes-Fourier system around the planar rarefaction wave is not uniform-in-$\varepsilon$ and consequently the vanishing dissipation limit of the planar rarefaction wave could not be justified as in Theorem \ref{theorem1}.

In comparison with the two-dimensional vanishing viscosity limit  for the compressible isentropic Navier-Stokes equations in our previous work \cite{LWW2}, the 3D compressible Navier-Stokes-Fourier system \eqref{NS} is a more physical model involving the thermal conduction, but is significantly more difficult in analysis,  and hence new techniques are needed    mainly due to the higher dimensionality and the additional  energy equation. First, our approach in \cite{LWW2} for the 2D case could not be applied  directly to three-dimensional case here due to the fact that the additional $x_3$ dimensionality would be $O(\frac 1\varepsilon)$ order in the scaled variable $\frac{x_3}\varepsilon$.
 If we still use the hyperbolic scaled variables for the space and time like $\frac x\varepsilon$ and $\frac t \varepsilon$ to normalize the dissipations to be $O(1)$ order as the two-dimensional case in \cite{LWW2}, then the same uniform estimation process with respect to the dissipation coefficients as in \cite{LWW2} could not be obtained directly due to 3D setting. Therefore we use the original non-scaled variables $x$ and $t$ here, but then the dissipative terms are more singular compared with the scaled variables case in \cite{LWW2}.
Consequently, we need to carry out the more accurate {\it a priori} assumptions with respect to the dissipation coefficients (see \eqref{PA}) to overcome this difficulty.  Second, we discover some new   cancellations of the physical structures for the flux terms and viscous terms, which are essentially used to justify the limit process for the 3D planar rarefaction wave.
With these observations and more delicate analysis of energy estimates, we can close the {\it a priori} assumptions and prove the 3D vanishing dissipation limit. The detailed result can be found in Theorem \ref{theorem1} below. Note also that the periodicity of the domain $\Omega$ in $x_2$ and $x_3$ directions and the original non-scaled spatial variables are crucially utilized in the lower order estimates.
For the derivative estimates in Lemmas \ref{lemma4.2} and \ref{lemma4.3} we mainly apply the cancelations between the flux terms, which is quite different from the two-dimensional case in \cite{LWW2}. Moreover, it is observed that the decay rate with respect to the dissipation coefficients is determined by the nonlinear terms in the original variables for  the 3D limit here, while for the two-dimensional case \cite{LWW2} the decay rate is fully determined by the error terms in the scaled variables.

Next we describe the planar rarefaction wave to the 3D Euler system \eqref{2ES}, which can be viewed as a 1D rarefaction wave to the 1D Euler system \eqref{ES}. For $\rho >0$ and $\theta>0$, the strictly hyperbolic system \eqref{ES} has three distinct eigenvalues:
\[
\lambda_j(\rho, v_1, S) = v_1 + (-1)^{\frac{j+1}{2}}\sqrt{p_\rho(\rho, S)},~ j=1,3,\qquad
\lambda_2(\rho, v_1, S) = v_1,~
\]
with the corresponding right eigenvectors
\[
r_j(\rho, v_1, S) = \Big((-1)^{\frac{j+1}{2}}\rho, \sqrt{p_\rho(\rho, S)}, 0\Big)^\top,~j=1,3, \qquad r_2(\rho, v_1, S)=(p_{\scriptscriptstyle S}, 0, -p_\rho)^\top,
\]
such that
\[
\nabla_{(\rho, v_1, S)}\lambda_j(\rho, v_1, S) \cdot r_j(\rho, v_1, S) \neq 0,~ j=1,3, \quad
{\rm and}\quad
\nabla_{(\rho, v_1, S)}\lambda_2(\rho, v_1, S) \cdot r_2(\rho, v_1, S) \equiv 0.
\]
Thus the two $j$-Riemann invariants $\Sigma_j^{(i)}(i=1,2, j=1,3)$ can be defined by
\begin{equation} \label{RI}
\Sigma_j^{(1)} = v_1 + (-1)^{\frac{j-1}{2}}\int^{\rho}\frac{\sqrt{p_z(z, S)}}{z}dz, \qquad
\Sigma_j^{(2)} = S,
\end{equation}
such that
\[
\nabla_{(\rho, v_1, S)} \Sigma_j^{(i)}(\rho, v_1, S) \cdot r_j(\rho, v_1, S) \equiv 0, \quad i=1,2,~ j=1,3.
\]
Here we only consider the single 3-rarefaction wave without loss of generality, while the single 1-rarefaction wave and the superposition of these two   waves could be treated similarly. Along the 3-rarefaction wave curve, the 3-Riemann invariant $\Sigma_3^{(i)}(i=1,2)$ keeps constant and the third eigenvalue $\lambda_3(\rho, v_1, S)$ is expanding, in particular,
\begin{equation} \label{RC}
\Sigma_3^{(i)}(\rho_+, v_{1+}, S_+) = \Sigma_3^{(i)}(\rho_-, v_{1-}, S_-), \qquad \lambda_3(\rho_+, v_{1+}, S_+) > \lambda_3(\rho_-, v_{1-}, S_-).
\end{equation}
The Riemann problem \eqref{ES}-\eqref{ESI} has a self-similar 3-rarefaction wave solution $(\rho^r, v_1^r, \theta^r)(\frac{x_1}{t})$ (cf. \cite{L}).
Corespondingly, the planar rarefaction wave solution to the 3D Euler equations \eqref{2ES}-\eqref{2ESI} could be defined as $(\rho^r, \bv^r, \theta^r)(\frac{x_1}{t})$ with $\bv^r := (v_1^r, 0, 0)^\top$.

Now our main result can be stated as follows.

\begin{theorem} \label{theorem1}
	Let $(\rho^r, \bv^r, \theta^r)(\frac{x_1}{t})$ be the planar 3-rarefaction wave to the 3D compressible Euler system \eqref{2ES} and  $T>0$ be any fixed  
	time. Then there exists a   constant $\varepsilon_0>0$ such that, for any $\varepsilon \in (0, \varepsilon_0)$,
	the system \eqref{NS} has a family of smooth solutions $(\rho^\varepsilon, \bv^\varepsilon, \theta^\varepsilon)(t, x)$ up to time $T$
	satisfying
	\[
	\begin{cases}
	(\rho^\varepsilon - \rho^r, \bv^\varepsilon - \bv^r, \theta^\varepsilon - \theta^r) \in C^0(0, T; L^2(\Omega)), \\
	(\nabla\rho^\varepsilon, \nabla\bv^\varepsilon, \nabla\theta^\varepsilon) \in C^0(0, T; H^1(\Omega)), \\
	(\nabla^3\bv^\varepsilon, \nabla^3\theta^\varepsilon) \in L^2(0, T; L^2(\Omega));
	\end{cases}
	\]
and,  for any small constant $h>0$, there exists a constant $C_{h,T}>0$ independent of $\varepsilon$, such that
	\begin{equation} \label{AB}
	\sup_{h \leq t \leq T} \left\|(\rho^\varepsilon, \bv^\varepsilon, \theta^\varepsilon)(t, x) - (\rho^r, \bv^r, \theta^r)\left(\frac{x_1}{t}\right)\right\|_{L^\infty(\Omega)} \leq C_{h,T}\ \varepsilon^{\frac{1}{6}}|\ln \varepsilon|^2.
	\end{equation}
Furthermore, as   $\varepsilon \to 0+$, the solution $(\rho^\varepsilon, \bv^\varepsilon, \theta^\varepsilon)(t, x)$ converges to the planar 3-rarefaction wave fan $(\rho^r, \bv^r, \theta^r)(\frac{x_1}{t})$ pointwise in $\bbr^+\times\Omega$. 

\end{theorem}

\begin{remark}
Theorem \ref{theorem1} presents rigorously the vanishing dissipation limit to the planar rarefaction wave with arbitrarily large strength for the three-dimensional compressible Navier-Stokes-Fourier system \eqref{NS}. We remark that the corresponding vanishing dissipation limits of the system \eqref{NS} for the planar shock or planar contact discontinuity case are still completely open to our knowledge.
\end{remark}
\begin{remark}
As in the two-dimensional case \cite{LWW2}, in order to justify the vanishing dissipation limit we need to introduce the hyperbolic wave to recover the physical dissipations of the   system \eqref{NS} for the inviscid approximate rarefaction wave profile. Otherwise, the $H^2$-norm of the   perturbation of the 3D   system  \eqref{NS} around the planar rarefaction wave is not uniform in $\varepsilon$.

\end{remark}

\begin{remark}
Note also that our vanishing dissipation limit to the single planar 3-rarefaction wave in Theorem \ref{theorem1} could be applied to the limit problem for the superposition of two planar rarefaction waves in the first and the third characteristic families for  \eqref{NS},  provided that the additional wave interaction estimates between  the two rarefaction waves are considered.
\end{remark}

The rest of the paper is organized as follows. The approximate rarefaction wave to the Euler system \eqref{ES} or \eqref{2ES} will be constructed and the hyperbolic wave will be introduced in Section 2.
Then in Section 3, the system for  the perturbation of the solution to the  system \eqref{NS} around the solution profile consisting of   the approximate rarefaction wave and the hyperbolic wave  will be reformulated and the Theorem \ref{theorem1} will be proved. In Section 4, the detailed {\it a priori} estimates for the perturbation system will be carried out by using the $L^2$ energy method.

The following notation will be used in this paper. Denote by $H^l(\Omega)(l \geq 0, l\in \mathbb{Z})$ the usual Sobolev space with the norm $\|\cdot\|_l$. Write $L^2(\Omega) := H^0(\Omega)$ with the simplified notation $\|\cdot\| := \|\cdot\|_0$.  We denote by  $C$   a generic positive constant  that does not depend on $\varepsilon, \delta$ or $T$, but may depend on $(\rho_\pm, v_{1\pm}, \theta_\pm)$; and  denote  by $C_T$  a positive constant that does not depend on $\varepsilon$ or $\delta$, but may depend on $T$.

\bigskip
%
%
\section{Construction of Solution Profile and Preliminaries}
\setcounter{equation}{0}



This section is devoted to the construction of the solution profile including the approximate rarefaction wave for the Euler system \eqref{ES} or \eqref{2ES} and then the hyperbolic wave. 

\subsection{Smooth approximate rarefaction wave}

We first construct a smooth approximate rarefaction wave through the Burgers' equation as in \cite{X-1, LWW2, H-L, H-L-W}.
If $B_- < B_+$, then the Riemann problem of the inviscid Burgers' equation
\begin{equation} \label{BE}
\begin{cases}
\displaystyle B_t + BB_{x_1} = 0, \\
\displaystyle B(0, x_1) = B_0^r(x_1) = \begin{cases}
B_-,  \quad x_1 < 0,\\
B_+,  \quad x_1 > 0,
\end{cases}
\end{cases}
\end{equation}
admits a self-similar rarefaction wave fan solution $B^r(t, x_1) = B^r(x_1/t)$ given explicitly by
\begin{equation} \label{BES}
B^r(t, x_1) = B^r(\frac {x_1}{t}) = \begin{cases}
B_- , \qquad x_1 < B_-t, \\
\frac{x_1}{t}, \qquad B_-t \leq x_1 \leq B_+t, \\
B_+ , \qquad x_1 > B_+t.
\end{cases}
\end{equation}
Set $B_\pm = \lambda_3(\rho_\pm, v_{1\pm}, \theta_\pm)$ and then the self-similar 3-rarefaction wave $(\rho^r, v_1^r, \theta^r)(t, x_1) = (\rho^r, v_1^r, \theta^r)(\frac{x_1}t)$ to the Riemann problem \eqref{ES}-\eqref{ESI} can be given explicitly by
\begin{equation}\label{3r}
\begin{array}{ll}
  \lambda_3(\rho^r, v_1^r, \theta^r)(t, x_1) = B^r(t, x_1), \\
  \Sigma_3^{(i)}(\rho^r,v_1^r,\theta^r)(t, x_1) = \Sigma_3^{(i)}(\rho_\pm,v_{1\pm},\theta_\pm),\quad i=1,2,
\end{array}
\end{equation}
where $\Sigma_3^{(i)}~(i=1,2)$ are the 3-Riemann invariants defined in \eqref{RI}.

Since the self-similar rarefaction wave $B^r(t, x_1)$ in \eqref{BES} is only Lipschitz continuous, in order to justify the vanishing dissipation limit to the planar rarefaction wave of the compressible Navier-Stokes equations \eqref{NS} with second order derivatives, we need to construct an approximate smooth rarefaction wave profile by using the following Burgers' equation (c.f. \cite{X-1, H-L-W}):
\begin{equation} \label{ABE}
\begin{cases}
\displaystyle \bar B_t + \bar B \bar B_{x_1} = 0, \\
\displaystyle \bar B(0, x_1) = \bar B_0(x_1) = \frac{B_+ +B_-}{2} + \frac{B_+ - B_-}{2} \tanh \frac{x_1}{\delta},
\end{cases}
\end{equation}
where the approximate parameter $\delta>0$ is chosen by
\begin{align}\label{del}
	\delta := \varepsilon^b |\ln\varepsilon|,
\end{align}
with the power $b$ being a positive constant to be determined. In fact, we take $b=\frac{1}{6}$, i.e. $\delta=\varepsilon^{\frac 16}|\ln\varepsilon|,$ as explained later in \eqref{b} in order to obtain the sharp decay rate in the present paper. Since $\bar B_0^\prime(x_1)>0$, the Burgers' problem \eqref{ABE} has a unique, global classical solution $\bar B(t, x_1)$ satisfying $\bar B_{x_1}>0$ for all $t>0$ by the characteristic methods (cf. \cite{X-1, H-L-W}).



Similarly, the self-similar 3-rarefaction wave $(\rho^r, v_1^r, \theta^r)(t, x_1) = (\rho^r, v_1^r, \theta^r)(\frac{x_1}t)$ defined in \eqref{3r} is Lipschitz continuous. Correspondingly, the smooth approximate rarefaction wave $(\bar{\rho}, \bar{v}_1, \bar{\theta})$ $(t, x_1)$ can be constructed by
\begin{align}
  \begin{aligned} \label{AR}
    &\lambda_3(\bar{\rho}, \bar{v}_1, \bar{\theta})(t, x_1) = \bar B(t, x_1), \\
    &\Sigma_3^{(i)}(\bar{\rho},\bar{v}_1,\bar{\theta})(t, x_1) = \Sigma_3^{(i)}(\rho_\pm, v_{1\pm}, \theta_\pm), \quad i=1,2,
  \end{aligned}
\end{align}
where $\bar B(t, x_1)$ is the classical solution to the Burgers' equation \eqref{ABE}.
It could be checked that the above approximate rarefaction wave $(\bar{\rho}, \bar{v}_1, \bar{\theta})(t,x_1)$ satisfies the one-dimensional compressible Euler system:
\begin{equation} \label{ANS}
  \begin{cases}
    \displaystyle \bar{\rho}_t + (\bar{\rho}\bar{v}_1)_{x_1} = 0, \\
    \displaystyle (\bar{\rho} \bar v_1)_t + (\bar{\rho} \bar v_1^2 + \bar{p})_{x_1} = 0, \\
    \displaystyle \frac{R}{\gamma-1}\left[(\bar\rho\bar\theta)_t+(\bar\rho \bar v_1\bar\theta)_{x_1}\right] + \bar p\bar v_{1x_1} = 0,  
  \end{cases}
\end{equation}
with the initial value $(\bar{\rho}_0, \bar{v}_{10}, \bar\theta_0)(x_1):=(\bar{\rho}, \bar v_1, \bar\theta)(0, x_1)$.
With the supplement $v_2^r = v_3^r \equiv 0$ to the one-dimensional 3-rarefaction wave $(\rho^r, v_1^r, \theta^r)(\frac{x_1}t)$,
we see that $(\rho^r, \bv^r, \theta^r)(\frac{x_1}t)$ is the unique planar rarefaction wave solution to the three-dimensional Riemann problem \eqref{2ES}-\eqref{2ESI}. Then the corresponding approximate rarefaction wave is $(\bar \rho, \bar{\bv},\bar\theta)$ with $(\bar \rho, \bar v_1,\bar\theta)$ defined in \eqref{AR} and $\bar v_2 = \bar v_3\equiv0$.

\

The next lemma follows directly from the properties of $\bar B(t, x_1)$ (c.f. \cite{H-L-W}).
\begin{lemma} \label{lemma2.2}
	The approximate smooth 3-rarefaction wave $(\bar{\rho}, \bar v_1, \bar\theta)$ constructed in \eqref{AR} satisfies the following properties:
	
	\item[(i)]$\bar{v}_{1x_1} = \frac{2}{\gamma + 1}\bar B_{x_1} > 0$, $\bar{\rho}_{x_1} = \frac{1}{\sqrt{R\gamma \rho_+^{1-\gamma}\theta_+}}\bar{\rho}^{\frac{3 - \gamma}{2}} \bar{v}_{1x_1}>0$ and $\bar{\theta}_{x_1} = \frac{\gamma - 1}{\sqrt{R\gamma}}\sqrt{\bar{\theta}}\bar{v}_{1x_1} > 0$ for all $x_1 \in \bbr$ and $t \geq 0$.
	
	\item[(ii)] For all $t \geq 0, \delta > 0$ and $p \in [1, + \infty]$, the following estimates hold:
	\begin{align*}
	&\|(\bar{\rho}_{x_1}, \bar{v}_{1x_1}, \bar\theta_{x_1})\|_{L^p(\mathbb{R})} \leq C  (\delta + t)^{-1+1/p}, \\
	&\|(\bar{\rho}_{x_1x_1}, \bar{v}_{1x_1x_1}, \bar\theta_{x_1x_1})\|_{L^p(\mathbb{R})} \leq C (\delta + t)^{-1} \delta^{-1+1/p}, \\
	&\|(\bar{\rho}_{x_1x_1x_1}, \bar{v}_{1x_1x_1x_1}, \bar\theta_{x_1x_1x_1})\|_{L^p(\mathbb{R})} \leq C (\delta + t)^{-1} \delta^{-2+1/p}.
	\end{align*}
	
	\item[(iii)]There exists a constant $\delta_0 \in (0, 1)$ and a uniform constant $C$ such that for all $\delta \in (0, \delta_0]$ and $t > 0$,
	\[
	\left\|(\bar{\rho}, \bar v_1, \bar\theta)(t, \cdot) - (\rho^r, v_1^r, \theta^r)\left(\frac{\cdot}{t}\right)\right\|_{L^\infty(\mathbb{R})} \leq C \delta \frac{ \left[\ln(1 + t) + |\ln \delta|\right]}{t}.
	\]
\end{lemma}

\begin{remark}\label{remark-rare} In order to justify the vanishing dissipation limit to the planar rarefaction wave of 3D Navier-Stokes-Fourier system \eqref{NS} in Theorem \ref{theorem1}, if we only use the smooth rarefaction wave $(\bar{\rho}, \bar{v}_1, \bar\theta)(t, x_1)$ defined in \eqref{AR} as the approximate wave profile, then the error 
arising from the viscous dissipation terms to the Navier-Stokes-Fourier system \eqref{NS} around the inviscid approximate rarefaction wave  
makes it difficult to obtain the desired uniform estimates with respect to the dissipation coefficients.
\end{remark}

\subsection{Hyperbolic wave}

In order to overcome the difficulties mentioned in Remark \ref{remark-rare}, as in two-dimensional case \cite{LWW2}, a new wave, called hyperbolic wave, need to be introduced to recover the physical dissipations for the inviscid rarefaction wave profile, which is crucial for our justification of vanishing dissipation limit for the 3D Navier-Stokes-Fourier system \eqref{NS}. However, compared with the two-dimensional isentropic case in \cite{LWW2}, the hyperbolic wave here has one more linearly degenerate eigenvalue in the second characteristic field, besides the two genuinely nonlinear characteristic fields. 
We now present a detailed construction of this hyperbolic wave. Denote the hyperbolic wave by $\textbf{z}:=(z_1,z_2,z_3)^\top(t, x_1)$ satisfying the following linearized hyperbolic system around the smooth rarefaction wave $(\bar{\rho}, \bar{v}_1, \bar\theta)(t, x_1)$ with the viscous dissipation terms as the source:
\begin{equation} \label{HW}
\left\{
\begin{array}{ll}
\textbf{z}_t +
\left(
\bar{\textbf{A}}
\textbf{z}
\right)_{x_1} =
\left[
\begin{array}{c}
0\\
\displaystyle (2\mu + \lam)\varepsilon \bar{v}_{1x_1x_1}\\
\displaystyle \kappa\varepsilon\bar{\theta}_{x_1x_1} + (2\mu + \lambda)\varepsilon(\bar v_1\bar{v}_{1x_1})_{x_1}
\end{array}
\right],\\[6mm]
\textbf{z}(0,x_1)=(z_1, z_2, z_3)(0, x_1) := (0, 0, 0),
\end{array}
\right.
\end{equation}
where the Jacobian matrix
\begin{equation}\label{Jac-A}
\bar{\textbf{A}} =
\left[
\begin{array}{ccc}
0 & 1 & 0\\
-\frac{\bar{m}_1^2}{\bar{\rho}^2} + \bar{p}_{\bar{\rho}} & \frac{2\bar{m}_1}{\bar{\rho}} + \bar{p}_{\bar{m}_1} & \bar{p}_{\bar{\mathcal{E}}}\\
-\frac{\bar{m}_1\bar{\mathcal{E}}}{\bar{\rho}^2} + \frac{\bar{m}_1}{\bar{\rho}}\bar{p}_{\bar{\rho}} - \frac{\bar{p}\bar{m}_1}{\bar{\rho}^2} & \frac{\bar{\mathcal{E}}}{\bar{\rho}} + \frac{\bar{m}_1}{\bar{\rho}}\bar{p}_{\bar{m}_1} + \frac{\bar{p}}{\bar{\rho}} & \frac{\bar{m}_1}{\bar{\rho}} + \frac{\bar{m}_1}{\bar{\rho}}\bar{p}_{\bar{\mathcal{E}}}
\end{array}
\right],
\end{equation}
and $\bar{m}_1 := \bar{\rho}\bar v_1,~ \bar{\mathcal{E}} := \bar{\rho}\bar{E} = \bar{\rho}\left(\frac{R}{\gamma - 1}\bar{\theta} + \frac12\bar v_1^2\right)$ represents the momentum and the total energy associated with the approximate rarefaction wave, respectively. Note that the hyperbolic wave $\textbf{z}(t, x_1)$ is motivated but quite different from \cite{HWY-2}
since the initial values for the hyperbolic wave constructed here all start from $t=0$ including the initial layer while in \cite{HWY-2} the initial layer effect is essentially neglected to justify the vanishing viscosity limit of one-dimensional compressible Navier-Stokes equations with the superposition of both shock wave and rarefaction wave and therefore, the data for the hyperbolic wave in \cite{HWY-2} is only starting from $t=h$ with $h>0$ or from $t=T$.

In order to solve this linear hyperbolic system \eqref{HW} in the fixed time interval $[0, T]$, we first diagonalize the above hyperbolic system \eqref{HW}. Direct calculations show that the Jacobian matrix $\bar{\textbf{A}}$ in \eqref{Jac-A} has three distinct eigenvalues 
$$\bar{\lam}_j = \bar{\lam}_j(\bar{\rho}, \bar v_1, S_\pm) = \bar{v}_1 + (-1)^{\frac{j+1}{2}}\sqrt{\bar{p}_{\bar{\rho}}(\bar{\rho}, S_\pm)} \ (j=1,3), 
\quad \bar{\lam}_2 = \bar{\lam}_2(\bar{\rho}, \bar v_1, S_\pm) = \bar v_1$$
 with the corresponding left and right eigenvectors $\bar{l}_j = \bar{l}_j(\bar{\rho}, \bar v_1, S_\pm)$,  
$\bar{r}_j = \bar{r}_j(\bar{\rho}, \bar v_1, S_\pm)\  (j = 1, 2, 3)$ satisfying
$
\bar{\textbf{L}}\bar{\textbf{A}}\bar{\textbf{R}} = \text{diag}(\bar{\lam}_1, \bar{\lam}_2, \bar{\lambda}_3) := \bar{\mathbf{\Lambda}}, \  \bar{\textbf{L}}\bar{\textbf{R}} = \mathbf{I}, 
$
with $\bar{\textbf{L}} := (\bar{l}_1, \bar{l}_2, \bar{l}_3)^\top, \bar{\textbf{R}} := (\bar{r}_1, \bar{r}_2, \bar{r}_3)$,  and $\mathbf{I}$ as the $3 \times 3$ identity matrix. It should be emphasized that the second characteristic field of $\bar{\textbf{A}}$ is linearly degenerate and the other two characteristic fields are genuinely nonlinear. Denote $\textbf{Z}:=(Z_1, Z_2, Z_3)^\top $ and set
$
\textbf{Z} := \bar{\textbf{L}}\textbf{z},
$
then $\textbf{z}= \bar{\textbf{R}}\textbf{Z}$, 
and $\textbf{Z}$ satisfies the diagonalized system
\begin{equation} \label{DHW}
\textbf{Z}_t +
\left(
\bar{\mathbf{\Lambda}}
\textbf{Z}
\right)_{x_1} = \bar{\textbf{L}}
\left[
\begin{array}{c}
0\\
\displaystyle (2\mu + \lam)\varepsilon \bar{v}_{1x_1x_1}\\
\displaystyle \kappa\varepsilon\bar{\theta}_{x_1x_1} + (2\mu + \lambda)\varepsilon(\bar v_1\bar{v}_{1x_1})_{x_1}
\end{array}
\right] + (\bar{\textbf{L}}_t\bar{\textbf{R}}+\bar{\textbf{L}}_{x_1}\bar{\textbf{A}}\bar{\textbf{R}})
\textbf{Z}, 
\end{equation}
with the initial data $\textbf{Z}(0, x_1) := (0, 0, 0)^\top$. Along the 3-rarefaction wave curve, it holds that the 3-Riemann invariant is constant, that is,
\begin{equation} \label{sc}
\bar{\textbf{L}}_t = -\bar{\lam}_3 \bar{\textbf{L}}_{x_1}.
\end{equation}
Note that the structure condition \eqref{sc} is crucially used to solve the linear hyperbolic system \eqref{DHW} in the finite time interval $[0,T]$. Otherwise, it does not seem  obvious how  to solve this strongly coupled hyperbolic system \eqref{DHW} by the classical characteristic method.
Applying the structure relation \eqref{sc} to $\eqref{DHW}$, the diagonalized hyperbolic system $\eqref{DHW}$ can be written as
\begin{equation} \label{DHWS}
\begin{array}{ll}
\textbf{Z}_t +
\left(
\bar{\mathbf{\Lambda}}
\textbf{Z}
\right)_{x_1} = \bar{\textbf{L}}
\left[
\begin{array}{c}
0\\
\displaystyle (2\mu + \lam)\varepsilon \bar{v}_{1x_1x_1}\\
\displaystyle \kappa\varepsilon\bar{\theta}_{x_1x_1} + (2\mu + \lambda)\varepsilon(\bar v_1\bar{v}_{1x_1})_{x_1}
\end{array}
\right] \\[7mm]
\qquad\qquad\qquad\quad +
\left[
\begin{array}{c}
\displaystyle  \bar{l}_{1x_1}\cdot\bar{r}_{1} \\
\displaystyle \bar{l}_{2x_1}\cdot\bar{r}_{1}\\
\displaystyle  \bar{l}_{3x_1}\cdot\bar{r}_{1}
\end{array}
\right] (\bar{\lambda}_1 - \bar{\lambda}_3)Z_1
+
\left[
\begin{array}{c}
\displaystyle  \bar{l}_{1x_1}\cdot\bar{r}_{2}\\
\displaystyle  \bar{l}_{2x_1}\cdot\bar{r}_{2}\\
\displaystyle  \bar{l}_{3x_1}\cdot\bar{r}_{2}
\end{array}
\right](\bar{\lambda}_2 - \bar{\lambda}_3)Z_2,
\end{array}
\end{equation}
where the equations of $Z_1, Z_2$ are decoupled from  $Z_3$ due to   \eqref{sc},
which is quite important to solve the linear hyperbolic system on the bounded interval $[0,T]$. In fact, we can first solve the hyperbolic equations of $Z_1, Z_2$ and then the equation of $Z_3$ by the classical characteristic method due to this decoupling. Moreover, we have the following key estimates for the hyperbolic wave $\textbf{z}$ or $\textbf{Z}$:
\begin{lemma} \label{lemma2.3}
	There exists a positive constant $C_T$ independent of $\delta$ and $\varepsilon$, such that
	\[
	\left\|\frac{\partial^k}{\partial x_1^k}(\textbf{Z},\textbf{z})(t, \cdot)\right\|_{L^2(\mathbb{R})}^2 \leq C_T \left(\frac{\varepsilon}{\delta^{k + 1}}\right)^2,  \quad k=0,1,2,3.
	\]
Moreover,  one has
	$$
	\sup_{t\in[0,T]}\left\|\frac{\partial^k}{\partial x_1^k}(\textbf{Z},\textbf{z})(t, \cdot)\right\|_{L^\infty(\mathbb{R})}
	=O(1)\left(\frac{\varepsilon}{\delta^{\frac 32+k}}\right),\quad k=0,1,2.
	$$	
\end{lemma}
\begin{proof} This lemma can be proved by using weighted energy estimates and the subtle structure of the diagonalized linear hyperbolic system \eqref{DHWS}. The estimates for $Z_1$ and $Z_3$ in the genuinely  nonlinear fields are similar to our previous work \cite{LWW2} for the isentropic case, thus we need to care more about the estimates of $Z_2$ in the second linearly degenerate characteristic field.

As in \cite{LWW2}, multiplying the equation $\eqref{DHWS}_3$ by $Z_3$ and integrating the resulting equation over $[0, t]\times\mathbb{R}$ with $t \in (0, T)$, we can infer
	\begin{align}
	\begin{aligned} \label{1}
	\int_{\bbr} Z_3^2 (t, x_1) dx_1 + \int_{0}^{t}\int_{\bbr} \bar{v}_{1x_1} Z_3^2 dx_1dt
	\leq C_T \left(\frac{\varepsilon}{\delta}\right)^2 + C_T \int_{0}^{t}\int_{\bbr} \bar{v}_{1x_1}(Z_1^2 + Z_2^2) dx_1dt.
	\end{aligned}
	\end{align}
Similarly, multiplying the equation $\eqref{DHWS}_1$ by $\bar{\rho}^\Gamma Z_1$ with $\Gamma$ being a large positive constant to be determined, and then integrating the resulting equation over $[0, t]\times\mathbb{R}$ with $t \in (0, T)$, we can obtain
	\begin{align}
	\begin{aligned} \label{2}
	&\int_{\bbr} \bar{\rho}^\Gamma Z_1^2(t, x_1) dx_1 + \int_{0}^{t}\int_{\bbr} \Gamma\bar{\rho}^\Gamma\bar{v}_{1x_1}Z_1^2 dx_1dt\\
	&\leq C_T \left(\frac{\varepsilon}{\delta}\right)^2 + C_T \int_{0}^{t}\int_{\bbr} \bar{\rho}^\Gamma\bar{v}_{1x_1}(Z_1^2 + Z_2^2) dx_1dt.
	\end{aligned}
	\end{align}
	Now we give some details for the estimation of $Z_2$. Multiplying the equation $\eqref{DHWS}_2$ by $\bar{\rho}^\Gamma Z_2$ 
	and then integrating the resulting equation over $[0, t]\times\mathbb{R}$ with $t \in (0, T)$, we can derive that
	\begin{align*}
	&\int_{\bbr} \bar{\rho}^\Gamma\frac{Z_2^2}{2}(t, x_1) dx_1 + \int_{0}^{t}\int_{\bbr} \Gamma\bar{\rho}^\Gamma\bar{v}_{1x_1} \frac{Z_2^2}{2} + \bar{\rho}^\Gamma\bar{\lambda}_{2x_1} \frac{Z_2^2}{2} dx_1dt \\
	&= \int_{0}^{t}\int_{\bbr} \Big[(2\mu + \lam)\varepsilon\bar{\rho}^\Gamma \bar{l}_{22}\bar{v}_{1x_1x_1}Z_2 + \bar{\rho}^\Gamma\bar{l}_{23}(\kappa\varepsilon\bar{\theta}_{x_1x_1} + (2\mu + \lambda)\varepsilon(\bar{v}_1\bar{v}_{1x_1})_{x_1})Z_2 \\
	&\quad + (\bar{l}_{2x_1}\cdot\bar{r}_{1}(\bar{\lambda}_1 - \bar{\lambda}_3)Z_1 + \bar{l}_{2x_1}\cdot\bar{r}_{2}(\bar{\lambda}_2 - \bar{\lambda}_3)Z_2)\bar{\rho}^\Gamma Z_2\Big] dx_1dt \\
	&\leq C \int_{0}^{t}\int_{\bbr} \bar{\rho}^\Gamma Z_2^2 dx_1dt + C \varepsilon^2 \int_{0}^{t}\int_{\bbr} (\bar{v}_{1x_1x_1}^2 + \bar{\theta}_{x_1x_1}^2 + \bar{v}_{1x_1}^4) dx_1dt \\
	&\quad + C \int_{0}^{t}\int_{\bbr} \bar{\rho}^\Gamma\bar{v}_{1x_1}(Z_1^2 + Z_2^2) dx_1dt \\
	&\leq C \int_{0}^{t}\int_{\bbr} \bar{\rho}^\Gamma Z_2^2 dx_1dt + C \left(\frac{\varepsilon}{\delta}\right)^2 + C \int_{0}^{t}\int_{\bbr} \bar{\rho}^\Gamma\bar{v}_{1x_1}(Z_1^2 + Z_2^2) dx_1dt.
	\end{align*}
	Using Gronwall's inequality gives
	\begin{align}
	\begin{aligned} \label{3}
	&\int_{\bbr} \bar{\rho}^\Gamma Z_2^2(t, x_1) dx_1 + \int_{0}^{t}\int_{\bbr}\Gamma\bar{\rho}^\Gamma\bar{v}_{1x_1}Z_2^2 dx_1dt\\
	&\leq C_T \left(\frac{\varepsilon}{\delta}\right)^2 + C_T \int_{0}^{t}\int_{\bbr} \bar{\rho}^\Gamma\bar{v}_{1x_1}(Z_1^2 + Z_2^2) dx_1dt.
	\end{aligned}
	\end{align}
Combining \eqref{1}, \eqref{2} and \eqref{3} together and choosing $\Gamma$ large enough, we can obtain
	\begin{align*}
	\int_{\bbr} |\textbf{Z}|^2(t, x_1) dx_1 + \int_{0}^{t}\int_{\bbr} \bar{v}_{1x_1}|\textbf{Z}|^2 dx_1dt
	\leq C_T \left(\frac{\varepsilon}{\delta}\right)^2.
	\end{align*}
Thus we have proved Lemma \ref{lemma2.3} for the case $k = 0$. The $k$-th order derivative estimates in Lemma \ref{lemma2.3} for $k = 1, 2, 3$ can be justified similarly after differentiating the system \eqref{DHWS} $k$ times with respect to $x_1$, and the details are  omitted. 
\end{proof}

\subsection{Construction of solution profile}
Define the approximate solution profile $(\tilde{\rho}, \tilde{v}_1, \tilde{\theta})$ of the full compressible Navier-Stokes-Fourier equations \eqref{NS} as
\begin{equation} \label{AW0}
\tilde{\rho} = \bar{\rho} + z_1, \quad \tilde{m}_1 = \bar{m}_1 + z_2 := \tilde{\rho}\tilde{v}_1, \quad \tilde{\mathcal{E}} = \bar{\mathcal{E}} + z_3 := \tilde{\rho}\tilde{E} = \tilde{\rho}\left(\frac{R}{\gamma - 1}\tilde{\theta} + \frac12\tilde{v}_1^2\right),
\end{equation}
where $(\bar\rho, \bar m:=\bar\rho\bar v_1, \bar{\mathcal{E}}:=\bar{\rho}(\frac{R}{\gamma - 1}\bar{\theta} + \frac12\bar{v}_1^2))$ is the planar rarefaction wave in \eqref{AR} and $(z_1, z_2, z_3)$ is the hyperbolic wave in \eqref{HW}.
Correspondingly, it holds that
\begin{equation}\label{ue}
\begin{array}{ll}
\di\tilde{v}_1=\bar v_1+\frac{1}{\tilde\rho}(-\bar v_1z_1+z_2)=\bar v_1+O(1)|(z_1,z_2)|,\\[3mm]
\di\tilde{\theta}\ =\bar \theta+\frac{\gamma - 1}{R\tilde\rho}\left[-\frac{R}{\gamma - 1}\bar \theta z_1 + z_3-\frac{1}{2}\bar v_1^2z_1 - \bar v_1(-\bar v_1z_1+z_2)\right]-\frac{\gamma - 1}{2R\tilde\rho^2}(-\bar v_1z_1+z_2)^2\\[3mm]
\quad=\bar \theta+O(1)|(z_1,z_2,z_3)|.
\end{array}
\end{equation}
Then the approximate wave profile $(\tilde{\rho}, \tilde{v}_1, \tilde{\theta})$ satisfies the system
\begin{equation} \label{AW}
\begin{cases}
\displaystyle \tilde{\rho}_t + (\tilde{\rho}\tilde{v}_1)_{x_1} = 0, \\
\displaystyle (\tilde{\rho} \tilde{v}_1)_t + (\tilde{\rho} \tilde{v}_1^2 + R\tilde{\rho}\tilde{\theta})_{x_1} = (2\mu + \lam)\varepsilon \bar{v}_{1x_1x_1} +Q_1, \\[3mm]
\displaystyle \frac{R}{\gamma - 1}\left[(\tilde{\rho}\tilde{\theta})_t + (\tilde{\rho}\tilde{v}_1\tilde{\theta})_{x_1}\right] + R\tilde{\rho}\tilde{\theta}\tilde{v}_{1x_1} = \kappa\varepsilon\bar{\theta}_{x_1x_1} + (2\mu + \lambda)\varepsilon\bar{v}_{1x_1}^2 +Q_2,
\end{cases}
\end{equation}
with the initial data
\begin{equation} \label{AWI}
(\tilde{\rho}, \tilde{v}_1, \tilde{\theta})(0, x_1) = (\bar{\rho}_0, \bar{v}_{10}, \bar{\theta}_0)(x_1),
\end{equation}
and the error terms
\begin{equation}\label{Q1}
\begin{array}{ll}
\displaystyle  Q_1:=\left(\frac{\tilde m_1^2}{\tilde\rho}-\frac{\bar m_1^2}{\bar\rho}+\frac{\bar m_1^2}{\bar\rho^2}z_1-\frac{2\bar m_1}{\bar\rho}z_2\right)_{x_1} + (\tilde p-\bar p-\bar{p}_{\bar{\rho}} z_1 - \bar{p}_{\bar{m}_1}z_2 - \bar{p}_{\bar{\mathcal{E}}} z_3)_{x_1} \\[4mm]
\displaystyle \ \ \quad=\left[\frac{3-\gamma}{2\tilde{\rho}}(\bar{v}_1z_1 - z_2)^2\right]_{x_1}=O(1)\left[|\bar v_{1x_1}||(z_1,z_2)|^2+|(z_1,z_2)||(z_{1x_1},z_{2x_1})|\right],
\end{array}
\end{equation}
and
\begin{equation}\label{Q2}
\begin{array}{ll}
\di Q_2:=\left(\frac{\tilde m_1\tilde{\mathcal{E}}}{\tilde\rho}-\frac{\bar m_1\bar{\mathcal{E}}}{\bar\rho}+\frac{\bar m_1\bar{\mathcal{E}}}{\bar\rho^2}z_1-\frac{\bar{\mathcal{E}}}{\bar\rho}z_2-\frac{\bar m_1}{\bar\rho}z_3\right)_{x_1}\\[4mm]
\ \ \qquad\di +\left(\tilde p\frac{\tilde m_1}{\tilde\rho}-\bar p\frac{\bar m_1}{\bar\rho}-\bar{p}_{\bar{\rho}}\frac{\bar m_1}{\bar\rho} z_1 +\bar p \frac{\bar m_1}{\bar\rho^2}z_1- \bar{p}_{\bar{m}_1}\frac{\bar m_1}{\bar\rho}z_2-\frac{\bar p}{\bar\rho}z_2 - \bar{p}_{\bar{\mathcal{E}}}\frac{\bar m_1}{\bar\rho} z_3\right)_{x_1} - \tilde{v}_1Q_1\\[4mm]
\ \ \quad= \left[\frac{-\bar{v}_1z_1 + z_2}{\tilde{\rho}}\left(\gamma z_3 - (\gamma - 1)\bar{v}_1z_2 - \frac{R\gamma}{\gamma - 1}\bar{\theta}z_1 + \frac{\gamma - 2}{2}\bar v_1^2z_1\right)  - (\gamma - 1)\tilde{v}_1\frac{(-\bar{v}_1z_1 + z_2)^2}{2\tilde{\rho}}\right]_{x_1}\\ [3mm]
 \ \ \qquad\di - (2\mu + \lambda)\varepsilon\bar{v}_{1x_1x_1}\frac{-\bar{v}_1z_1 + z_2}{\tilde{\rho}}
- \tilde{v}_1\left(\frac{3-\gamma}{2\tilde{\rho}}(-\bar{v}_1z_1 + z_2)^2\right)_{x_1}\\[4mm]
\ \ \quad =O(1)\left[|\bar v_{1x_1}||(z_1,z_2,z_3)|^2+|(z_1,z_2,z_3)||(z_{1x_1},z_{2x_1},z_{3x_1})|+\varepsilon|\bar{v}_{1x_1x_1}||(z_1,z_2)|\right].
\end{array}
\end{equation}
\begin{remark}
With the approximate solution profile $(\tilde{\rho}, \tilde{v}_1, \tilde{\theta})$ of the full compressible Navier-Stokes-Fourier equations \eqref{NS} combined by both the approximate rarefaction wave and hyperbolic wave, the error terms in \eqref{Q1} and \eqref{Q2} are good enough to obtain the desired uniform estimates with respect to the dissipation coefficients for the justification of the vanishing dissipation limit.
\end{remark}

Finally, by the estimate of the hyperbolic wave in Lemma \ref{lemma2.3} and noting that $\delta = \varepsilon^\frac 16|\ln\varepsilon|$ in the present paper, we have
\[
|z_i| \leq C_T \frac{\varepsilon}{\delta^{3/2}} = C_T \varepsilon^{\frac{3}{4}}|\ln\varepsilon|^{-\frac32} \leq \frac{1}{4}\rho_-, \qquad i=1,2,3,
\]
provided that $\varepsilon \ll 1$. Then we have from \eqref{ue} that
\begin{equation}\label{t-rho}
0 < \frac{3}{4}\rho_- = \rho_- - \frac{1}{4}\rho_- \leq \tilde{\rho} = \bar{\rho} + z_1 \leq \rho_+ + \frac{1}{4}\rho_-, \quad |\tilde{v}_1| \leq C_1,
\end{equation}
since $0 < \rho_- \leq \bar{\rho} \leq \rho_+ $ and $ |\bar{v}_1| \leq C$.
Similarly, since $0 < \theta_- \leq \bar{\theta} \leq \theta_+$, we have from Lemma \ref{lemma2.3} and \eqref{ue} that
\begin{equation}\label{t-theta}
0 < \frac{3}{4}\theta_- \leq \tilde{\theta} \leq \theta_+ + \frac{1}{4}\theta_-,
\end{equation}
provided that $\varepsilon \ll 1$. The uniform boundedness in \eqref{t-rho} and \eqref{t-theta} for the solution profile  $(\tilde{\rho}, \tilde{v}_1, \tilde{\theta})$, in particular,  the lower bound of the approximate density and temperature are quite important for the existence of   solution to the 3D Navier-Stokes-Fourier system \eqref{NS} around this profile.

%
%

\section{Proof of Theorem \ref{theorem1}}
\setcounter{equation}{0}

As the first step of the proof of Theorem \ref{theorem1} we
reformulate the problem as the perturbation of the solution $(\rho^\varepsilon, \bv^\varepsilon, \theta^\varepsilon)$ to the 3D Navier-Stokes-Fourier system \eqref{NS} around the approximate wave profile $(\tilde{\rho}, \tilde{\bv}:= (\tilde{v}_1, 0, 0)^\top, \tilde{\theta})(t, x_1)$ defined in \eqref{AW0} and \eqref{AW}.  In the setting of classical solutions, we can rewrite the 3D full compressible Navier-Stokes-Fourier system \eqref{NS} as
\begin{equation}  \label{RNS}
\begin{cases}
\displaystyle \rho_t + \div(\rho \bv) = 0,    \\ 
\displaystyle (\rho \bv)_t + \div(\rho \bv \otimes \bv) + R\nabla(\rho\theta) = \mu\varepsilon\triangle\bv + (\mu + \lambda)\varepsilon\nabla \div\bv, \\
\displaystyle \frac{R}{\gamma-1}\left[(\rho\theta)_t + \div(\rho\theta\bv)\right] + R\rho\theta\div \bv = \kappa\varepsilon\triangle\theta + \frac{\mu\varepsilon}{2}|\nabla \bv + (\nabla\bv)^\top|^2 + \lambda\varepsilon(\div \bv)^2.
\end{cases}
\end{equation} 
Denote the perturbation of $(\rho^\varepsilon, \bv^\varepsilon, \theta^\varepsilon)$ around  
$(\tilde{\rho}, \tilde{\bv}, \tilde{\theta})(t, x_1)$ by
\begin{equation} \label{per}
\begin{array}{ll}
\di \varphi(t, x) := \rho^\varepsilon(t, x) - \tilde{\rho}(t, x_1), \\[3mm]
\di \Psi(t, x) = (\psi_1, \psi_2, \psi_3)^\top(t, x) := \bv^\varepsilon(t, x) - \tilde{\bv}(t, x)
   =(v_1^\varepsilon, v_2^\varepsilon, v_3^\varepsilon)^\top(t, x) - (\tilde{v}_1, 0, 0)^\top(t, x_1), \\[3mm]
\di \xi(t, x) := \theta^\varepsilon(t, x) - \tilde{\theta}(t, x_1), 
\end{array}
\end{equation}
then the solution $(\rho^\varepsilon, \bv^\varepsilon, \theta^\varepsilon)$ to the full compressible Navier-Stokes-Fourier system \eqref{NS} or \eqref{RNS} has the following initial data:
\begin{equation} \label{PNSI}
(\rho^\varepsilon, \bv^\varepsilon, \theta^\varepsilon)(0, x) := (\bar{\rho}_0, \bar{\bv}_0, \bar{\theta}_0)(x_1) + (\varphi_0, \Psi_0, \xi_0)(x),
\end{equation}
where $\Psi_0 := (\psi_{10}, \psi_{20}, \psi_{30})^\top$. 
For the sake of simplicity, the superscript $\varepsilon$ will be dropped in $(\rho^\varepsilon, \bv^\varepsilon, \theta^\varepsilon)$ when there is no confusion.
From \eqref{RNS} and \eqref{AW}, the following system for   $(\varphi, \Psi, \xi)(t,x)$ can be derived:
\begin{equation} \label{REF}
    \begin{cases}
    \displaystyle \varphi_t + \bv\cdot\nabla\varphi + \rho \div\Psi + \tilde{\rho}_{x_1}\psi_1 + \tilde{v}_{1x_1}\varphi = 0, \\
    \displaystyle \rho \Psi_t + \rho \bv\cdot\nabla\Psi + R\theta\nabla\varphi + R\rho\nabla\xi + (\rho \tilde{v}_{1x_1} \psi_1, 0, 0)^\top + \left(R\tilde{\rho}_{x_1}\left(\theta - \frac{\rho}{\tilde{\rho}} \tilde{\theta}\right) , 0, 0\right)^\top \\[4mm]
    = \mu\varepsilon\triangle\Psi + (\mu + \lam)\varepsilon\nabla \div\Psi + \left((2\mu + \lam)\varepsilon\left(\frac{-\bar{v}_1z_1 + z_2}{\tilde{\rho}}\right)_{x_1x_1}, 0, 0\right)^\top \\[4mm]
    \quad - \left((2\mu + \lam)\varepsilon\frac{\bar{v}_{1x_1x_1}}{\tilde{\rho}}\varphi, 0, 0\right)^\top - \left(Q_1\frac{\rho}{\tilde{\rho}}, 0, 0\right)^\top, \\[4mm]
    \displaystyle \frac{R}{\gamma - 1}(\rho \xi_t + \rho \bv\cdot\nabla\xi) + R\rho\theta \div\Psi + \frac{R}{\gamma - 1}\rho \tilde{\theta}_{x_1} \psi_1 + R\rho\tilde{v}_{1x_1}\xi \\[4mm]
    = \kappa\varepsilon\triangle\xi + \frac{\mu\varepsilon}{2}|\nabla\Psi + (\nabla\Psi)^\top|^2 + \lambda\varepsilon(\div\Psi)^2 + 2\tilde{v}_{1x_1}(2\mu\varepsilon\psi_{1x_1} + \lambda\varepsilon \div\Psi) \\[3mm]
    \quad + F_1 + F_2  - \frac{\rho}{\tilde{\rho}}Q_2,
    \end{cases}
\end{equation}
supplemented with the initial data
\begin{equation} \label{REFI}
  (\varphi, \Psi, \xi)(0, x) = (\varphi_0, \Psi_0, \xi_0)(x),
\end{equation}
where
\begin{align}\label{F12}
&F_1 :=-\kappa\varepsilon(\tilde \theta_{x_1x_1}-\bar \theta_{x_1x_1})-(2\mu + \lambda)\varepsilon(\tilde{v}_{1x_1}^2-\bar{v}_{1x_1}^2)\nonumber\\
&\quad= \frac{\gamma - 1}{R}\kappa\varepsilon\left\{\frac{1}{\tilde{\rho}}\left[\left(\frac12\bar v_1^2 - \frac{R}{\gamma - 1}\bar{\theta}\right)z_1 - \bar{v}_1z_2 + z_3\right]\right\}_{x_1x_1} - \frac{\gamma - 1}{2R}\kappa\varepsilon\left[\left(\frac{-\bar{v}_1z_1 + z_2}{\tilde{\rho}}\right)^2\right]_{x_1x_1} \nonumber\\[3mm]
&\qquad + 2(2\mu + \lambda)\varepsilon\bar{v}_{1x_1}\left(\frac{-\bar{v}_1z_1 + z_2}{\tilde{\rho}}\right)_{x_1} 
+ (2\mu + \lambda)\varepsilon\left[\left(\frac{-\bar{v}_1z_1 + z_2}{\tilde{\rho}}\right)_{x_1}\right]^2\\[3mm]
&\quad=O(\varepsilon)\big[|(z_{1x_1x_1},z_{2x_1x_1}, z_{3x_1x_1})|+|(z_{1x_1},z_{2x_1})|^2+|\bar v_{1x_1}(z_{1x_1},z_{2x_1},z_{3x_1})|\nonumber\\[2mm]
&\qquad\qquad\quad  \di +|\bar v_{1x_1}(z_{1},z_{2},z_{3})|^2\big] ,\nonumber\\[3mm]
&F_2 := - \kappa\varepsilon\frac{\bar{\theta}_{x_1x_1}}{\tilde{\rho}}\varphi - (2\mu + \lambda)\varepsilon\frac{\bar{v}_{1x_1}^2}{\tilde{\rho}}\varphi.
\end{align}
We choose the initial perturbation $(\varphi_0, \Psi_0, \xi_0)(x)$ such that
\begin{equation}\label{ip}
\|(\nabla^i\varphi_0, \nabla^i\Psi_0, \nabla^i\xi_0)\|^2 = O\left(\frac{\varepsilon^{4-i}}{\delta^{7+i}}\right) = O\left(\varepsilon^{\frac{17-7i}{6}}|\ln\varepsilon|^{-7-i}\right), \quad i = 0, 1, 2.
\end{equation}
We note that 
the solution to \eqref{REF}, \eqref{REFI} and \eqref{ip} will be constructed in the functional space $\Pi(0, T)$ defined by
\begin{align*}
\Pi(0, t_1) = \left\{ (\varphi, \Psi, \xi)| \, (\varphi, \Psi, \xi) \in C^0([0, t_1]; H^2(\Omega)), (\nabla \Psi, \nabla\xi) \in L^2(0, t_1; H^2(\Omega))\right\}, \forall t_1 \in(0, T].
\end{align*}
To carry out the analysis,  we first set the {\it a priori} assumptions:
\begin{equation} \label{PA}
\begin{array}{cc}
\di \sup_{t \in [0, t_1(\varepsilon)]}\| (\varphi, \Psi, \xi)(t) \|_{L^\infty} \ll 1, \\[4mm]
\di \sup_{t \in [0, t_1(\varepsilon)]}\| (\nabla\varphi, \nabla\Psi, \nabla\xi)(t) \| \leq \varepsilon^{a_1}|\ln\varepsilon|^{-1}, \\[4mm]
\di \sup_{t \in [0, t_1(\varepsilon)]}\| (\nabla^2\varphi, \nabla^2\Psi, \nabla^2\xi)(t) \| \leq \varepsilon^{a_2}|\ln\varepsilon|^{-1},
\end{array}
\end{equation}
where $[0, t_1(\varepsilon)]$ is the time-interval of existence of solutions and   may depend on $\varepsilon$,  and both $a_1$ and $a_2$ are positive constants  to be determined. 
Note that the above {\it a priori} assumptions \eqref{PA} with different rates are very important to control the nonlinear terms for justifying the vanishing dissipation limit in the three-dimensional case, which dorminate the decay rate of the vanishing dissipation limit.
Under the {\it a priori} assumptions \eqref{PA} and the boundedness in \eqref{t-rho} and \eqref{t-theta}, we can get
\begin{equation} \label{db}
	0 < \frac{1}{2} \rho_- \leq \rho = \varphi + \tilde{\rho} \leq \rho_+ + \frac{1}{2}\rho_-, \quad 0 < \frac{1}{2} \theta_- \leq \theta = \xi + \tilde{\theta} \leq \theta_+ + \frac{1}{2}\theta_-, \quad |\bv| \leq C,
\end{equation}
because we can take $\| (\varphi, \Psi, \xi) \|_{L^\infty} \leq \min \{\frac{1}{4}\rho_-, \frac14\theta_-\}$.
The above uniform bounds of the density $\rho$ imply that   the two equations $\eqref{RNS}_2$ and $\eqref{RNS}_3$   are strictly parabolic, which ensures  the   existence of classical solution to the hyperbolic-parabolic coupled perturbation system \eqref{REF}. 
To prove our main result Theorem \ref{theorem1}, it is sufficient to prove the following result.

\begin{theorem} \label{proposition3.1}
{\rm (Global existence and uniform estimates)}
There exists a positive constant $\varepsilon_1<1$ 
such that if $0 < \varepsilon \leq \varepsilon_1$ and under the choice $\delta = \varepsilon^{\frac{1}{6}}|\ln\varepsilon|$ for the approximate rarefaction wave in \eqref{ABE},  the perturbation problem \eqref{REF}-\eqref{REFI}  has a unique global solution $(\varphi, \Psi, \xi) \in \Pi(0, T)$ satisfying
  \begin{align}
  \begin{aligned} \label{PRO3.1}
    &\sup_{0 \leq t \leq T} \| (\varphi, \Psi, \xi)(t) \|^2 + \int_{0}^{T} \Big[\|\bar{v}_{1x_1}^{1/2} (\varphi, \psi_1, \xi)\|^2 + \varepsilon\|(\nabla\Psi, \nabla\xi)\|^2 \Big]dt \\
    &\leq C_T \varepsilon^{\frac{17}{6}}|\ln\varepsilon|^{-7} + C \| (\varphi_0, \Psi_0, \xi_0) \|^2, \\
    &\sup_{0 \leq t \leq T} \| (\nabla^i\varphi, \nabla^i\Psi, \nabla^i\xi)(t) \|^2 + \int_{0}^{T} \Big[\|\bar{v}_{1x_1}^{1/2} \nabla^i\varphi\|^2 + \varepsilon\|(\nabla^{1+i}\Psi, \nabla^{1+i}\xi)\|^2 \Big]dt \\
    &\leq C_T \varepsilon^{\frac{17-7i}{6}}|\ln\varepsilon|^{-7-i} + C \| (\nabla^i\varphi_0, \nabla^i\Psi_0, \nabla^i\xi_0) \|^2,\quad i=1,2,
  \end{aligned}
  \end{align}
for some  constant $C_T>0$ that is independent of $\varepsilon$ and $\delta$, but may depend on $T$.
\end{theorem}

Once Theorem \ref{proposition3.1} is proved, by the initial perturbation satisfying \eqref{ip} and the Sobolev inequality in dimension three (cf. \cite{WW}), we have from \eqref{PRO3.1} that
\begin{equation}\label{dr1}
\begin{array}{ll}
  \di \sup_{0 \leq t \leq T} \| (\varphi, \Psi, \xi)(t, x) \|_{L^\infty(\Omega)} \\
  \di \leq C \sup_{0 \leq t \leq T} \Big[\| (\varphi, \Psi, \xi) \|^{\frac12} \| (\nabla\varphi, \nabla\Psi, \nabla\xi) \|^{\frac12} + \| (\nabla\varphi, \nabla\Psi, \nabla\xi) \|^{\frac12}\| (\nabla^2\varphi, \nabla^2\Psi, \nabla^2\xi) \|^{\frac12}\Big] \\[4mm]
  \di \leq C_T \varepsilon^{\frac{9}{8}}|\ln\varepsilon|^{-\frac{15}4} + C_T \varepsilon^{\frac{13}{24}}|\ln\varepsilon|^{-\frac{17}4}=O(1) \varepsilon^{\frac{13}{24}}|\ln\varepsilon|^{-\frac{17}4}, 
  \end{array}
\end{equation}
and then
\begin{align}\label{dr}
  &\left\|(\rho, \bv, \theta)(t, x) - (\rho^r, \bv^r, \theta^r)\left(\frac{x_1}{t}\right)\right\|_{L^\infty(\Omega)}\nonumber \\
  &\leq \| (\varphi, \Psi, \xi)(t, x) \|_{L^\infty(\Omega)} + C \| (z_1, z_2, z_3)(t, x_1) \|_{L^\infty(\mathbb{R})}\nonumber \\
  &\quad + \left\|(\bar{\rho}, \bar v_1, \bar{\theta})(t, x_1) - (\rho^r, v_1^r, \theta^r)\left(\frac{x_1}{t}\right)\right\|_{L^\infty(\mathbb{R})} \\
  &\leq C_T \varepsilon^{\frac{13}{24}}|\ln\varepsilon|^{-\frac{17}{4}} + C_T \frac{\varepsilon}{\delta^{3/2}} + C \delta t^{-1} [\ln(1 + t) + |\ln \delta|]\nonumber\\
  &=O(1)\varepsilon^{\frac16}|\ln\varepsilon|^2,\nonumber
\end{align}
which completes the proof of Theorem \ref{theorem1}.

We remark that even though the decay rate with respect to $\varepsilon$ in \eqref{dr1} for the perturbation around the approximate rarefaction wave and the hyperbolic wave is higher than the final vanishing dissipation rate in \eqref{dr}, all the {\it a priori} estimates are performed and the nonlinear terms are controlled under the choice of the parameter $\delta=\varepsilon^{\frac{1}{6}}|\ln\varepsilon|$ in the approximate rarefaction wave \eqref{ABE}. In other words, here the decay rate with respect to the dissipation parameters in the 3D case is determined by the nonlinear terms in the original variables $x$ and $t$, which is quite different from the two-dimensional case for the scaled variables in \cite{LWW2} where the final decay rate with respect to the viscosities is dorminated by the error terms due to the inviscid rarefaction wave profile and the hyperbolic wave. Moreover, the vanishing dissipation rate of the self-similar rarefaction wave fan in \eqref{dr} for the 3D full compressible Navier-Stokes-Fourier equations \eqref{NS} seems optimal at least in our framework.

In fact, to control the nonlinear term in \eqref{25} of Section 4 below, we need
\begin{align}\label{a-2}
	-\frac 13 + \frac 43 a_2 \geq 0,
\end{align}
which is equivalent to $a_2\geq\frac14$.
Also,  in order to close the {\it a priori} assumption \eqref{PA}, by the estimates of the nonlinear terms in \eqref{dr}, \eqref{dr4}, \eqref{dr1} and \eqref{dr3} and noting that $\delta=\varepsilon^b |\ln \varepsilon|$ from \eqref{del}, we also require
\begin{align*}
	2 - 9b \geq 2a_2,
\end{align*}
that is,
\begin{align}\label{b}
	b\leq\frac{2-2a_2}9.
\end{align}
Therefore, by \eqref{a-2} and \eqref{b}, it is optimal to take $a_2 = \frac 14$ in \eqref{PA} and then $b = \frac 16$ in \eqref{del} for the approximate rarefaction wave \eqref{ABE}, i.e. $\delta = \varepsilon^{\frac{1}{6}}|\ln\varepsilon|$ as we choose. Correspondingly, by the estimate of  the nonlinear term in \eqref{15} of Section 4 below, we need
\begin{align*}
	-\frac{3}{5} + \frac{2}{5}a_1 + \frac{6}{5}a_2 \geq 0,
\end{align*}
to close the {\it a priori} assumption \eqref{PA}, that is,
\begin{align*}
 2a_1\geq 3-6a_2=\frac 32.
\end{align*}
Thus we can take $a_1=\frac34$ in \eqref{PA}.
Then we can close our {\it a priori} assumptions \eqref{PA} and complete the proof of Theorem \ref{theorem1}.

Now it remains to prove Theorem \ref{proposition3.1}. The local existence and uniqueness of the   classical solution to the hyperbolic-parabolic coupled system \eqref{REF}-\eqref{REFI}
is well-known and hence its proof is omitted (c.f. \cite{MN, So, LWW}). Therefore, to prove Theorem \ref{proposition3.1}, it suffices to obtain the following uniform {\it a priori} estimates.

\begin{proposition}[Uniform {\it a priori} estimates] \label{proposition3.2}
Let $(\varphi, \Psi, \xi) \in \Pi(0, t_1(\varepsilon))$ be a solution to the problem \eqref{REF}-\eqref{REFI}   for some $t_1(\varepsilon) \in (0, T]$. 
There exists a   constant $\varepsilon_2>0$   independent of $\varepsilon, \delta$ and $t_1(\varepsilon
)$, such that, for  $0<\varepsilon \leq \varepsilon_2$ and under the choice $\delta = \varepsilon^{\frac{1}{6}}|\ln\varepsilon|$ and the {\it a priori} assumptions \eqref{PA},
it holds that
  \begin{align}
	\begin{aligned} \label{PRO3.2}
	  &\sup_{0 \leq t \leq t_1(\varepsilon)} \| (\varphi, \Psi, \xi)(t) \|^2 + \int_{0}^{t_1(\varepsilon)} \left[\|\bar{v}_{1x_1}^{1/2} (\varphi, \psi_1, \xi)\|^2 + \varepsilon\|(\nabla\Psi, \nabla\xi)\|^2 \right]dt \\
	  &\leq C_T \varepsilon^{\frac{17}{6}}|\ln\varepsilon|^{-7} + C \| (\varphi_0, \Psi_0, \xi_0) \|^2, \\
	  &\sup_{0 \leq t \leq t_1(\varepsilon)} \| (\nabla^i\varphi, \nabla^i\Psi, \nabla^i\xi)(t) \|^2 + \int_{0}^{t_1(\varepsilon)} \left[\|\bar{v}_{1x_1}^{1/2} \nabla^i\varphi\|^2 + \varepsilon\|(\nabla^{1+i}\Psi, \nabla^{1+i}\xi)\|^2 \right]dt \\
	  &\leq C_T \varepsilon^{\frac{17-7i}{6}}|\ln\varepsilon|^{-7-i} + C \| (\nabla^i\varphi_0, \nabla^i\Psi_0, \nabla^i\xi_0) \|^2,\quad  i=1, 2,
	\end{aligned}
  \end{align}
for some constant $C_T>0$ that is independent of $\varepsilon$ and $\delta$, but may depend on $T$.
\end{proposition}
The proof of Proposition \ref{proposition3.2} will be given in the next section.

\bigskip

%
%

\section{Proof of Uniform {\it a priori} Estimates in Proposition \ref{proposition3.2}}
\setcounter{equation}{0}

We now give the proof of Proposition \ref{proposition3.2} for the uniform $H^2$ {\it a priori} estimates. First in subsection 4.1 we  derive the lower order  $L^2-$estimates in Lemma \ref{lemma4.1}, where the periodicity of the domain $\Omega$ in $x_2$ and $x_3$ directions and the original non-scaled spatial variables are crucial. Then in subsection 4.2 and subsection 4.3,   the first-order and second-order derivative estimates are proved in Lemma \ref{lemma4.2} and Lemma \ref{lemma4.3}, respectively, thanks to  the cancelations between the flux terms,  which are quite different from the two-dimensional case in \cite{LWW2}.
We remark that the decay rate  is determined by the nonlinear flux terms in the original variables for the 3D limit here,  but  by the error terms   in the scaled variables for the 2D case of \cite{LWW2}.


\subsection{Lower order estimates}

We start with the lower order  $L^2-$relative entropy estimates.
\begin{lemma} \label{lemma4.1}
 Under the assumption of Proposition \ref{proposition3.2}, there exists a positive constant $C_T$ independent of $\varepsilon$ such that for $0 \leq t \leq t_1(\varepsilon)$,
  \begin{align}
    \begin{aligned} \label{LEM4.1}
    &\sup_{0 \leq t \leq t_1(\varepsilon)} \| (\varphi, \Psi, \xi)(t) \|^2 + \int_{0}^{t_1(\varepsilon)}\left[ \|\bar{v}_{1x_1}^{1/2} (\varphi, \psi_1, \xi) \|^2 + \varepsilon \|(\nabla \Psi, \nabla\xi)\|^2 \right] dt \\
    &\quad \leq C_T \frac{\varepsilon^4}{\delta^7} + C \| (\varphi_0, \Psi_0, \xi_0) \|^2 = C_T \varepsilon^{\frac{17}{6}}|\ln\varepsilon|^{-7} + C \| (\varphi_0, \Psi_0, \xi_0) \|^2.
    \end{aligned}
  \end{align}
\end{lemma}
\begin{proof}
	For ideal polytropic flows,  
	\[
	S = -R\ln\rho + \frac{R}{\gamma-1}\ln\theta + \frac{R}{\gamma-1}\ln\frac{R}{A}, \quad
	p = R\rho\theta = A\rho^{\gamma}\exp\Big(\frac{\gamma-1}{R}S\Big).
	\]
	Denote
	\[
	\begin{array}{l}
	\displaystyle \mathbf{X}=\Big(\rho, \rho v_1, \rho v_2, \rho v_3, \rho \Big(\frac{R}{\gamma-1}\theta + \frac{|\bv|^2}{2}\Big)\Big)^\top, \\
	\displaystyle \mathbf{Y}=\Big(\rho\bv, \rho\bv v_1 + p\mathbb{I}_1, \rho\bv v_2 + p\mathbb{I}_2, \rho\bv v_3 + p\mathbb{I}_3, \rho\bv\Big(\frac{R}{\gamma-1}\theta + \frac{|\bv|^2}{2}\Big) + p\bv\Big)^\top,
	\end{array}
	\]
	where $\mathbb{I}_1=(1,0,0)^\top, \mathbb{I}_2=(0,1,0)^\top$ and $\mathbb{I}_3=(0,0,1)^\top$.
	Then we can rewrite the system \eqref{NS} as
	\[
	\mathbf{X}_t + \div\mathbf{Y} =
	\left(
	\begin{array}{c}
	0\\
	\displaystyle \mu\varepsilon\triangle v_1 + (\mu + \lambda)\varepsilon\div\bv_{x_1} \\[2mm]
	\displaystyle \mu\varepsilon\triangle v_2 + (\mu + \lambda)\varepsilon\div\bv_{x_2} \\[2mm]
	\displaystyle \mu\varepsilon\triangle v_3 + (\mu + \lambda)\varepsilon\div\bv_{x_3} \\[2mm]
	\displaystyle \kappa\varepsilon\triangle\theta + \div(\bv\mathbb{S})
	\end{array}
	\right).
	\]
Define a relative entropy-entropy flux pair $(\eta_*, \bq_*)$ as
	\[
	\left\{
	\begin{array}{l}
	\displaystyle \eta_* = \tilde{\theta}\left\{ -\rho S + \tilde\rho\tilde{S} + \nabla_\mathbf{X}(\rho
	S)\Big|_{\mathbf{X}=\tilde{\mathbf{X}}}\cdot(\mathbf{X} - \tilde{\mathbf{X}}) \right\}, \\[4mm]
	\displaystyle q_{*j} = \tilde\theta\left\{ -\rho v_j S + \tilde\rho\tilde{v}_j\tilde{S} + \nabla_\mathbf{X}(\rho
	S)\Big|_{\mathbf{X}=\tilde{\mathbf{X}}}\cdot(\mathbf{Y}_j - \tilde{\mathbf{Y}}_j) \right\} \quad j=1,2,3.
	\end{array}
	\right.
	\]
From the following
	\[
	\displaystyle (\rho S)_{\rho} = S + \frac{|\bv|^2}{2\theta} - \frac{R\gamma}{\gamma-1}, \qquad
	\displaystyle (\rho S)_{m_i} = -\frac{v_i}{\theta},~i=1,2,3, \qquad
	\displaystyle (\rho S)_{\mathcal{E}} = \frac{1}{\theta},
	\]
	with $m_i=\rho v_i~(i=1,2,3)$ and $\mathcal{E}=\rho\left(\frac{R}{\gamma-1}\theta + \frac{|\bv|^2}{2}\right)$, we have
	\[
	\left\{\begin{aligned}
	\eta_*& = \frac{R}{\gamma-1}\rho\theta - \tilde{\theta}\rho S + \rho\left[\left(\tilde{S} - \frac{R\gamma}{\gamma-1}\right)\tilde{\theta} + \frac{|\bv - \tilde{\bv}|^2}{2}\right] + R\tilde{\rho}\tilde{\theta} \\[1mm]
	&= R\rho\tilde{\theta} \Phi\Big(\frac{\tilde{\rho}}{\rho}\Big) + \frac{R}{\gamma-1}\rho\tilde{\theta} \Phi\Big(\frac{\theta}{\tilde{\theta}}\Big) + \frac12\rho|\bv-\tilde{\bv}|^2, \\[2mm]
	\bq_*& = \bv\eta + R(\bv - \tilde{\bv})(\rho\theta - \tilde{\rho}\tilde{\theta}),
	\end{aligned}
	\right.
	\]
	where $\Phi(\cdot)$ is the strictly convex function
	\[
	\Phi(s) = s - \ln s - 1.
	\]
	Then, for $\mathbf{X}$ in any closed bounded region in $\Xi = \{\mathbf{X}:\rho>0,\theta>0\}$, there exists a positive constant $C_0$ such that
	\[
	C_0^{-1}|(\varphi, \Psi, \xi)|^2 \leq \eta_* \leq C_0|(\varphi, \Psi, \xi)|^2.
	\]
	A straightforward  computation shows that
	\begin{equation}\label{01}
	\begin{array}{l}
	\di \eta_{*t} + \div \bq_* - \div\Big[\Psi(2\mu\varepsilon \mathbb{D}(\Psi) + \lambda\varepsilon\div\Psi\mathbb{I}) + \frac{\kappa\varepsilon\xi}{\theta}\nabla\xi\Big] +
	\frac{\tilde\theta}{\theta}\left(\frac{\mu\varepsilon}{2}|\nabla\Psi + (\nabla\Psi)^\top|^2 + \lambda\varepsilon(\div \Psi)^2\right) \\[3mm]
	\di \quad + \frac{\kappa\varepsilon\tilde\theta}{\theta^2}|\nabla\xi|^2 + \tilde{v}_{1x_1}\Big[ \rho\psi_1^2 + R(\gamma - 1)\rho\tilde{\theta}\Phi\Big(\frac{\tilde{\rho}}{\rho}\Big) + R\rho\tilde{\theta}\Phi\Big(\frac{\theta}{\tilde{\theta}}\Big)\Big] + \tilde{\theta}_{x_1}\rho\psi_1\Big(R\ln\frac{\tilde{\rho}}{\rho} + \frac{R}{\gamma - 1}\ln\frac{\theta}{\tilde{\theta}}\Big) \\[4mm]
	\di= \frac{2\tilde{v}_{1x_1}}{\theta}\xi(2\mu\varepsilon\psi_{1x_1} + \lambda\varepsilon\div\Psi) + \frac{\kappa\varepsilon}{\theta^2}\tilde{\theta}_{x_1}\xi\xi_{x_1} - \frac{(2\mu+\lambda)\varepsilon}{\tilde{\rho}}\bar{v}_{1x_1x_1}\varphi\psi_1 \\[3mm]
	\di \quad + (2\mu+\lambda)\varepsilon\Big(\frac{-\bar{v}_1z_1 + z_2}{\tilde{\rho}}\Big)_{x_1x_1}\psi_1 - \rho\psi_1Q_1 + \Big(F_1+F_2-\frac{\rho}{\tilde\rho}Q_2\Big)\frac{\xi}{\theta} \\[4mm]
	\di \quad + \frac{\rho}{\tilde{\rho}}\Big[(\gamma - 1)\Phi\Big(\frac{\tilde{\rho}}{\rho}\Big) - \Phi\Big(\frac{\tilde{\theta}}{\theta}\Big)\Big]\left[\kappa\varepsilon\bar{\theta}_{x_1x_1} + (2\mu + \lambda)\varepsilon\bar{v}_{1x_1}^2 + Q_2\right].
	\end{array}
	\end{equation}
	There exists a positive constant $C>0$ such that
	\[
	\begin{array}{ll}
	\di \tilde{v}_{1x_1}\Big[ \rho\psi_1^2 + R(\gamma - 1)\rho\tilde{\theta}\Phi\Big(\frac{\tilde{\rho}}{\rho}\Big) + R\rho\tilde{\theta}\Phi\Big(\frac{\theta}{\tilde{\theta}}\Big) \Big]  + \tilde{\theta}_{x_1}\rho\psi_1\Big(R\ln\frac{\tilde{\rho}}{\rho} + \frac{R}{\gamma - 1}\ln\frac{\theta}{\tilde{\theta}}\Big) \\[4mm]
	\di \geq 2C^{-1} \bar v_{1x_1}(\varphi^2 + \psi_1^2 + \xi^2) + \Big(\frac{-\bar{v}_1z_1 + z_2}{\tilde{\rho}}\Big)_{x_1}\Big[\rho\psi_1^2 + R(\gamma - 1)\rho\tilde{\theta}\Phi\Big(\frac{\tilde{\rho}}{\rho}\Big) + R\rho\tilde{\theta}\Phi\Big(\frac{\theta}{\tilde{\theta}}\Big)\Big] \\[4mm]
	\di + (\gamma - 1)\bar{v}_{1x_1}\Big[\Big(\frac{\big(\frac12\bar v_1^2 - \frac{R}{\gamma - 1}\bar{\theta}\big)z_1 - \bar{v}_1z_2 + z_3}{\tilde{\rho}} - \frac{(-\bar{v}_1z_1 + z_2)^2}{2\tilde{\rho}^2}\Big)\Big((\gamma - 1)\rho \Phi\Big(\frac{\tilde{\rho}}{\rho}\Big) + \rho \Phi\Big(\frac{\theta}{\tilde{\theta}}\Big)\Big)\Big] \\[5mm]
	\di + \Big[\frac{1}{\tilde{\rho}}\Big(\big(\frac12\bar v_1^2 - \frac{R}{\gamma - 1}\bar{\theta}\big)z_1 - \bar{v}_1z_2 + z_3\Big) - \frac{(-\bar{v}_1z_1 + z_2)^2}{2\tilde{\rho}^2}\Big]_{x_1}\rho\psi_1\Big((\gamma - 1)\ln\frac{\tilde\rho}{\rho} + \ln\frac{\theta}{\tilde{\theta}}\Big)\\[5mm]
	\di \geq C^{-1} \bar v_{1x_1}(\varphi^2 + \psi_1^2 + \xi^2) -C|(z_{1x_1},z_{2x_1},z_{3x_1})|(\varphi^2 + \psi_1^2 + \xi^2),
	\end{array}
	\]
	provided that $\varepsilon$ is sufficiently small.
	Integrating \eqref{01} over $[0, ~t]\times\Omega$ and using the above relation imply that
	\begin{align}
	\begin{aligned} \label{02}
	&\|(\varphi, \Psi, \xi)(t)\|^2 + \int_{0}^{t} \big[\|\bar{v}_{1x_1}^{1/2} (\varphi, \psi_1, \xi) \|^2 + \varepsilon\|(\nabla \Psi, \nabla\xi)\|^2\big] dt \\
	&\leq C \| (\varphi_0, \Psi_0, \xi_0) \|^2 + C \Big|\int_{0}^{t}\int_{\Omega} \Big[\frac{2\tilde{v}_{1x_1}}{\theta}\xi(2\mu\varepsilon\psi_{1x_1} + \lambda\varepsilon\div\Psi) + \frac{\kappa\varepsilon}{\theta^2}\tilde{\theta}_{x_1}\xi\xi_{x_1}\Big] dxdt \Big| \\
	& \quad + C \Big|\int_{0}^{t}\int_{\Omega} \Big[ \frac{(2\mu+\lambda)\varepsilon}{\tilde{\rho}}\bar{v}_{1x_1x_1}\varphi\psi_1 \Big] dxdt \Big|\\
	&\quad + C \Big|\int_{0}^{t}\int_{\Omega} \Big[(2\mu+\lambda)\varepsilon\Big(\frac{-\bar{v}_1z_1 + z_2}{\tilde{\rho}}\Big)_{x_1x_1}\psi_1\Big] dxdt \Big|\\
	&\quad + C \Big|\int_{0}^{t}\int_{\Omega} \Big[- \rho\psi_1Q_1+ \Big(F_1+F_2-\frac{\rho}{\tilde\rho}Q_2 \Big)\frac{\xi}{\theta}\Big] dxdt \Big|\\
	&\quad + C \Big|\int_{0}^{t}\int_{\Omega} \Big[ \frac{\rho}{\tilde{\rho}}\Big[(\gamma - 1)\Phi\Big(\frac{\tilde{\rho}}{\rho}\Big) - \Phi\Big(\frac{\tilde{\theta}}{\theta}\Big)\Big]\left[\kappa\varepsilon\bar{\theta}_{x_1x_1} + (2\mu + \lambda)\varepsilon\bar{v}_{1x_1}^2 + Q_2\right]\Big] dxdt \Big|\\
	&\di \quad + C \Big|\int_{0}^{t}\int_{\Omega} \Big[|(z_{1x_1},z_{2x_1},z_{3x_1})|(\varphi^2 + \psi_1^2 + \xi^2)\Big] dxdt \Big|\\
	&:=C \| (\varphi_0, \Psi_0, \xi_0) \|^2 +\sum_{i=1}^6 I_i.
	\end{aligned}
	\end{align}
  First, it follows from the Young  inequality, Lemma \ref{lemma2.2} and Lemma \ref{lemma2.3} that
  \begin{align*}
  & I_1\leq \frac{\varepsilon}{20} \int_{0}^{t} \|(\nabla\Psi,\nabla\xi)\|^2 dt + C\varepsilon\int_{0}^{t} \|(\tilde{v}_{1x_1},\tilde\theta_{x_1})\xi\|^2 dt\\
  &\quad \leq \frac{\varepsilon}{20} \int_{0}^{t} \|(\nabla\Psi,\nabla\xi)\|^2 dt +
   C\varepsilon\int_{0}^{t}\|\bar{v}_{1x_1}\xi\|^2 dt + C\varepsilon\int_{0}^{t} \|(z_{1x_1}, z_{2x_1}, z_{3x_1})\xi\|^2 dt \\
  &\quad \leq \frac{\varepsilon}{20} \int_{0}^{t} \|(\nabla\Psi,\nabla\xi)\|^2 dt + C\varepsilon\sup_{0 \leq t \leq T}\|\bar{v}_{1x_1}\|_{L^\infty}\int_{0}^{t} \|\bar{v}_{1x_1}^{1/2}\xi\|^2 dt+ C_T\frac{\varepsilon^3}{\delta^5}\sup_{0 \leq t \leq t_1(\varepsilon)}\|\xi\|^2 \\
  &\quad\leq \frac{\varepsilon}{20} \int_{0}^{t} \|(\nabla\Psi,\nabla\xi)\|^2 dt + C\varepsilon^{\frac{5}{6}}|\ln\varepsilon|^{-1}\int_{0}^{t} \|\bar{v}_{1x_1}^{1/2}\xi\|^2 dt + C_T\varepsilon^{\frac{13}{6}}|\ln\varepsilon|^{-5}\sup_{0 \leq t \leq t_1(\varepsilon)}\|\xi\|^2,
  \end{align*}
and also
  \begin{align*}
  & I_2 \leq C_T\varepsilon\sup_{0 \leq t \leq T}\|\bar{v}_{1x_1x_1}\|_{L^\infty}\sup_{0 \leq t \leq t_1(\varepsilon)}\|(\varphi, \psi_1)\|^2
  \leq C_T\varepsilon^{\frac{2}{3}}|\ln\varepsilon|^{-2}\sup_{0 \leq t \leq t_1(\varepsilon)}\|(\varphi, \psi_1)\|^2,\\
   &I_3 \leq \frac{1}{20}\sup_{0 \leq t \leq t_1(\varepsilon)} \|\psi_1\|^2 + C_T \varepsilon^2 \int_{0}^{t}\int_{\bbr} \Big|\Big(\frac{-\bar{v}_1z_1 + z_2}{\tilde{\rho}}\Big)_{x_1x_1}\Big|^2 dx_1dt \\
   &\quad \leq \frac{1}{20}\sup_{0 \leq t \leq t_1(\varepsilon)} \|\psi_1\|^2 + C_T \frac{\varepsilon^4}{\delta^6} \leq \frac{1}{20}\sup_{0 \leq t \leq t_1(\varepsilon)} \|\psi_1\|^2 + C_T \varepsilon^{3}|\ln\varepsilon|^{-6}.
  \end{align*}
By \eqref{Q1} and \eqref{Q2}, it holds that
\begin{align*}
&\Big|\int_{0}^{t}\int_{\Omega} \rho\psi_1Q_1 dxdt \Big|
\leq \frac{1}{20}\sup_{0 \leq t \leq t_1(\varepsilon)} \|\psi_1\|^2 + C_T \int_{0}^{t}\int_{\bbr} \Big|\Big[\frac{(\bar{v}_1z_1 - z_2)^2}{\tilde{\rho}}\Big]_{x_1}\Big|^2 dx_1dt \\
&\qquad \leq \frac{1}{20}\sup_{0 \leq t \leq t_1(\varepsilon)} \|\psi_1\|^2 + C_T \frac{\varepsilon^4}{\delta^7} \leq \frac{1}{20}\sup_{0 \leq t \leq t_1(\varepsilon)} \|\psi_1\|^2 + C_T \varepsilon^{\frac{17}{6}}|\ln\varepsilon|^{-7}.
\end{align*}
Then we can obtain the estimate of $I_4$,
\begin{align*}
I_4 \leq \frac{1}{20}\sup_{0 \leq t \leq t_1(\varepsilon)} \|(\varphi, \psi_1, \xi)\|^2 + C_T \varepsilon^{\frac{17}{6}}|\ln\varepsilon|^{-7}.
\end{align*}
Moreover, it holds that
$$
\begin{array}{ll}
\di I_5 \leq C\int_0^t (\varepsilon\|(\bar\theta_{x_1x_1}, \bar v_{1x_1}^2)\|_{L^\infty} + \|Q_2\|_{L^\infty}) \|(\varphi,\xi)\|^2 dt\leq C_T \varepsilon^{\frac{2}{3}}|\ln\varepsilon|^{-2}\sup_{0 \leq t \leq t_1(\varepsilon)}\|(\varphi, \xi)\|^2.
\end{array}
$$
By Lemma \ref{lemma2.2} and Lemma \ref{lemma2.3}, one has
\begin{align*}
&I_6 \leq C_T \sup_{0 \leq t \leq T}\|(z_{1x_1},z_{2x_1}, z_{3x_1})\|_{L^\infty(\bbr)} \sup_{0 \leq t \leq t_1(\varepsilon)} \|(\varphi, \psi_1, \xi)\|^2 \\
&\quad \leq C_T \varepsilon^{\frac{7}{12}}|\ln\varepsilon|^{-\frac{5}{2}} \sup_{0 \leq t \leq t_1(\varepsilon)} \|(\varphi, \psi_1, \xi)\|^2.
\end{align*}
Substituting the above estimates for $I_i~(i=1,2,\cdots,6)$ into \eqref{02} and taking $\varepsilon$ suitably small,
we can prove \eqref{LEM4.1} in Lemma \ref{lemma4.1}.
\end{proof}

\begin{remark}
If the hyperbolic scaled variables for the space and time, i.e., $\frac x\varepsilon$ and $\frac t \varepsilon$, are still used to normalize the dissipation coefficients to be $O(1)$ order as in our previous work \cite{LWW2} for the two-dimensional limit case, then the exactly same proof as in \cite{LWW2} could not be applied here due to the spatial 3D setting. Alternatively, the {\it a priori} estimates here would be carried out for the original non-scaled variables $(x,t)$ and then the dissipation terms are more singular compared with the two-dimensional scaled case in \cite{LWW2}. Consequently more accurate {\it a priori} assumptions with respect to the dissipations as in \eqref{PA} are crucially needed and some new observations on the cancellations of the physical structures for the flux terms and viscous terms are essentially used to justify the 3D limit.
\end{remark}
\begin{remark}
	The hyperbolic wave $(z_1, z_2, z_3)$ are crucially utilized in Lemma \ref{lemma4.1}. Without this hyperbolic wave and using only the planar rarefaction wave $(\bar\rho, \bar v_1, \bar{\theta})$ as the profile, the decay rate of the error terms in Lemma \ref{lemma4.1} would not be good enough, and we could not achieve the  estimates that are uniform in  the dissipation coefficients, 
		which is quite different from the   1D case where the hyperbolic wave is not necessary in order to justify the limit process.	 
%
\end{remark}

\subsection{First-order derivative estimates}

In this subsection, we derive the first-order derivative estimates, which are quite different from our previous paper for the 2D case in \cite{LWW2}. In order to obtain the decay rate, here we mainly apply the cancellations between the flux terms in the mass equation, momentum equation and the energy equation in \eqref{NS} due to the physical structures of the system. While for the 2D case in our previous work \cite{LWW2}, we crucially used the cancellations between the flux terms and viscosity terms since both the flux and viscous terms are the same order in the scaled independent variables $\frac{x_i}{\varepsilon}$ and $\frac t\varepsilon$.
\begin{lemma} \label{lemma4.2}
Under the assumption of Proposition \ref{proposition3.2}, there exists a positive constant $C_T$ independent of $\varepsilon$,  such that for $0\leq t \leq t_1(\varepsilon)$,
  \begin{align}
	\begin{aligned} \label{LEM4.2}
 	  &\sup_{0 \leq t \leq t_1(\varepsilon)} \| (\nabla\varphi, \nabla\Psi, \nabla\xi)(t) \|^2 + \int_{0}^{t_1(\varepsilon)} \Big[\|\bar{v}_{1x_1}^{1/2} \nabla\varphi\|^2 + \varepsilon\|(\nabla^2\Psi, \nabla^2\xi)\|^2 \Big]dt \\
 	  &\leq C_T\frac{\varepsilon^3}{\delta^8} + C \| (\nabla\varphi_0, \nabla\Psi_0, \nabla\xi_0) \|^2
 = C_T \varepsilon^{\frac{5}{3}}|\ln\varepsilon|^{-8} + C \| (\nabla\varphi_0, \nabla\Psi_0, \nabla\xi_0) \|^2.
	\end{aligned}
  \end{align}
\end{lemma}
\begin{proof}
	Multiplying the equation $\eqref{REF}_2$ by $-\triangle\Psi$ leads to
	\begin{align}
	\begin{aligned} \label{11}
	&\Big(\rho\frac{|\nabla\Psi|^2}{2}\Big)_t - \div\Big(\rho\psi_{it}\nabla\psi_i + \rho v_i\psi_{jx_i}\nabla\psi_j - \rho\bv\frac{|\nabla\Psi|^2}{2} + (\mu + \lam)\varepsilon\div\Psi\nabla\div\Psi \\
	& \quad - (\mu + \lam)\varepsilon\div\Psi\triangle\Psi\Big) + \mu\varepsilon|\triangle\Psi|^2 + (\mu + \lam)\varepsilon|\nabla \div\Psi|^2 - R\theta\nabla\varphi\cdot\triangle\Psi - R\rho\nabla\xi\cdot\triangle\Psi \\
	& = - \varphi_{x_i}\Psi_{x_i}\cdot\Psi_t - \tilde{\rho}_{x_1}\Psi_{x_1}\cdot\Psi_t - v_i\psi_{jx_i}\nabla\rho\cdot\nabla\psi_j - \rho\psi_{jx_i}\nabla\psi_i\cdot\nabla\psi_j - \rho\tilde{v}_{1x_1}|\Psi_{x_1}|^2 \\
	& \quad + \rho\tilde{v}_{1x_1}\psi_1\triangle\psi_1 + R\tilde{\rho}_{x_1}\Big(\theta - \frac{\rho}{\tilde{\rho}}\tilde{\theta}\Big)\triangle\psi_1 - (2\mu + \lam)\varepsilon\Big(\frac{-\bar{v}_1z_1 + z_2}{\tilde{\rho}}\Big)_{x_1x_1}\triangle\psi_1 \\
	& \quad + (2\mu + \lam)\varepsilon\bar{v}_{1x_1x_1}\frac{\varphi}{\tilde{\rho}}\triangle\psi_1 + Q_1\frac{\rho}{\tilde{\rho}}\triangle\psi_1\\
&	:= I(t,x).
	\end{aligned}
	\end{align}
We now take the gradient  in the first equation of \eqref{REF} and then multiply it by $\frac{R\theta}{\rho}\nabla\varphi$ to obtain
	\begin{align}
	\begin{aligned} \label{12}
	&\Big(\frac{R\theta}{\rho}\frac{|\nabla\varphi|^2}{2}\Big)_t + \div\Big(\frac{R\theta}{\rho}\bv\frac{|\nabla\varphi|^2}{2} - R\theta\varphi_{x_i}\nabla\psi_i\Big) + (R\theta\nabla\varphi\cdot\nabla\psi_i)_{x_i} \\
	&\quad + \frac{R(\gamma - 1)\theta}{\rho}\bar{v}_{1x_1}\frac{|\nabla\varphi|^2}{2} + \frac{R\theta}{\rho}\bar{v}_{1x_1}\varphi_{x_1}^2 + R\theta\nabla\varphi\cdot\triangle\Psi \\
	&= - \frac{R(\gamma - 1)\theta}{\rho}\div\Psi\frac{|\nabla\varphi|^2}{2} - \frac{R\theta}{\rho}\varphi_{x_i}\nabla\varphi\cdot\nabla\psi_i + R\xi_{x_i}\nabla\varphi\cdot\nabla\psi_i - R\varphi_{x_i}\nabla\xi\cdot\nabla\psi_i \\
	&\quad + \frac{\gamma - 1}{\rho^2}\frac{|\nabla\varphi|^2}{2}\Big[\kappa\varepsilon\triangle\xi +\kappa\varepsilon\tilde{\theta}_{x_1x_1} + \frac{\mu\varepsilon}{2}|\nabla\Psi + (\nabla\Psi)^\top|^2 + \lambda\varepsilon(\div\Psi)^2 \\
	&\quad + 2\tilde{v}_{1x_1}(2\mu\varepsilon\psi_{1x_1} + \lambda\varepsilon \div\Psi) + (2\mu + \lambda)\varepsilon(\tilde{v}_1\tilde{v}_{1x_1})_{x_1}\Big] + R\tilde{\theta}_{x_1}\nabla\varphi\cdot\nabla\psi_1 \\
	&\quad  - R\tilde{\theta}_{x_1}\nabla\varphi\cdot\Psi_{x_1}  - \frac{R\theta}{\rho}\Big(\frac{-\bar{v}_1z_1 + z_2}{\tilde{\rho}}\Big)_{x_1}\varphi_{x_1}^2- \frac{R(\gamma - 1)\theta}{\rho}\Big(\frac{-\bar{v}_1z_1 + z_2}{\tilde{\rho}}\Big)_{x_1}\frac{|\nabla\varphi|^2}{2}\\[2mm]
	&\quad  - \frac{R\theta}{\rho}\tilde{\rho}_{x_1}\varphi_{x_1}\div\Psi- \frac{R\theta}{\rho}\tilde{\rho}_{x_1x_1}\psi_1\varphi_{x_1} - \frac{R\theta}{\rho}\tilde{\rho}_{x_1}\nabla\varphi\cdot\nabla\psi_1 - \frac{R\theta}{\rho}\tilde{v}_{1x_1x_1}\varphi\varphi_{x_1}\\[2mm]
	& := J(t, x).
	\end{aligned}
	\end{align}
	Multiplying the third equation of \eqref{REF} by $-\frac{1}{\theta}\triangle\xi$, one has
	\begin{align}
	\begin{aligned} \label{13}
	&\Big(\frac{R}{\gamma - 1}\frac{\rho}{\theta}\frac{|\nabla\xi|^2}{2}\Big)_t - \div\Big(\frac{R}{\gamma - 1}\frac{\rho}{\theta}\xi_t\nabla\xi + \frac{R}{\gamma - 1}\frac{\rho}{\theta}v_i\xi_{x_i}\nabla\xi -\frac{R}{\gamma - 1}\frac{\rho}{\theta}\bv\frac{|\nabla\xi|^2}{2} \\
	&\quad + R\rho\div\Psi\nabla\xi + R\rho\nabla\psi_i\xi_{x_i}\Big) + (R\rho\nabla\psi_i\cdot\nabla\xi)_{x_i} + \frac{\kappa\varepsilon}{\theta}|\triangle\xi|^2 + R\rho\nabla\xi\cdot\triangle\Psi \\
	&= \frac{-R}{\gamma - 1}\frac{1}{\theta}\nabla\varphi\cdot\nabla\xi\xi_t - \frac{R}{\gamma - 1}\frac{1}{\theta}\tilde{\rho}_{x_1}\xi_{x_1}\xi_t + \frac{R}{\gamma - 1}\frac{\rho}{\theta^2}|\nabla\xi|^2\xi_t + \frac{R}{\gamma - 1}\frac{\rho}{\theta^2}\tilde{\theta}_{x_1}\xi_{x_1}\xi_t \\
	&\quad - \frac{R}{\gamma - 1}\frac{1}{\theta}v_i\xi_{x_i}\nabla\rho\cdot\nabla\xi + \frac{R}{\gamma - 1}\frac{\rho}{\theta^2}v_i\xi_{x_i}\nabla\theta\cdot\nabla\xi - \frac{R}{\gamma - 1}\frac{\rho}{\theta}\xi_{x_i}\nabla\psi_i\cdot\nabla\xi \\
	&\quad - \frac{R}{\gamma - 1}\frac{\rho}{\theta}\tilde{v}_{1x_1}|\xi_{x_1}|^2 + \frac{R\rho}{\theta}\div\Psi\frac{|\nabla\xi|^2}{2} + \frac{R\rho}{\theta}\tilde{v}_{1x_1}\frac{|\nabla\xi|^2}{2} - \frac{1}{\theta^2}\frac{|\nabla\xi|^2}{2}\Big[\kappa\varepsilon\triangle\xi +\kappa\varepsilon\tilde{\theta}_{x_1x_1} \\[2mm]
	&\quad + \frac{\mu\varepsilon}{2}|\nabla\Psi + (\nabla\Psi)^\top|^2 + \lambda\varepsilon(\div\Psi)^2 + 2\tilde{v}_{1x_1}(2\mu\varepsilon\psi_{1x_1} + \lambda\varepsilon \div\Psi) + (2\mu + \lambda)\varepsilon(\tilde{v}_1\tilde{v}_{1x_1})_{x_1}\Big] \\[2mm]
	&\quad - R\div\Psi\nabla\varphi\cdot\nabla\xi - R\tilde{\rho}_{x_1}\xi_{x_1}\div\Psi + R\varphi_{x_i}\nabla\psi_i\cdot\nabla\xi + R\tilde{\rho}_{x_1}\nabla\psi_1\cdot\nabla\xi - R\nabla\varphi\cdot\nabla\psi_i\xi_{x_i} \\
	&\quad - R\tilde{\rho}_{x_1}\psi_{ix_1}\xi_{x_i} + \frac{R}{\gamma - 1}\frac{\rho}{\theta}\tilde{\theta}_{x_1}\psi_1\triangle\xi + \frac{R\rho}{\theta}\tilde{v}_{1x_1}\xi\triangle\xi - \frac{\mu\varepsilon}{2\theta}|\nabla\Psi + (\nabla\Psi)^\top|^2\triangle\xi \\
	&\quad - \frac{\lambda\varepsilon}{\theta}(\div\Psi)^2\triangle\xi - \frac{2\tilde{v}_{1x_1}}{\theta}(2\mu\varepsilon\psi_{1x_1} + \lambda\varepsilon\div\Psi)\triangle\xi - \frac{F_1 + F_2}{\theta}\triangle\xi  + \frac{\rho}{\tilde\rho \theta}Q_2\triangle\xi\\ 
	&:= K(t, x).
	\end{aligned}
	\end{align}
  Now we add \eqref{11}, \eqref{12} and \eqref{13} together, 
  then take the integration   over $[0, ~t]\times\Omega$ to obtain
  \begin{align}
    \begin{aligned} \label{14}
      &\| (\nabla\varphi, \nabla\Psi, \nabla\xi)(t) \|^2 + \int_{0}^{t} \Big[\|\bar{v}_{1x_1}^{1/2} \nabla\varphi\|^2 + \varepsilon\|(\nabla^2\Psi, \nabla^2\xi)\|^2 \Big]dt \\
      &\leq C \| (\nabla\varphi_0, \nabla\Psi_0, \nabla\xi_0) \|^2 + C \Big|\int_{0}^{t}\int_{\Omega} I(t, x) + J(t, x) + K(t, x) dxdt\Big|,
    \end{aligned}
  \end{align}
  where $I, J$ and $K$ are defined in \eqref{11}, \eqref{12} and \eqref{13}, respectively.
  
We now make estimates  on  the right hand side of \eqref{14}.   
First, it holds that
  \begin{align}
  	\begin{aligned} \label{15}
    &C\Big|\int_{0}^{t}\int_{\Omega} \frac{R\theta}{\rho}\varphi_{x_i}\nabla\varphi\cdot\nabla\psi_i dxdt\Big|
    \leq C \int_{0}^{t} \|\nabla\varphi\|\|\nabla\varphi\|_{L^4}\|\nabla\Psi\|_{L^4} dt \\
    & \leq C \int_{0}^{t} \|\nabla\varphi\|\|\nabla\varphi\|^{1/4}\|\nabla\varphi\|_1^{3/4}\|\nabla\Psi\|^{1/4}\|\nabla\Psi\|_1^{3/4} dt \\
    & \leq \frac{\varepsilon}{160}\int_{0}^{t} \|\nabla\Psi\|_1^2 dt + C\varepsilon^{-3/5} \int_{0}^{t} \|\nabla\varphi\|^2\|\nabla\varphi\|_1^{6/5}\|\nabla\Psi\|^{2/5} dt \\
    & \leq \frac{\varepsilon}{160}\int_{0}^{t} \|\nabla\Psi\|_1^2 dt + C_T\ \varepsilon^{-\frac{3}{5} + \frac{2}{5}a_1 + \frac{6}{5}a_2}|\ln\varepsilon|^{-8/5}\sup_{0 \leq t \leq t_1(\varepsilon)}\|\nabla\varphi\|^2,
   \end{aligned}
  \end{align}
  where we have used H\"{o}lder's inequality, Sobolev's inequality and Young's inequality and in the last inequality we fully used the {\it a priori} assumptions \eqref{PA} which is one of the key points in the present paper.
  Similarly, the {\it a priori} assumptions \eqref{PA} will be utilized to estimate the nonlinear terms in the sequel. Then one has
  \begin{align*}
  &C\Big|\int_{0}^{t}\int_{\Omega} \frac{(\gamma - 1)\kappa\varepsilon}{\rho^2}\frac{|\nabla\varphi|^2}{2}\triangle\xi dxdt\Big|
   \leq \frac{\varepsilon}{160}\int_{0}^{t} \|\nabla^2\xi\|^2 dt + C\varepsilon \int_{0}^{t} \|\nabla\varphi\|_{L^4}^4 dt \\
  & \leq \frac{\varepsilon}{160}\int_{0}^{t} \|\nabla^2\xi\|^2 dt + C\varepsilon \int_{0}^{t} \|\nabla\varphi\|\|\nabla\varphi\|_1^3 dt \\
  & \leq \frac{\varepsilon}{160}\int_{0}^{t} \|\nabla^2\xi\|^2 dt + \frac{1}{160}\sup_{0 \leq t \leq t_1(\varepsilon)}\|\nabla\varphi\|^2 + C_T\varepsilon^2\sup_{0 \leq t \leq t_1(\varepsilon)}\|\nabla\varphi\|_1^6 \\
  & \leq \frac{\varepsilon}{160}\int_{0}^{t} \|\nabla^2\xi\|^2 dt + \frac{1}{160}\sup_{0 \leq t \leq t_1(\varepsilon)}\|\nabla\varphi\|^2 + C_T\ \varepsilon^2 (\varepsilon^{1/4}|\ln\varepsilon|^{-1})^6.
  \end{align*}
By Lemma \ref{lemma2.2} and Lemma \ref{lemma2.3}  we have 
  \begin{align*}
  &C\Big|\int_{0}^{t}\int_{\Omega} \frac{(\gamma - 1)\kappa\varepsilon}{\rho^2}\tilde{\theta}_{x_1x_1}\frac{|\nabla\varphi|^2}{2} dxdt\Big|
  \leq C\varepsilon \int_{0}^{t} \|\tilde{\theta}_{x_1x_1}\|_{L^\infty}\|\nabla\varphi\|^2 dt \\
  & \leq C_T\varepsilon\left(\frac{1}{\delta^2} + \frac{\varepsilon}{\delta^{7/2}}\right)\sup_{0 \leq t \leq t_1(\varepsilon)}\|\nabla\varphi\|^2 \leq C_T \varepsilon^{\frac{2}{3}}|\ln\varepsilon|^{-2} \sup_{0 \leq t \leq t_1(\varepsilon)}\|\nabla\varphi\|^2.
  \end{align*}
Using the  H\"{o}lder  inequality, the Sobolev  inequality and the Young  inequality yields
  \begin{align*}
  &C\Big|\int_{0}^{t}\int_{\Omega} \frac{(\gamma - 1)\lambda\varepsilon}{\rho^2}\frac{|\nabla\varphi|^2}{2}(\div\Psi)^2 dxdt\Big|
  \leq C\varepsilon \int_{0}^{t} \|\nabla\varphi\|\|\nabla\varphi\|_{L^6}\|\nabla\Psi\|_{L^6}^2 dt \\
  & \leq C\varepsilon \int_{0}^{t} \|\nabla\varphi\|\|\nabla\varphi\|_1\|\nabla\Psi\|_1^2 dt
  \leq C\varepsilon^2|\ln\varepsilon|^{-2}\int_{0}^{t} \|\nabla\Psi\|_1^2 dt.
  \end{align*}
  By H\"{o}lder's inequality, Sobolev's inequality, Young's inequality, Lemma \ref{lemma2.2} and Lemma \ref{lemma2.3},  it holds that
  \begin{align*}
  &C\Big|\int_{0}^{t}\int_{\Omega} \frac{\gamma - 1}{\rho^2}\tilde{v}_{1x_1}(2\mu\varepsilon\psi_{1x_1} + \lambda\varepsilon \div\Psi)|\nabla\varphi|^2 dxdt\Big| \\
  &\leq C\varepsilon \int_{0}^{t} \|\tilde{v}_{1x_1}\|_{L^\infty}\|\nabla\varphi\|\|\nabla\varphi\|_{L^4}\|\nabla\Psi\|_{L^4} dt \\
  & \leq C \varepsilon\int_{0}^{t} \|\tilde{v}_{1x_1}\|_{L^\infty}\|\nabla\varphi\|\|\nabla\varphi\|^{1/4}\|\nabla\varphi\|_1^{3/4}\|\nabla\Psi\|^{1/4}\|\nabla\Psi\|_1^{3/4} dt \\
  & \leq \frac{\varepsilon}{160}\int_{0}^{t} \|\nabla\Psi\|_1^2 dt + C\varepsilon \int_{0}^{t} \|\tilde{v}_{1x_1}\|_{L^\infty}^{8/5}\|\nabla\varphi\|^2\|\nabla\varphi\|_1^{6/5}\|\nabla\Psi\|^{2/5} dt \\
  & \leq \frac{\varepsilon}{160}\int_{0}^{t} \|\nabla\Psi\|_1^2 dt + C_T \varepsilon^{\frac{4}{3}}|\ln\varepsilon|^{-\frac{16}{5}} \sup_{0 \leq t \leq t_1(\varepsilon)}\|\nabla\varphi\|^2.
  \end{align*}
By Lemma \ref{lemma2.2} and Lemma \ref{lemma2.3} one has
  \begin{align*}
  C\Big|\int_{0}^{t}\int_{\Omega} \frac{R(\gamma - 1)\theta}{\rho}\left(\frac{-\bar{v}_1z_1 + z_2}{\tilde{\rho}}\right)_{x_1}\frac{|\nabla\varphi|^2}{2} dxdt\Big|
  \leq C_T\varepsilon^{\frac{7}{12}}|\ln\varepsilon|^{-\frac{5}{2}}\sup_{0 \leq t \leq t_1(\varepsilon)}\|\nabla\varphi\|^2.
  \end{align*}
  Then one has  the following estimate, from Young's inequality, Lemma \ref{lemma2.2}, Lemma \ref{lemma2.3} and Lemma \ref{lemma4.1},  
  \begin{align*}
  &C\Big|\int_{0}^{t}\int_{\Omega} R\tilde{\theta}_{x_1}\nabla\varphi\cdot\nabla\psi_1 dxdt\Big| \\
  & \leq C\int_{0}^{t}\int_{\Omega} |\bar{\theta}_{x_1}\nabla\varphi\cdot\nabla\psi_1| dxdt \\
  & \quad + C\int_{0}^{t}\int_{\Omega} \Big|\left(\frac{1}{\tilde{\rho}}\Big(\Big(\frac12v_1^2 - \frac{R}{\gamma - 1}\bar{\theta}\Big)z_1 - \bar{v}_1z_2 + z_3\Big) - \frac{(-\bar{v}_1z_1 + z_2)^2}{2\tilde{\rho}^2}\right)_{x_1}\nabla\varphi\cdot\nabla\psi_1\Big| dxdt \\
  &\leq \frac{1}{160} \int_{0}^{t} \|\bar{v}_{1x_1}^{1/2}\nabla\varphi\|^2 dt + C\frac{1}{\delta}\int_{0}^{t} \|\nabla\psi_1\|^2 dt + \frac{1}{160}\sup_{0 \leq t \leq t_1(\varepsilon)}\|\nabla\varphi\|^2 + C_T\frac{\varepsilon^2}{\delta^5}\int_{0}^{t} \|\nabla\psi_1\|^2 dt \\
  &\leq \frac{1}{160} \Big(\sup_{0 \leq t \leq t_1(\varepsilon)}\|\nabla\varphi\|^2 + \int_{0}^{t} \|\bar{v}_{1x_1}^{1/2}\nabla\varphi\|^2 dt\Big) + C_T \varepsilon^{\frac{5}{3}}|\ln\varepsilon|^{-8}.
  \end{align*}
  Similarly,  
  \begin{align*}
  &C\Big|\int_{0}^{t}\int_{\Omega} \frac{R\theta}{\rho}\tilde{\rho}_{x_1x_1}\psi_1\varphi_{x_1} dxdt\Big| \\
  & \leq C\int_{0}^{t}\int_{\Omega} |\bar{\rho}_{x_1x_1}\psi_1\varphi_{x_1}| dxdt + C\int_{0}^{t}\int_{\Omega} |z_{1x_1x_1}\psi_1\varphi_{x_1}| dxdt \\
  & \leq C\frac{1}{\delta}\int_{0}^{t}\int_{\Omega} |\bar{v}_{1x_1}\psi_1\varphi_{x_1}| dxdt + C\int_{0}^{t} \|z_{1x_1x_1}\|_{L^\infty}\|\psi_1\|\|\varphi_{x_1}\| dt \\
  &\leq \frac{1}{160} \int_{0}^{t} \|\bar{v}_{1x_1}^{1/2}\varphi_{x_1}\|^2 dt + C\frac{1}{\delta^2}\int_{0}^{t} \|\bar{v}_{1x_1}^{1/2}\psi_1\|^2 dt + \frac{1}{160}\sup_{0 \leq t \leq t_1(\varepsilon)}\|\varphi_{x_1}\|^2 + C_T\frac{\varepsilon^2}{\delta^7}\sup_{0 \leq t \leq t_1(\varepsilon)}\|\psi_1\|^2 \\
  &\leq \frac{1}{160} \Big(\sup_{0 \leq t \leq t_1(\varepsilon)}\|\varphi_{x_1}\|^2 + \int_{0}^{t} \|\bar{v}_{1x_1}^{1/2}\varphi_{x_1}\|^2 dt\Big) + C_T \varepsilon^{\frac{5}{2}}|\ln\varepsilon|^{-9}.
  \end{align*}
 From Lemma \ref{lemma2.2} and Lemma \ref{lemma2.3}, we obtain
  \begin{align*}
  C\Big|\int_{0}^{t}\int_{\Omega} \rho\tilde{v}_{1x_1}|\Psi_{x_1}|^2 dxdt\Big|
  \leq C\int_{0}^{t} \|\tilde{v}_{1x_1}\|_{L^\infty}\|\Psi_{x_1}\|^2 dt
  \leq C_T \frac{\varepsilon^3}{\delta^8}
  \leq C_T \varepsilon^{\frac{5}{3}}|\ln\varepsilon|^{-8}.
  \end{align*}
 Using Young's inequality, Lemma \ref{lemma2.2} and Lemma \ref{lemma2.3} yields
  \begin{align*}
  &C\Big|\int_{0}^{t}\int_{\Omega} \rho\tilde{v}_{1x_1}\psi_1\triangle\psi_1 dxdt\Big| \\
  & \leq \frac{\varepsilon}{160} \int_{0}^{t} \|\nabla^2\psi_1\|^2 dt + C\varepsilon^{-1}\int_{0}^{t} \Big(\|\bar{v}_{1x_1}\psi_1\|^2 + \Big\|\Big(\frac{-\bar{v}_1z_1 + z_2}{\tilde{\rho}}\Big)_{x_1}\psi_1\Big\|^2\Big) dt \\
  & \leq \frac{\varepsilon}{160} \int_{0}^{t} \|\nabla^2\psi_1\|^2 dt + C\varepsilon^{-1}\frac{1}{\delta}\int_{0}^{t} \|\bar{v}_{1x_1}^{1/2}\psi_1\|^2 dt + C_T\varepsilon^{-1}\frac{\varepsilon^2}{\delta^5} \sup_{0 \leq t \leq t_1(\varepsilon)}\|\psi_1\|^2 \\
  & \leq \frac{\varepsilon}{160} \int_{0}^{t} \|\nabla^2\psi_1\|^2 dt + C_T\frac{\varepsilon^3}{\delta^8}
  \leq \frac{\varepsilon}{160} \int_{0}^{t} \|\nabla^2\psi_1\|^2 dt + C_T \varepsilon^{\frac{5}{3}}|\ln\varepsilon|^{-8}.
  \end{align*}
  Similarly, one has
  \begin{align*}
  &C\Big|\int_{0}^{t}\int_{\Omega} (2\mu + \lam)\varepsilon\Big(\frac{-\bar{v}_1z_1 + z_2}{\tilde{\rho}}\Big)_{x_1x_1}\triangle\psi_1 dxdt\Big| \\
  & \leq \frac{\varepsilon}{160} \int_{0}^{t} \|\nabla^2\psi_1\|^2 dt + C\varepsilon\int_{0}^{t}\int_{\bbr} \Big|\Big(\frac{-\bar{v}_1z_1 + z_2}{\tilde{\rho}}\Big)_{x_1x_1}\Big|^2 dx_1dt \\
  & \leq \frac{\varepsilon}{160} \int_{0}^{t} \|\nabla^2\psi_1\|^2 dt + C_T \varepsilon^{2}|\ln\varepsilon|^{-6},\\[4mm]
  &C\Big|\int_{0}^{t}\int_{\Omega} (2\mu + \lam)\varepsilon\bar{v}_{1x_1x_1}\frac{\varphi}{\tilde{\rho}}\triangle\psi_1 dxdt\Big| \\
  &
   \leq \frac{\varepsilon}{160} \int_{0}^{t} \|\nabla^2\psi_1\|^2 dt + C\varepsilon\int_{0}^{t} \|\bar{v}_{1x_1x_1}\varphi\|^2 dt \\
  & \leq \frac{\varepsilon}{160} \int_{0}^{t} \|\nabla^2\psi_1\|^2 dt + C\frac{\varepsilon}{\delta^3}\int_{0}^{t} \|\bar{v}_{1x_1}^{1/2}\varphi\|^2 dt \\
  &
   \leq \frac{\varepsilon}{160} \int_{0}^{t} \|\nabla^2\psi_1\|^2 dt + C_T \varepsilon^{\frac{10}{3}}|\ln\varepsilon|^{-10}
  \end{align*}
  and
  \begin{align*}
  &C\Big|\int_{0}^{t}\int_{\Omega} Q_1\frac{\rho}{\tilde{\rho}}\triangle\psi_1 dxdt\Big| \\
  & \leq \frac{\varepsilon}{160} \int_{0}^{t} \|\nabla^2\psi_1\|^2 dt + C\varepsilon^{-1}\int_{0}^{t}\int_{\bbr} \Big|\Big(\frac{(-\bar{v}_1z_1 + z_2)^2}{\tilde{\rho}}\Big)_{x_1}\Big|^2 dx_1dt \\
  & \leq \frac{\varepsilon}{160} \int_{0}^{t} \|\nabla^2\psi_1\|^2 dt + C_T\varepsilon^{\frac{11}{6}}|\ln\varepsilon|^{-7}.
  \end{align*}
Applying Young's inequality, Lemma \ref{lemma2.2} and Lemma \ref{lemma2.3} leads to
  \begin{align*}
  &C\Big|\int_{0}^{t}\int_{\Omega} \tilde{\rho}_{x_1}\tilde{v}_{1x_1}\psi_1\psi_{1x_1} dxdt\Big| \\
  & \leq C\int_{0}^{t} \|\psi_{1x_1}\|^2 dt + C\int_{0}^{t} \|\bar{\rho}_{x_1}\tilde{v}_{1x_1}\psi_1\|^2 dt + C_T\sup_{0 \leq t \leq t_1(\varepsilon)}\|z_{1x_1}\tilde{v}_{1x_1}\psi_1\|^2 \\
  & \leq C_T\varepsilon^{\frac{11}{6}}|\ln\varepsilon|^{-7} + C\sup_{0 \leq t \leq T}\|\bar{v}_{1x_1}\|_{L^\infty}\|\tilde{v}_{1x_1}\|_{L^\infty}^2\int_{0}^{t} \|\bar{v}_{1x_1}^{1/2}\psi_1\|^2 dt\\
  &\qquad + C_T\sup_{0 \leq t \leq t_1(\varepsilon)}\|z_{1x_1}\|_{L^\infty}^2\|\tilde{v}_{1x_1}\|_{L^\infty}^2\|\psi_1\|^2 \\
  & \leq C_T\varepsilon^{\frac{11}{6}}|\ln\varepsilon|^{-7} + C_T\varepsilon^{\frac{7}{3}}|\ln\varepsilon|^{-10},
  \end{align*}
  \begin{align*}
  &C\Big|\int_{0}^{t}\int_{\Omega} \frac{\mu\varepsilon}{\rho}\tilde{\rho}_{x_1}\Psi_{x_1}\cdot\triangle\Psi dxdt\Big| \\
  &\leq \frac{\varepsilon}{160} \int_{0}^{t} \|\nabla^2\Psi\|^2 dt + C\varepsilon\int_{0}^{t} \|\tilde{\rho}_{x_1}\|_{L^\infty}^2\|\Psi_{x_1}\|^2 dt \qquad\qquad\qquad\qquad \quad \\
  & \leq \frac{\varepsilon}{160} \int_{0}^{t} \|\nabla^2\Psi\|^2 dt + C_T \varepsilon^{\frac{5}{2}}|\ln\varepsilon|^{-9},
  \end{align*}
  \begin{align*}
  &C\Big|\int_{0}^{t}\int_{\Omega} \frac{(2\mu + \lam)\varepsilon}{\rho}\tilde{\rho}_{x_1}\Big(\frac{-\bar{v}_1z_1 + z_2}{\tilde{\rho}}\Big)_{x_1x_1}\psi_{1x_1} dxdt\Big| \\
  & \leq \frac{1}{160} \sup_{0 \leq t \leq t_1(\varepsilon)}\|\psi_{1x_1}\|^2 + C_T\varepsilon^2\sup_{0 \leq t \leq T}\Big\|\tilde{\rho}_{x_1}\Big(\frac{-\bar{v}_1z_1 + z_2}{\tilde{\rho}}\Big)_{x_1x_1}\Big\|^2 \\
  & \leq \frac{1}{160} \sup_{0 \leq t \leq t_1(\varepsilon)}\|\psi_{1x_1}\|^2 + C_T \varepsilon^{\frac{8}{3}}|\ln\varepsilon|^{-8},
  \end{align*}

  \begin{align*}
  &C\Big|\int_{0}^{t}\int_{\Omega} \frac{(2\mu + \lam)\varepsilon}{\rho\tilde{\rho}}\tilde{\rho}_{x_1}\bar{v}_{1x_1x_1}\varphi\psi_{1x_1} dxdt\Big| \\
  &\leq \frac{1}{160}\sup_{0 \leq t \leq t_1(\varepsilon)}\|\psi_{1x_1}\|^2 + C_T\varepsilon^2\int_{0}^{t} \|\tilde{\rho}_{x_1}\|_{L^\infty}^2\|\bar{v}_{1x_1x_1}\varphi\|^2 dt \\
  &\leq \frac{1}{160}\sup_{0 \leq t \leq t_1(\varepsilon)}\|\psi_{1x_1}\|^2 + C_T\varepsilon^2\Big(\frac{1}{\delta^2} + \frac{\varepsilon^2}{\delta^5}\Big)\frac{1}{\delta^3}\int_{0}^{t} \|\bar{v}_{1x_1}^{1/2}\varphi\|^2 dt \\
  & \leq \frac{1}{160}\sup_{0 \leq t \leq t_1(\varepsilon)}\|\psi_{1x_1}\|^2 + C_T \varepsilon^{4}|\ln\varepsilon|^{-12},
  \end{align*}
  and
  \begin{align*}
  &C\Big|\int_{0}^{t}\int_{\Omega} \frac{1}{\tilde{\rho}}Q_1\tilde{\rho}_{x_1}\psi_{1x_1} dxdt\Big| \\
  & \leq \frac{1}{160} \sup_{0 \leq t \leq t_1(\varepsilon)}\|\psi_{1x_1}\|^2 + C_T\sup_{0 \leq t \leq T}\Big\|\Big(\frac{1}{\tilde{\rho}}(-\bar{v}_1z_1 + z_2)^2\Big)_{x_1}\tilde{\rho}_{x_1}\Big\|^2 \\
  & \leq \frac{1}{160} \sup_{0 \leq t \leq t_1(\varepsilon)}\|\psi_{1x_1}\|^2 + C_T \varepsilon^{\frac{5}{2}}|\ln\varepsilon|^{-9}.
  \end{align*}
Moreover, the remaining terms in \eqref{14} can be estimated similarly,  and  we omit the details  for the sake of conciseness. 
  Then substituting all these estimates into \eqref{14}, using the standard elliptic estimates $\|\triangle\Psi\| \sim\|\nabla^2\Psi\|$ and $\|\triangle\xi\| \sim\|\nabla^2\xi\|$ and taking
  $\varepsilon$ suitably small,
  we can prove Lemma \ref{lemma4.2}.

\end{proof}

\subsection{Second-order derivative estimates}
In this subsection, we carry out the second-order derivative estimates, which are also different from the 2D case in \cite{LWW2}.
\begin{lemma} \label{lemma4.3}
Under the assumption of Proposition \ref{proposition3.2}, there exists a positive constant $C_T$ independent of $\varepsilon$,  such that for $0\leq t \leq t_1(\varepsilon)$,
	\begin{align}
	\begin{aligned} \label{LEM4.3}
	&\sup_{0 \leq t \leq t_1(\varepsilon)} \| (\nabla^2\varphi, \nabla^2\Psi, \nabla^2\xi)(t) \|^2 + \int_{0}^{t_1(\varepsilon)} \Big[\|\bar{v}_{1x_1}^{1/2} \nabla^2\varphi\|^2 + \varepsilon\|(\nabla^3\Psi, \nabla^3\xi)\|^2 \Big]dt \\
	&\leq C_T \frac{\varepsilon^2}{\delta^9} + C \| (\nabla^2\varphi_0, \nabla^2\Psi_0, \nabla^2\xi_0) \|^2 =C_T\varepsilon^{\frac{1}{2}}|\ln\varepsilon|^{-9} + C \| (\nabla^2\varphi_0, \nabla^2\Psi_0, \nabla^2\xi_0) \|^2.
	\end{aligned}
	\end{align}
\end{lemma}
\begin{proof}
	First, we apply the  operator $\nabla^2$ on   $\eqref{REF}_1$ and   multiply  the resulting equation by $R\nabla^2\varphi$ to get
		\begin{align}
	\begin{aligned} \label{22}
	&\Big(R\frac{|\nabla^2\varphi|^2}{2}\Big)_t + \div\Big(R\bv\frac{|\nabla^2\varphi|^2}{2} - R\rho\varphi_{x_ix_j}\nabla\psi_{ix_j}\Big) + (R\rho\nabla^2\varphi\cdot\nabla^2\psi_i)_{x_i} + R\bar{v}_{1x_1}\frac{|\nabla^2\varphi|^2}{2} \\[2mm]
	&\quad + 2R\bar{v}_{1x_1}|\nabla\varphi_{x_1}|^2 + R\rho\nabla^2\varphi\cdot\nabla\triangle\Psi \\[2mm]
	&= -R\div\Psi\frac{|\nabla^2\varphi|^2}{2} - R\Big(\frac{-\bar{v}_1z_1 + z_2}{\tilde{\rho}}\Big)_{x_1}\frac{|\nabla^2\varphi|^2}{2} - 2R\psi_{ix_j}\nabla\varphi_{x_i}\cdot\nabla\varphi_{x_j} \\[2mm]
	&\quad - 2R\Big(\frac{-\bar{v}_1z_1 + z_2}{\tilde{\rho}}\Big)_{x_1}|\nabla\varphi_{x_1}|^2 - R\tilde{v}_{1x_1x_1}\varphi_{x_1}\varphi_{x_1x_1} - R\tilde{\rho}_{x_1x_1}\div\Psi\varphi_{x_1x_1} \\[2mm]
	&\quad - 2R\varphi_{x_i}\nabla\varphi_{x_i}\cdot\nabla\div\Psi - 2R\tilde{\rho}_{x_1}\nabla\varphi_{x_1}\cdot\nabla\div\Psi - R\nabla\varphi\cdot\nabla\psi_{ix_j}\varphi_{x_ix_j} \\[2mm]
	&\quad - R\tilde{\rho}_{x_1}\nabla^2\varphi\cdot\nabla\Psi_{x_1} - R\tilde{\rho}_{x_1x_1x_1}\psi_1\varphi_{x_1x_1} - 2R\tilde{\rho}_{x_1x_1} \nabla\psi_1\cdot\nabla\varphi_{x_1} - R\tilde{v}_{1x_1x_1x_1}\varphi\varphi_{x_1x_1} \\[2mm]
	&\quad - 2R\tilde{v}_{1x_1x_1}\nabla\varphi\cdot\nabla\varphi_{x_1}\\[2mm]
	&
	:= L(t, x).
	\end{aligned}
	\end{align}
Now we divide $\eqref{REF}_2$ by $\rho$, apply  the gradient to the resulting equation and then multiply it by $-\frac{\rho^2}{\theta}\nabla\triangle\Psi$ to obtain
	\begin{align}
	\begin{aligned} \label{21}
	&\Big(\frac{\rho^2}{\theta}\frac{|\nabla^2\Psi|^2}{2}\Big)_t - \Big(\frac{\rho^2}{\theta}\nabla\Psi_t\cdot\nabla\Psi_{x_i}\Big)_{x_i} - \Big(\frac{\rho^2}{\theta}v_i\nabla\Psi_{x_i}\cdot\nabla\Psi_{x_j}\Big)_{x_j} + \div\Big(\frac{\rho^2}{\theta}\bv\frac{|\nabla^2\Psi|^2}{2}\Big) \\[3mm]
	&  - (\mu + \lam)\varepsilon\div\Big(\frac{\rho}{\theta}\div\Psi_{x_j}\nabla\div\Psi_{x_j}\Big) + (\mu + \lam)\varepsilon\Big(\frac{\rho}{\theta}\div\Psi_{x_j}\triangle\psi_{ix_j}\Big)_{x_i} + \mu\varepsilon\frac{\rho}{\theta}|\nabla\triangle\Psi|^2 \\
	& + (\mu + \lam)\varepsilon\frac{\rho}{\theta}|\nabla^2\div\Psi|^2 - R\rho\nabla^2\varphi\cdot\nabla\triangle\Psi - R\frac{\rho^2}{\theta}\nabla^2\xi\cdot\nabla\triangle\Psi=\sum_{i=1}^6M_i(t,x),\\
	\end{aligned}
	\end{align}
	where
\begin{align*}
  & 
M_1(t,x) := - \frac{2\rho}{\theta}\varphi_{x_i}\nabla\Psi_{x_i}\cdot\nabla\Psi_t - \frac{2\rho}{\theta}\tilde{\rho}_{x_1}\nabla\Psi_{x_1}\cdot\nabla\Psi_t + \frac{\rho^2}{\theta^2}\xi_{x_i}\nabla\Psi_{x_i}\cdot\nabla\Psi_t + \frac{\rho^2}{\theta^2}\tilde{\theta}_{x_1}\nabla\Psi_{x_1}\cdot\nabla\Psi_t,\\[3mm]
&
\di M_2(t,x):=- \frac{2\rho}{\theta}\rho_{x_j}v_i\nabla\Psi_{x_i}\cdot\nabla\Psi_{x_j} + \frac{\rho^2}{\theta^2}\theta_{x_j}v_i\nabla\Psi_{x_i}\cdot\nabla\Psi_{x_j} - \frac{\rho^2}{\theta}\psi_{ix_j}\nabla\Psi_{x_i}\cdot\nabla\Psi_{x_j}\\[3mm]
& \qquad\qquad\quad - \frac{\rho^2}{\theta}\tilde{v}_{1x_1}|\nabla\Psi_{x_1}|^2 + (\gamma - 2)\frac{\rho^2}{\theta}\div\Psi\frac{|\nabla^2\Psi|^2}{2} + (\gamma - 2)\frac{\rho^2}{\theta}\tilde{v}_{1x_1}\frac{|\nabla^2\Psi|^2}{2},\\[3mm]
&
M_3(t,x):=- \frac{(\gamma - 1)}{R}\frac{\rho}{\theta^2}\frac{|\nabla^2\Psi|^2}{2}\Big[\kappa\varepsilon\triangle\xi +\kappa\varepsilon\tilde{\theta}_{x_1x_1} + \frac{\mu\varepsilon}{2}|\nabla\Psi + (\nabla\Psi)^\top|^2 + \lambda\varepsilon(\div\Psi)^2 \\[3mm]
&\qquad\qquad\quad +2\tilde{v}_{1x_1}(2\mu\varepsilon\psi_{1x_1} + \lambda\varepsilon \div\Psi) + (2\mu + \lambda)\varepsilon(\tilde{v}_1\tilde{v}_{1x_1})_{x_1}\Big],
\end{align*}
and 
\begin{align*}
  &
M_4(t,x):= \frac{\rho^2}{\theta}\psi_{jx_i}\nabla\psi_i\cdot\nabla\triangle\psi_j + \frac{\rho^2}{\theta}\tilde{v}_{1x_1}\Psi_{x_1}\cdot\triangle\Psi_{x_1} + \frac{R\rho}{\theta}\varphi_{x_i}\nabla\xi\cdot\nabla\triangle\psi_i+ \frac{R\rho}{\theta}\tilde{\theta}_{x_1}\nabla\varphi\cdot\triangle\Psi_{x_1} \\[3mm]
& \qquad\qquad\quad  - R\varphi_{x_i}\nabla\varphi\cdot\nabla\triangle\psi_i - R\tilde{\rho}_{x_1}\nabla\varphi\cdot\triangle\Psi_{x_1} + \frac{\rho^2}{\theta}\tilde{v}_{1x_1x_1}\psi_1\triangle\psi_{1x_1} + \frac{\rho^2}{\theta}\tilde{v}_{1x_1}\nabla\psi_1\cdot\nabla\triangle\psi_1 \\
& \qquad\qquad\quad  + \frac{R\rho^2}{\theta}\tilde{\rho}_{x_1x_1}\Big(\frac{\theta}{\rho} - \frac{\tilde{\theta}}{\tilde{\rho}}\Big)\triangle\psi_{1x_1} + \frac{R\rho}{\theta}\tilde{\rho}_{x_1}\nabla\xi\cdot\nabla\triangle\psi_1 - \frac{R\rho}{\theta\tilde{\rho}}\tilde{\rho}_{x_1}\tilde{\theta}_{x_1}\varphi\triangle\psi_{1x_1}\\
&\qquad\qquad\quad  - R\tilde{\rho}_{x_1}\nabla\varphi\cdot\nabla\triangle\psi_1- \frac{R\rho^2}{\theta}\tilde{\rho}_{x_1}^2\Big(\frac{\theta}{\rho^2} - \frac{\tilde{\theta}}{\tilde{\rho}^2}\Big)\triangle\psi_{1x_1} + \frac{\mu\varepsilon}{\theta}\triangle\psi_i\nabla\varphi\cdot\nabla\triangle\psi_i \\
&\qquad\qquad\quad  + \frac{\mu\varepsilon}{\theta}\tilde{\rho}_{x_1}\triangle\Psi\cdot\triangle\Psi_{x_1},\\[4mm]
&
M_5(t,x):= \frac{(\mu + \lam)\varepsilon}{\theta}\varphi_{x_i}\div\Psi_{x_j}\triangle\psi_{ix_j} + \frac{(\mu + \lam)\varepsilon}{\theta}\tilde{\rho}_{x_1}\div\Psi_{x_j}\triangle\psi_{1x_j} \qquad\qquad\qquad\qquad\qquad\qquad\\[3mm]
&\qquad\qquad\quad  - (\mu + \lam)\varepsilon\frac{\rho}{\theta^2}\xi_{x_i}\div\Psi_{x_j}\triangle\psi_{ix_j} - (\mu + \lam)\varepsilon\frac{\rho}{\theta^2}\tilde{\theta}_{x_1}\div\Psi_{x_j}\triangle\psi_{1x_j} \\[3mm]
& \qquad\qquad\quad - \frac{(\mu + \lam)\varepsilon}{\theta}\div\Psi_{x_j}\nabla\varphi\cdot\nabla\div\Psi_{x_j} - \frac{(\mu + \lam)\varepsilon}{\theta}\tilde{\rho}_{x_1}\div\Psi_{x_j}\div\Psi_{x_1x_j} \\[3mm]
&\qquad\qquad\quad  + (\mu + \lam)\varepsilon\frac{\rho}{\theta^2}\div\Psi_{x_j}\nabla\xi\cdot\nabla\div\Psi_{x_j} + (\mu + \lam)\varepsilon\frac{\rho}{\theta^2}\tilde{\theta}_{x_1}\div\Psi_{x_j}\div\Psi_{x_1x_j} \\[3mm]
& \qquad\qquad\quad + \frac{(\mu + \lam)\varepsilon}{\theta}\div\Psi_{x_i}\nabla\varphi\cdot\nabla\triangle\psi_i + \frac{(\mu + \lam)\varepsilon}{\theta}\tilde{\rho}_{x_1}\nabla\div\Psi\cdot\triangle\Psi_{x_1},\\[3mm]
& M_6(t,x):=- (2\mu + \lam)\varepsilon\frac{\rho}{\theta}\Big(\frac{-\bar{v}_1z_1 + z_2}{\tilde{\rho}}\Big)_{x_1x_1x_1}\triangle\psi_{1x_1} + \frac{(2\mu + \lam)\varepsilon}{\theta}\Big(\frac{-\bar{v}_1z_1 + z_2}{\tilde{\rho}}\Big)_{x_1x_1}\tilde{\rho}_{x_1}\triangle\psi_{1x_1} \\
&\qquad\qquad\quad + \frac{(2\mu + \lam)\varepsilon}{\theta}\Big(\frac{-\bar{v}_1z_1 + z_2}{\tilde{\rho}}\Big)_{x_1x_1}\nabla\varphi\cdot\nabla\triangle\psi_1  + \frac{(2\mu + \lam)\varepsilon\rho}{\tilde{\rho}\theta}\bar{v}_{1x_1x_1x_1}\varphi\triangle\psi_{1x_1} \\
& \qquad\qquad\quad + \frac{(2\mu + \lam)\varepsilon}{\theta}\bar{v}_{1x_1x_1}\nabla\varphi\cdot\nabla\triangle\psi_1 - \frac{(2\mu + \lam)\varepsilon}{\tilde{\rho}^2\theta}(\tilde{\rho} + \rho)\tilde{\rho}_{x_1}\bar{v}_{1x_1x_1}\varphi\triangle\psi_{1x_1}\\
&\qquad\qquad\quad  + \frac{\rho^2}{\theta}\Big(\frac{Q_1}{\tilde{\rho}}\Big)_{x_1}\triangle\psi_{1x_1}.
\end{align*}
Dividing the equation $\eqref{REF}_3$ by $\rho$, applying the differential operator $\nabla$ to the resulting equation and then multiplying the final equation by $-\frac{\rho^2}{\theta^2}\nabla\triangle\xi$, we get
	\begin{align}
	\begin{aligned} \label{23}
	&\Big(\frac{R}{\gamma - 1}\frac{\rho^2}{\theta^2}\frac{|\nabla^2\xi|^2}{2}\Big)_t - \Big(\frac{R}{\gamma - 1}\frac{\rho^2}{\theta^2}\nabla\xi_t\cdot\nabla\xi_{x_i}\Big)_{x_i} - \Big(\frac{R}{\gamma - 1}\frac{\rho^2}{\theta^2}v_i\nabla\xi_{x_i}\cdot\nabla\xi_{x_j}\Big)_{x_j} \\[2mm]
	&\quad + \div\Big(\frac{R}{\gamma - 1}\frac{\rho^2}{\theta^2}\bv\frac{|\nabla^2\xi|^2}{2} - R\frac{\rho^2}{\theta}\nabla\psi_{jx_i}\xi_{x_ix_j}\Big) - \Big(R\frac{\rho^2}{\theta}\nabla\div\Psi\cdot\nabla\xi_{x_i}\Big)_{x_i} \\
	&\quad + \Big(R\frac{\rho^2}{\theta}\nabla\psi_{jx_i}\cdot\nabla\xi_{x_i}\Big)_{x_j} + \frac{\kappa\varepsilon\rho}{\theta^2}|\nabla\triangle\xi|^2 + R\frac{\rho^2}{\theta}\nabla^2\xi\cdot\nabla\triangle\Psi=\sum_{i=1}^6N_i(t,x),
	\end{aligned}
	\end{align}	
where
\begin{align*}
&N_1(t,x):= -\frac{R}{\gamma - 1}\frac{2\rho}{\theta^2}\varphi_{x_i}\nabla\xi_{x_i}\cdot\nabla\xi_t - \frac{R}{\gamma - 1}\frac{2\rho}{\theta^2}\tilde{\rho}_{x_1}\nabla\xi_{x_1}\cdot\nabla\xi_t\hspace{4cm}\quad\\[3mm]
&\qquad\qquad\quad  + \frac{R}{\gamma - 1}\frac{2\rho^2}{\theta^3}\xi_{x_i}\nabla\xi_{x_i}\cdot\nabla\xi_t + \frac{R}{\gamma - 1}\frac{2\rho^2}{\theta^3}\tilde{\theta}_{x_1}\nabla\xi_{x_1}\cdot\nabla\xi_t,\\[4mm]
&
N_2(t,x):=- \frac{R}{\gamma - 1}\frac{2\rho}{\theta^2}\rho_{x_j}v_i\nabla\xi_{x_i}\cdot\nabla\xi_{x_j} + \frac{R}{\gamma - 1}\frac{2\rho^2}{\theta^3}\theta_{x_j}v_i\nabla\xi_{x_i}\cdot\nabla\xi_{x_j} \\
&\qquad\qquad\quad  - R\frac{2\rho}{\theta}\rho_{x_i}\nabla\div\Psi\cdot\nabla\xi_{x_i} + R\frac{\rho^2}{\theta^2}\theta_{x_i}\nabla\div\Psi\cdot\nabla\xi_{x_i} - \frac{R}{\gamma - 1}\frac{\rho^2}{\theta^2}\psi_{ix_j}\nabla\xi_{x_i}\cdot\nabla\xi_{x_j} \\[2mm]
&\qquad\qquad\quad - \frac{R}{\gamma - 1}\frac{\rho^2}{\theta^2}\tilde{v}_{1x_1}|\nabla\xi_{x_1}|^2 + \frac{(2\gamma - 3)R}{\gamma-1}\frac{\rho^2}{\theta^2}\div\Psi\frac{|\nabla^2\xi|^2}{2} + \frac{(2\gamma - 3)R}{\gamma-1}\frac{\rho^2}{\theta^2}\tilde{v}_{1x_1}\frac{|\nabla^2\xi|^2}{2},\\[4mm]
&
N_3(t,x):=- \frac{2\rho}{\theta^3}\frac{|\nabla^2\xi|^2}{2}\Big[\kappa\varepsilon\triangle\xi +\kappa\varepsilon\tilde{\theta}_{x_1x_1} + \frac{\mu\varepsilon}{2}|\nabla\Psi + (\nabla\Psi)^\top|^2 + \lambda\varepsilon(\div\Psi)^2 \qquad\qquad\quad\\[3mm]
& \qquad\qquad\quad + 2\tilde{v}_{1x_1}(2\mu\varepsilon\psi_{1x_1} + \lambda\varepsilon \div\Psi) + (2\mu + \lambda)\varepsilon(\tilde{v}_1\tilde{v}_{1x_1})_{x_1}\Big],\\[4mm]
& N_4(t,x):= R\frac{2\rho}{\theta}\varphi_{x_j}\nabla\psi_{jx_i}\cdot\nabla\xi_{x_i} + R\frac{2\rho}{\theta}\tilde{\rho}_{x_1}\nabla\psi_{1x_i}\cdot\nabla\xi_{x_i} \\
&\qquad\qquad\quad - R\frac{\rho^2}{\theta^2}\xi_{x_j}\nabla\psi_{jx_i}\cdot\nabla\xi_{x_i} - R\frac{\rho^2}{\theta^2}\tilde{\theta}_{x_1}\nabla\psi_{1x_i}\cdot\nabla\xi_{x_i} \\[2mm]
&\qquad\qquad\quad  - R\frac{2\rho}{\theta}\nabla\varphi\cdot\nabla\psi_{jx_i}\xi_{x_ix_j} - R\frac{2\rho}{\theta}\tilde{\rho}_{x_1}\nabla\Psi_{x_1}\cdot\nabla^2\xi + R\frac{\rho^2}{\theta^2}\nabla\xi\cdot\nabla\psi_{jx_i}\xi_{x_ix_j}\qquad\qquad\qquad\quad \\[2mm]
&\qquad\qquad\quad  + R\frac{\rho^2}{\theta^2}\tilde{\theta}_{x_1}\nabla\Psi_{x_1}\cdot\nabla^2\xi + \frac{R}{\gamma - 1}\frac{\rho^2}{\theta^2}\xi_{x_i}\nabla\psi_i\cdot\nabla\triangle\xi + \frac{R}{\gamma - 1}\frac{\rho^2}{\theta^2}\tilde{v}_{1x_1}\xi_{x_1}\triangle\xi_{x_1} \\[2mm]
&\qquad\qquad\quad  + \frac{R\rho^2}{\theta^2}\div\Psi\nabla\xi\cdot\nabla\triangle\xi + \frac{R\rho^2}{\theta^2}\tilde{\theta}_{x_1}\div\Psi\triangle\xi_{x_1} + \frac{R}{\gamma - 1}\frac{\rho^2}{\theta^2}\tilde{\theta}_{x_1x_1}\psi_1\triangle\xi_{x_1} \\[2mm]
&\qquad\qquad\quad + \frac{R}{\gamma - 1}\frac{\rho^2}{\theta^2}\tilde{\theta}_{x_1}\nabla\psi_1\cdot\nabla\triangle\xi + \frac{R\rho^2}{\theta^2}\tilde{v}_{1x_1x_1}\xi\triangle\xi_{x_1} + \frac{R\rho^2}{\theta^2}\tilde{v}_{1x_1}\nabla\xi\cdot\nabla\triangle\xi,\\[4mm]
& N_5(t,x):= \frac{\kappa\varepsilon}{\theta^2}\triangle\xi\nabla\varphi\cdot\nabla\triangle\xi + \frac{\kappa\varepsilon}{\theta^2}\tilde{\rho}_{x_1}\triangle\xi\triangle\xi_{x_1} - \frac{\mu\varepsilon\rho}{2\theta^2}\nabla(|\nabla\Psi + (\nabla\Psi)^\top|^2)\cdot\nabla\triangle\xi\qquad\qquad \\[2mm]
&\qquad\qquad\quad  + \frac{\mu\varepsilon}{2\theta^2}|\nabla\Psi + (\nabla\Psi)^\top|^2\nabla\varphi\cdot\nabla\triangle\xi + \frac{\mu\varepsilon}{2\theta^2}\tilde{\rho}_{x_1}|\nabla\Psi + (\nabla\Psi)^\top|^2\triangle\xi_{x_1} \\[2mm]
& \qquad\qquad\quad  - \frac{\lam\varepsilon\rho}{\theta^2}\nabla(\div\Psi)^2\cdot\nabla\triangle\xi + \frac{\lambda\varepsilon}{\theta^2}(\div\Psi)^2\nabla\varphi\cdot\nabla\triangle\xi +\frac{\lam\varepsilon}{\theta^2}\tilde{\rho}_{x_1}(\div\Psi)^2\triangle\xi_{x_1},
\end{align*}	
	and
$$
\begin{array}{ll}
\di N_6(t,x):=- \frac{2\rho}{\theta^2}\tilde{v}_{1x_1x_1}(2\mu\varepsilon\psi_{1x_1} + \lambda\varepsilon\div\Psi)\triangle\xi_{x_1} - \frac{2\rho}{\theta^2}\tilde{v}_{1x_1}(2\mu\varepsilon\nabla\psi_{1x_1} + \lambda\varepsilon\nabla\div\Psi)\cdot\nabla\triangle\xi
	\\[4mm]
	\di\qquad\qquad\quad  + \frac{2\tilde{v}_{1x_1}}{\theta^2}(2\mu\varepsilon\psi_{1x_1} + \lambda\varepsilon\div\Psi)\nabla\varphi\cdot\nabla\triangle\xi + \frac{2\tilde{v}_{1x_1}\tilde{\rho}_{x_1}}{\theta^2}(2\mu\varepsilon\psi_{1x_1} + \lambda\varepsilon\div\Psi)\triangle\xi_{x_1} \qquad\qquad\qquad\\[4mm]
	\di\qquad\qquad\quad - \frac{\rho}{\theta^2}\nabla{F_1}\cdot\nabla\triangle\xi + \frac{F_1}{\theta^2}\nabla\varphi\cdot\nabla\triangle\xi + \frac{F_1}{\theta^2}\tilde{\rho}_{x_1}\triangle\xi_{x_1} + \frac{\rho^2}{\theta^2}\nabla\Big(\frac{Q_2}{\tilde{\rho}}\Big)\cdot\nabla\triangle\xi \\[4mm]
	\di\qquad\qquad\quad  - \frac{\rho}{\theta^2}\nabla{F_2}\cdot\nabla\triangle\xi + \frac{F_2}{\theta^2}\nabla\varphi\cdot\nabla\triangle\xi + \frac{F_2}{\theta^2}\tilde{\rho}_{x_1}\triangle\xi_{x_1}.
\end{array}
$$
We now add \eqref{21}, \eqref{22} and \eqref{23} together, then integrate the resulting equation over $[0, ~t]\times\Omega$ to get
	\begin{align}
	\begin{aligned} \label{24}
	&\| (\nabla^2\varphi, \nabla^2\Psi, \nabla^2\xi)(t) \|^2 + \int_{0}^{t} \Big[\|\bar{v}_{1x_1}^{1/2} \nabla^2\varphi\|^2 + \varepsilon\|(\nabla^3\Psi, \nabla^3\xi)\|^2 \Big]dt \\
	&\leq C \| (\nabla^2\varphi_0, \nabla^2\Psi_0, \nabla^2\xi_0) \|^2 + C \Big|\int_{0}^{t}\int_{\Omega} L(t, x) + \sum_{i=1}^6\big(M_i(t, x) + N_i(t, x)\big) dxdt\Big|,
	\end{aligned}
	\end{align}
	where $L, M_i$ and $N_i~( i=1,2,\cdots 6)$ are defined in \eqref{22}, \eqref{21} and \eqref{23}, respectively.
	
Now we   compute some terms on the right hand side of \eqref{24}. 
First, we estimate the term $\int_{0}^{t}\int_{\Omega} L(t, x)dxdt$.  Using H\"{o}lder's inequality, Sobolev's inequality and Young's inequality, one has
	\begin{align}
	\begin{aligned} \label{25}
	&C\Big|\int_{0}^{t}\int_{\Omega} R\div\Psi\frac{|\nabla^2\varphi|^2}{2} dxdt\Big|
	\leq C \int_{0}^{t} \|\nabla\Psi\|_{L^\infty} \|\nabla^2\varphi\|^2 dt \\
	&\leq C \int_{0}^{t} (\|\nabla\Psi\|^{1/2}\|\nabla^2\Psi\|^{1/2} + \|\nabla^2\Psi\|^{1/2}\|\nabla^3\Psi\|^{1/2}) \|\nabla^2\varphi\|^2 dt \\
	&\leq C \varepsilon^{\frac 12}|\ln\varepsilon|^{-1}\int_{0}^{t} \|\nabla^2\varphi\|^2 dt + \frac{\varepsilon}{160}\int_{0}^{t} \|\nabla^3\Psi\|^2 dt + C\varepsilon^{-\frac 13}\int_{0}^{t} \|\nabla^2\Psi\|^{2/3}\|\nabla^2\varphi\|^{8/3} dt \\
	&\leq C_T \varepsilon^{\frac 12}|\ln\varepsilon|^{-1} \sup_{0 \leq t \leq t_1(\varepsilon)}\|\nabla^2\varphi\|^2 + \frac{\varepsilon}{160}\int_{0}^{t} \|\nabla^3\Psi\|^2 dt + C_T\varepsilon^{\frac{4}{3}a_2 - \frac 13}|\ln\varepsilon|^{-\frac 43}\sup_{0 \leq t \leq t_1(\varepsilon)}\|\nabla^2\varphi\|^2.
	\end{aligned}
	\end{align}
Note that the last term on the last inequality of \eqref{25} is crucial to determine the index $a_2=\frac14$ in the {\it a priori} assumptions \eqref{PA}.
	It follows from Lemma \ref{lemma2.2} and Lemma \ref{lemma2.3} that
	\begin{align*}
	C\Big|\int_{0}^{t}\int_{\Omega} R\Big(\frac{-\bar{v}_1z_1 + z_2}{\tilde{\rho}}\Big)_{x_1}\frac{|\nabla^2\varphi|^2}{2} dxdt\Big|
	\leq C_T\varepsilon^{\frac{7}{12}}|\ln\varepsilon|^{-\frac{5}{2}}\sup_{0 \leq t \leq t_1(\varepsilon)}\|\nabla^2\varphi\|^2.
	\end{align*}
	From H\"{o}lder's inequality, Sobolev's inequality and Young's inequality, we obtain
	\begin{align*}
	&C\Big|\int_{0}^{t}\int_{\Omega} 2R\varphi_{x_i}\nabla\varphi_{x_i}\cdot\nabla\div\Psi dxdt\Big|
	\leq C \int_{0}^{t} \|\nabla^2\varphi\|\|\nabla\varphi\|_{L^4}\|\nabla^2\Psi\|_{L^4} dt \\
	& \leq C \int_{0}^{t} \|\nabla^2\varphi\|\|\nabla\varphi\|^{1/4}\|\nabla\varphi\|_1^{3/4}\|\nabla^2\Psi\|^{1/4}\|\nabla^2\Psi\|_1^{3/4} dt \\
	& \leq \frac{\varepsilon}{160}\int_{0}^{t} \|\nabla^2\Psi\|_1^2 dt + C\varepsilon^{-\frac 35} \int_{0}^{t} \|\nabla^2\varphi\|^{8/5} \|\nabla\varphi\|^{2/5}\|\nabla\varphi\|_1^{6/5}\|\nabla^2\Psi\|^{2/5} dt \qquad\qquad\\
	& \leq \frac{\varepsilon}{160}\int_{0}^{t} \|\nabla^2\Psi\|_1^2 dt + C_T(\varepsilon^{\frac 35} + 1)|\ln\varepsilon|^{-\frac 85}\sup_{0 \leq t \leq t_1(\varepsilon)}\|(\nabla^2\varphi, \nabla^2\Psi)\|^2.
	\end{align*}
	By H\"{o}lder's inequality, Young's inequality, Lemma \ref{lemma2.2} and Lemma \ref{lemma2.3}, it holds that
	\begin{align*}
	&C\Big|\int_{0}^{t}\int_{\Omega} R\tilde{\rho}_{x_1x_1}\div\Psi\varphi_{x_1x_1} dxdt\Big| \\
	& \leq C\int_{0}^{t}\int_{\Omega} |\bar{\rho}_{x_1x_1}\div\Psi\varphi_{x_1x_1}| dxdt + C\int_{0}^{t}\int_{\Omega} |z_{1x_1x_1}\div\Psi\varphi_{x_1x_1}| dxdt \\
	& \leq C\frac{1}{\delta^{3/2}}\int_{0}^{t} \|\nabla\Psi\|\|\bar{v}_{1x_1}^{1/2}\varphi_{x_1x_1}\| dt + C\int_{0}^{t} \|z_{1x_1x_1}\|_{L^\infty}\|\nabla\Psi\|\|\varphi_{x_1x_1}\| dt \\
	&\leq \frac{1}{160}\int_{0}^{t} \|\bar{v}_{1x_1}^{1/2}\varphi_{x_1x_1}\|^2 dt + \frac{1}{160}\sup_{0 \leq t \leq t_1(\varepsilon)}\|\varphi_{x_1x_1}\|^2 + C_T\Big(\frac{1}{\delta^3} + \frac{\varepsilon^2}{\delta^7}\Big)\sup_{0 \leq t \leq t_1(\varepsilon)} \|\nabla\Psi\|^2 \\
	&\leq \frac{1}{160}\int_{0}^{t} \|\bar{v}_{1x_1}^{1/2}\varphi_{x_1x_1}\|^2 dt + \frac{1}{160}\sup_{0 \leq t \leq t_1(\varepsilon)}\|\varphi_{x_1x_1}\|^2 + C_T\varepsilon^{\frac{7}{6}}|\ln\varepsilon|^{-11}.
	\end{align*}
	Similarly, one has
	\begin{align}\label{dr}
	&C\Big|\int_{0}^{t}\int_{\Omega} 2R\tilde{\rho}_{x_1}\nabla\varphi_{x_1}\cdot\nabla\div\Psi dxdt\Big|\nonumber \\
	& \leq C\int_{0}^{t}\int_{\Omega} |\bar{\rho}_{x_1}\nabla\varphi_{x_1}\cdot\nabla\div\Psi| dxdt + C\int_{0}^{t}\int_{\Omega} |z_{1x_1}\nabla\varphi_{x_1}\cdot\nabla\div\Psi| dxdt \nonumber \\
	& \leq C\frac{1}{\delta^{1/2}}\int_{0}^{t} \|\bar{v}_{1x_1}^{1/2}\nabla\varphi_{x_1}\|\|\nabla^2\Psi\| dt + C\int_{0}^{t} \|z_{1x_1}\|_{L^\infty}\|\nabla\varphi_{x_1}\|\|\nabla^2\Psi\| dt \\
	&\leq \frac{1}{160}\int_{0}^{t} \|\bar{v}_{1x_1}^{1/2}\nabla\varphi_{x_1}\|^2 dt + \frac{1}{160}\sup_{0 \leq t \leq t_1(\varepsilon)}\|\nabla\varphi_{x_1}\|^2 + C_T\Big(\frac{1}{\delta} + \frac{\varepsilon^2}{\delta^5}\Big) \int_{0}^{t} \|\nabla^2\Psi\|^2 dt\nonumber \\
	&\leq \frac{1}{160}\int_{0}^{t} \|\bar{v}_{1x_1}^{1/2}\nabla\varphi_{x_1}\|^2 dt + \frac{1}{160}\sup_{0 \leq t \leq t_1(\varepsilon)}\|\nabla\varphi_{x_1}\|^2 + C_T \frac{\varepsilon^2}{\delta^9} \nonumber \\
	&\leq \frac{1}{160}\int_{0}^{t} \|\bar{v}_{1x_1}^{1/2}\nabla\varphi_{x_1}\|^2 dt + \frac{1}{160}\sup_{0 \leq t \leq t_1(\varepsilon)}\|\nabla\varphi_{x_1}\|^2 + C_T\varepsilon^{\frac{1}{2}}|\ln\varepsilon|^{-9}, \nonumber 
	\end{align}
	and
	\begin{align*}
	&C\Big|\int_{0}^{t}\int_{\Omega} R\tilde{\rho}_{x_1x_1x_1}\psi_1\varphi_{x_1x_1} dxdt\Big| \\
	& \leq C\int_{0}^{t}\int_{\Omega} |\bar{\rho}_{x_1x_1x_1}\psi_1\varphi_{x_1x_1}| dxdt + C\int_{0}^{t}\int_{\Omega} |z_{1x_1x_1x_1}\psi_1\varphi_{x_1x_1}| dxdt \\
	& \leq C\frac{1}{\delta^{5/2}}\int_{0}^{t} \|\psi_1\|\|\bar{v}_{1x_1}^{1/2}\varphi_{x_1x_1}\| dt + C\int_{0}^{t} \|z_{1x_1x_1x_1}\|_{L^\infty}\|\psi_1\|\|\varphi_{x_1x_1}\| dt \\
	&\leq \frac{1}{160}\int_{0}^{t} \|\bar{v}_{1x_1}^{1/2}\varphi_{x_1x_1}\|^2 dt + \frac{1}{160}\sup_{0 \leq t \leq t_1(\varepsilon)}\|\varphi_{x_1x_1}\|^2 + C_T\Big(\frac{1}{\delta^5} + \frac{\varepsilon^2}{\delta^9}\Big)\sup_{0 \leq t \leq t_1(\varepsilon)} \|\psi_1\|^2 \\
	&\leq \frac{1}{160}\int_{0}^{t} \|\bar{v}_{1x_1}^{1/2}\varphi_{x_1x_1}\|^2 dt + \frac{1}{160}\sup_{0 \leq t \leq t_1(\varepsilon)}\|\varphi_{x_1x_1}\|^2 + C_T\varepsilon^{2}|\ln\varepsilon|^{-12}.
	\end{align*}
	The other terms in $L(t,x)$ can be estimated similarly. 
	
	Now we handle $\di  \int_{0}^{t}\int_{\Omega} \sum_{i=1}^6M_i(t, x)dxdt$ as follows. From Lemma \ref{lemma2.2} and Lemma \ref{lemma2.3}, it follows that	
	\begin{align}\label{dr4}
	C\Big|\int_{0}^{t}\int_{\Omega} \frac{\rho^2}{\theta}\tilde{v}_{1x_1}|\nabla\Psi_{x_1}|^2 dxdt\Big|
	\leq C\int_{0}^{t} \|\tilde{v}_{1x_1}\|_{L^\infty}\|\nabla\Psi_{x_1}\|^2 dt
	\leq C_T\frac{\varepsilon^2}{\delta^9}
	\leq C_T\varepsilon^{\frac{1}{2}}|\ln\varepsilon|^{-9}.
	\end{align}
	By H\"{o}lder's inequality, Sobolev's inequality and Young's inequality, we have
	\begin{align*}
	&C\Big|\int_{0}^{t}\int_{\Omega} \frac{\gamma - 1}{R}\kappa\varepsilon\frac{\rho}{\theta^2}\frac{|\nabla^2\Psi|^2}{2}\triangle\xi dxdt\Big|
	\leq C\varepsilon \int_{0}^{t} \|\nabla^2\xi\|\|\nabla^2\Psi\|_{L^4}^2 dt \\
	&\leq C\varepsilon \int_{0}^{t} \|\nabla^2\xi\|\|\nabla^2\Psi\|^{1/2}\|\nabla^2\Psi\|_1^{3/2} dt\\
	&
	\leq \frac{\varepsilon}{160}\int_{0}^{t} \|\nabla^2\Psi\|_1^2 dt + C\varepsilon \int_{0}^{t} \|\nabla^2\xi\|^4\|\nabla^2\Psi\|^2 dt \\
	& \leq \frac{\varepsilon}{160}\int_{0}^{t} \|\nabla^2\Psi\|_1^2 dt + C_T\varepsilon^{\frac{8}{3}}|\ln\varepsilon|^{-12}.
	\end{align*}	
We can deduce from Lemma \ref{lemma2.2} and Lemma \ref{lemma2.3} that
	\begin{align*}
	C\Big|\int_{0}^{t}\int_{\Omega} \frac{\gamma - 1}{R}\kappa\varepsilon\frac{\rho}{\theta^2}\tilde{\theta}_{x_1x_1}\frac{|\nabla^2\Psi|^2}{2} dxdt\Big|
	\leq C\varepsilon \int_{0}^{t} \|\tilde{\theta}_{x_1x_1}\|_{L^\infty}\|\nabla^2\Psi\|^2 dt
	 \leq C_T\varepsilon^{\frac{4}{3}}|\ln\varepsilon|^{-10}.
	\end{align*}	
	By Sobolev's inequality, H\"{o}lder's inequality and Young's inequality, it can be derived that
	\begin{align*}
	&C\Big|\int_{0}^{t}\int_{\Omega} \frac{\gamma - 1}{R}\lambda\varepsilon\frac{\rho}{\theta^2}(\div\Psi)^2\frac{|\nabla^2\Psi|^2}{2} dxdt\Big|
	\leq C\varepsilon \int_{0}^{t} \|\nabla\Psi\|_{L^4}^2\|\nabla^2\Psi\|_{L^4}^2 dt \\
	& \leq C\varepsilon \int_{0}^{t} \|\nabla\Psi\|^{1/2}\|\nabla\Psi\|_1^{3/2}\|\nabla^2\Psi\|^{1/2}\|\nabla^2\Psi\|_1^{3/2} dt \\
	& \leq \frac{\varepsilon}{160}\int_{0}^{t} \|\nabla^2\Psi\|_1^2 dt + C\varepsilon \int_{0}^{t} \|\nabla\Psi\|^2\|\nabla\Psi\|_1^6\|\nabla^2\Psi\|^2 dt \\
	& \leq \frac{\varepsilon}{160}\int_{0}^{t} \|\nabla^2\Psi\|_1^2 dt + C_T\varepsilon^{\frac{14}{3}}|\ln\varepsilon|^{-16},
	\end{align*}
	and
	\begin{align*}
	&C\Big|\int_{0}^{t}\int_{\Omega} R\varphi_{x_i}\nabla\varphi\cdot\nabla\triangle\psi_i dxdt\Big|
	 \leq \frac{\varepsilon}{160}\int_{0}^{t} \|\nabla^3\Psi\|^2 dt + C\varepsilon^{-1} \int_{0}^{t} \|\nabla\varphi\|_{L^4}^4 dt \\
	& \leq \frac{\varepsilon}{160}\int_{0}^{t} \|\nabla^3\Psi\|^2 dt + C\varepsilon^{-1} \int_{0}^{t} \|\nabla\varphi\|\|\nabla\varphi\|_1^3 dt \\
	& \leq \frac{\varepsilon}{160}\int_{0}^{t} \|\nabla^3\Psi\|^2 dt + C_T\varepsilon^{\frac{1}{12}}|\ln\varepsilon|^{-5}\sup_{0 \leq t \leq t_1(\varepsilon)}\|\nabla\varphi\|_1^2.
	\end{align*}
	By Young's inequality, Lemma \ref{lemma2.2} and Lemma \ref{lemma2.3}, one can get
	\begin{align}\label{dr1}
	&C\Big|\int_{0}^{t}\int_{\Omega} \frac{\rho^2}{\theta}\tilde{v}_{1x_1}\Psi_{x_1}\cdot\triangle\Psi_{x_1} dxdt\Big| \nonumber\\
	& \leq \frac{\varepsilon}{160} \int_{0}^{t} \|\nabla^3\Psi\|^2 dt + C\varepsilon^{-1}\int_{0}^{t} \|\tilde{v}_{1x_1}\|_{L^\infty}^2\|\Psi_{x_1}\|^2 dt \qquad\qquad\qquad \qquad \qquad \quad \\
	& \leq \frac{\varepsilon}{160} \int_{0}^{t} \|\nabla^3\Psi\|^2 dt + C_T\frac{\varepsilon^2}{\delta^9}\nonumber \\
	& \leq \frac{\varepsilon}{160} \int_{0}^{t} \|\nabla^3\Psi\|^2 dt + C_T \varepsilon^{\frac{1}{2}}|\ln\varepsilon|^{-9},\nonumber 
	\end{align}
	\begin{align}\label{dr3}
	&C\Big|\int_{0}^{t}\int_{\Omega} R\tilde{\rho}_{x_1}\nabla\varphi\cdot\triangle\Psi_{x_1} dxdt\Big| \nonumber\\
	& \leq \frac{\varepsilon}{160} \int_{0}^{t} \|\nabla^2\Psi_{x_1}\|^2 dt + C\varepsilon^{-1}\int_{0}^{t} (\|\bar{\rho}_{x_1}\nabla\varphi\|^2 + \|z_{1x_1}\nabla\varphi\|^2) dt \nonumber \\ 
	& \leq \frac{\varepsilon}{160} \int_{0}^{t} \|\nabla^2\Psi_{x_1}\|^2 dt + C\varepsilon^{-1}\frac{1}{\delta}\int_{0}^{t} \|\bar{v}_{1x_1}^{1/2}\nabla\varphi\|^2 dt + C_T\varepsilon^{-1}\frac{\varepsilon^2}{\delta^5} \sup_{0 \leq t \leq t_1(\varepsilon)}\|\nabla\varphi\|^2 \\
	& \leq \frac{\varepsilon}{160} \int_{0}^{t} \|\nabla^2\Psi_{x_1}\|^2 dt + C_T\frac{\varepsilon^2}{\delta^9} \nonumber \\ 
	& \leq \frac{\varepsilon}{160} \int_{0}^{t} \|\nabla^2\Psi_{x_1}\|^2 dt + C_T\varepsilon^{\frac{1}{2}}|\ln\varepsilon|^{-9},\nonumber
	\end{align}
	and
	\begin{align*}
	&C\Big|\int_{0}^{t}\int_{\Omega} \frac{\rho^2}{\theta}\tilde{v}_{1x_1x_1}\psi_1\triangle\psi_{1x_1} dxdt\Big| \\
	& \leq \frac{\varepsilon}{160} \int_{0}^{t} \|\nabla^2\psi_{1x_1}\|^2 dt + C\varepsilon^{-1}\int_{0}^{t} \Big(\|\bar{v}_{1x_1x_1}\psi_1\|^2 + \Big\|\Big(\frac{-\bar{v}_1z_1 + z_2}{\tilde{\rho}}\Big)_{x_1x_1}\psi_1\Big\|^2\Big) dt \\
	& \leq \frac{\varepsilon}{160} \int_{0}^{t} \|\nabla^2\psi_{1x_1}\|^2 dt + C\varepsilon^{-1}\frac{1}{\delta^3}\int_{0}^{t} \|\bar{v}_{1x_1}^{1/2}\psi_1\|^2 dt + C_T\varepsilon^{-1}\frac{\varepsilon^2}{\delta^7} \sup_{0 \leq t \leq t_1(\varepsilon)}\|\psi_1\|^2 \\
	& \leq \frac{\varepsilon}{160} \int_{0}^{t} \|\nabla^2\psi_{1x_1}\|^2 dt + C_T\varepsilon^{\frac{4}{3}}|\ln\varepsilon|^{-10}.
	\end{align*}
It follows from H\"{o}lder's inequality, Sobolev's inequality and Young's inequality that
	\begin{align*}
	&C\Big|\int_{0}^{t}\int_{\Omega} \frac{\mu\varepsilon}{\theta}\triangle\psi_i\nabla\varphi\cdot\nabla\triangle\psi_i dxdt\Big|
	\leq C\varepsilon \int_{0}^{t} \|\nabla^2\Psi\|_{L^4}\|\nabla\varphi\|_{L^4}\|\nabla^3\Psi\| dt \\
	& \leq C\varepsilon \int_{0}^{t} \|\nabla^2\Psi\|^{1/4}\|\nabla^2\Psi\|_1^{3/4}\|\nabla\varphi\|^{1/4}\|\nabla\varphi\|_1^{3/4}\|\nabla^3\Psi\| dt \\
	& \leq C\varepsilon \int_{0}^{t} \|\nabla^2\Psi\|^{1/4}\|\nabla\varphi\|^{1/4}\|\nabla\varphi\|_1^{3/4}\|\nabla^2\Psi\|_1^{7/4} dt \\
	& \leq \frac{\varepsilon}{160}\int_{0}^{t} \|\nabla^2\Psi\|_1^2 dt + C\varepsilon \int_{0}^{t} \|\nabla^2\Psi\|^2\|\nabla\varphi\|^2\|\nabla\varphi\|_1^6 dt \\
	& \leq \frac{\varepsilon}{160}\int_{0}^{t} \|\nabla^2\Psi\|_1^2 dt + C_T\varepsilon^{\frac{14}{3}}|\ln\varepsilon|^{-16}.
	\end{align*}
Then by Young's inequality, Lemma \ref{lemma2.2} and Lemma \ref{lemma2.3}, we infer that
	\begin{align*}
	&C\Big|\int_{0}^{t}\int_{\Omega} \frac{\mu\varepsilon}{\theta}\tilde{\rho}_{x_1}\triangle\Psi\cdot\triangle\Psi_{x_1} dxdt\Big| \\
	&
	\leq \frac{\varepsilon}{160} \int_{0}^{t} \|\nabla^2\Psi_{x_1}\|^2 dt + C\varepsilon \int_{0}^{t} \|\tilde{\rho}_{x_1}\|_{L^\infty}^2\|\nabla^2\Psi\|^2 dt \qquad\qquad\qquad \\
	& \leq \frac{\varepsilon}{160} \int_{0}^{t} \|\nabla^2\Psi_{x_1}\|^2 dt + C_T \varepsilon^{\frac{4}{3}}|\ln\varepsilon|^{-10},
	\end{align*}
	\begin{align*}
	&C\Big|\int_{0}^{t}\int_{\Omega} (2\mu + \lam)\varepsilon\Big(\frac{-\bar{v}_1z_1 + z_2}{\tilde{\rho}}\Big)_{x_1x_1x_1}\triangle\psi_{1x_1} dxdt\Big| \\
	& \leq \frac{\varepsilon}{160} \int_{0}^{t} \|\nabla^2\psi_{1x_1}\|^2 dt + C\varepsilon\int_{0}^{t}\int_{\bbr} \Big|\Big(\frac{-\bar{v}_1z_1 + z_2}{\tilde{\rho}}\Big)_{x_1x_1x_1}\Big|^2 dx_1dt \\
	& \leq \frac{\varepsilon}{160} \int_{0}^{t} \|\nabla^2\psi_{1x_1}\|^2 dt + C_T\varepsilon^{\frac{5}{3}}|\ln\varepsilon|^{-8},
	\end{align*}
	and
	\begin{align*}
	&C\Big|\int_{0}^{t}\int_{\Omega} \frac{(2\mu + \lam)\varepsilon}{\theta}\Big(\frac{-\bar{v}_1z_1 + z_2}{\tilde{\rho}}\Big)_{x_1x_1}\nabla\varphi\cdot\nabla\triangle\psi_1 dxdt\Big| \\
	& \leq \frac{\varepsilon}{160} \int_{0}^{t} \|\nabla^3\psi_1\|^2 dt + C\varepsilon \int_{0}^{t} \Big\|\Big(\frac{-\bar{v}_1z_1 + z_2}{\tilde{\rho}}\Big)_{x_1x_1}\Big\|_{L^\infty}^2\|\nabla\varphi\|^2 dt\\
	&
	\leq \frac{\varepsilon}{160} \int_{0}^{t} \|\nabla^3\psi_1\|^2 dt + C_T \varepsilon^{\frac{7}{2}}|\ln\varepsilon|^{-15}.
	\end{align*}
	On the other hand, it follows that
	\begin{align*}
	&C\Big|\int_{0}^{t}\int_{\Omega} \frac{(2\mu + \lam)\varepsilon\rho}{\tilde{\rho}\theta}\bar{v}_{1x_1x_1x_1}\varphi\triangle\psi_{1x_1} dxdt\Big| \\
	&
	\leq \frac{\varepsilon}{160} \int_{0}^{t} \|\nabla^2\psi_{1x_1}\|^2 dt + C\varepsilon \int_{0}^{t} \|\bar{v}_{1x_1x_1x_1}\varphi\|^2 dt  \qquad\qquad \qquad\qquad\qquad \qquad\\
	& \leq \frac{\varepsilon}{160} \int_{0}^{t} \|\nabla^2\psi_{1x_1}\|^2 dt + C\frac{\varepsilon}{\delta^5} \int_{0}^{t} \|\bar{v}_{1x_1}^{1/2}\varphi\|^2 dt\\
	&
	\leq \frac{\varepsilon}{160} \int_{0}^{t} \|\nabla^2\psi_{1x_1}\|^2 dt + C_T \varepsilon^{3}|\ln\varepsilon|^{-12},
	\end{align*}
	\begin{align*}
	&C\Big|\int_{0}^{t}\int_{\Omega} \frac{(2\mu + \lam)\varepsilon}{\theta}\bar{v}_{1x_1x_1}\nabla\varphi\cdot\nabla\triangle\psi_1 dxdt\Big| \\
	&
	\leq \frac{\varepsilon}{160} \int_{0}^{t} \|\nabla^3\psi_1\|^2 dt + C\varepsilon \int_{0}^{t} \|\bar{v}_{1x_1x_1}\nabla\varphi\|^2 dt \qquad\qquad \qquad\qquad\qquad\qquad \\
	& \leq \frac{\varepsilon}{160} \int_{0}^{t} \|\nabla^3\psi_1\|^2 dt + C\frac{\varepsilon}{\delta^3} \int_{0}^{t} \|\bar{v}_{1x_1}^{1/2}\nabla\varphi\|^2 dt\\
	&
	\leq \frac{\varepsilon}{160} \int_{0}^{t} \|\nabla^3\psi_1\|^2 dt + C_T \varepsilon^{\frac{13}{6}}|\ln\varepsilon|^{-11},
	\end{align*}
	and
	\begin{align*}
	&C\Big|\int_{0}^{t}\int_{\Omega} \frac{(2\mu + \lam)\varepsilon}{\tilde{\rho}^2\theta}(\tilde{\rho} + \rho)\tilde{\rho}_{x_1}\bar{v}_{1x_1x_1}\varphi\triangle\psi_{1x_1} dxdt\Big| \\
	&\leq \frac{\varepsilon}{160} \int_{0}^{t} \|\nabla^2\psi_{1x_1}\|^2 dt + C\varepsilon \int_{0}^{t} \|\tilde{\rho}_{x_1}\bar{v}_{1x_1x_1}\varphi\|^2 dt \\
	& \leq \frac{\varepsilon}{160} \int_{0}^{t} \|\nabla^2\psi_{1x_1}\|^2 dt + C_T\Big(\frac{\varepsilon}{\delta^5} + \frac{\varepsilon^3}{\delta^8}\Big) \int_{0}^{t} \|\bar{v}_{1x_1}^{1/2}\varphi\|^2 dt \qquad\qquad \qquad\qquad\\
	&\leq \frac{\varepsilon}{160} \int_{0}^{t} \|\nabla^2\psi_{1x_1}\|^2 dt + C_T \varepsilon^{3}|\ln\varepsilon|^{-12}.
	\end{align*}
	Also, we know that
	\begin{align*}
	&C\Big|\int_{0}^{t}\int_{\Omega} \frac{\rho^2}{\theta}\Big(\frac{Q_{1}}{\tilde{\rho}}\Big)_{x_1}\triangle\psi_{1x_1} dxdt\Big| \\
	& \leq \frac{\varepsilon}{160} \int_{0}^{t} \|\nabla^2\psi_{1x_1}\|^2 dt + C\varepsilon^{-1}\int_{0}^{t}\int_{\bbr} \Big|\Big(\frac{1}{\tilde{\rho}}\Big(\frac{1}{\tilde{\rho}}(-\bar{v}_1z_1 + z_2)^2\Big)_{x_1}\Big)_{x_1}\Big|^2 dx_1dt \\
	& \leq \frac{\varepsilon}{160} \int_{0}^{t} \|\nabla^2\psi_{1x_1}\|^2 dt + C_T\varepsilon^{\frac{3}{2}}|\ln\varepsilon|^{-9},
	\end{align*}
	\begin{align*}
	&C\Big|\int_{0}^{t}\int_{\Omega} \frac{2\rho}{\theta}\tilde{v}_{1x_1x_1}\psi_1\varphi_{x_i}\psi_{1x_ix_1} dxdt\Big| \\
	&\leq C \int_{0}^{t} \|\nabla\psi_{1x_1}\|^2 dt + C \int_{0}^{t} \Big(\|\bar{v}_{1x_1x_1}\nabla\varphi\|^2 + \Big\|\Big(\frac{-\bar{v}_1z_1 + z_2}{\tilde{\rho}}\Big)_{x_1x_1}\nabla\varphi\Big\|^2\Big) dt \\
	&\leq C_T\varepsilon^{\frac{2}{3}}|\ln\varepsilon|^{-8} + C_T \varepsilon^{\frac{7}{6}}|\ln\varepsilon|^{-11},
	\end{align*}
	and
	\begin{align*}
	&C\Big|\int_{0}^{t}\int_{\Omega} \frac{2(2\mu + \lam)\varepsilon}{\theta}\Big(\frac{-\bar{v}_1z_1 + z_2}{\tilde{\rho}}\Big)_{x_1x_1x_1}\nabla\varphi\cdot\nabla\psi_{1x_1} dxdt\Big| \\
	& \leq C\varepsilon \int_{0}^{t} \|\nabla\psi_{1x_1}\|^2 dt + C\varepsilon \int_{0}^{t} \Big\|\Big(\frac{-\bar{v}_1z_1 + z_2}{\tilde{\rho}}\Big)_{x_1x_1x_1}\Big\|_{L^\infty}^2\|\nabla\varphi\|^2 dt \\
	&\leq C_T \varepsilon^{\frac{5}{3}}|\ln\varepsilon|^{-8} + C_T \varepsilon^{\frac{19}{6}}|\ln\varepsilon|^{-17}.
	\end{align*}
It follows from H\"{o}lder's inequality, Sobolev's inequality, Young's inequality, Lemma \ref{lemma2.2} and Lemma \ref{lemma2.3} that
	\begin{align*}
	&C\Big|\int_{0}^{t}\int_{\Omega} \frac{2(2\mu + \lam)\varepsilon}{\rho\theta}\Big(\frac{-\bar{v}_1z_1 + z_2}{\tilde{\rho}}\Big)_{x_1x_1}\varphi_{x_i}\nabla\varphi\cdot\nabla\psi_{1x_i} dxdt\Big| \qquad\\
	& \leq C\varepsilon \int_{0}^{t} \Big\|\Big(\frac{-\bar{v}_1z_1 + z_2}{\tilde{\rho}}\Big)_{x_1x_1}\Big\|_{L^\infty}\|\nabla\varphi\|_{L^4}^2\|\nabla^2\psi_1\| dt \\
	&\leq C_T \frac{\varepsilon^2}{\delta^{7/2}} \int_{0}^{t} \|\nabla\varphi\|^{1/2}\|\nabla\varphi\|_1^{3/2}\|\nabla^2\psi_1\| dt\\
	&
	\leq \frac{1}{160} \sup_{0 \leq t \leq t_1(\varepsilon)}\|\nabla\varphi\|_1^2 + C_T \frac{\varepsilon^8}{\delta^{14}} \int_{0}^{t} \|\nabla\varphi\|^2\|\nabla^2\psi_1\|^4 dt \\
	&\leq \frac{1}{160} \sup_{0 \leq t \leq t_1(\varepsilon)}\|\nabla\varphi\|_1^2 + C_T \frac{\varepsilon^{10}}{\delta^{14}}|\ln\varepsilon|^{-4} \int_{0}^{t} \|\nabla^2\psi_1\|^2 dt \\
	&\leq \frac{1}{160} \sup_{0 \leq t \leq t_1(\varepsilon)}\|\nabla\varphi\|_1^2 + C_T \varepsilon^{\frac{25}{3}}|\ln\varepsilon|^{-26},
	\end{align*}
	\begin{align*}
	&C\Big|\int_{0}^{t}\int_{\Omega} \frac{2(2\mu + \lam)\varepsilon}{\tilde{\rho}\theta}\bar{v}_{1x_1x_1x_1}\varphi\nabla\varphi\cdot\nabla\psi_{1x_1} dxdt\Big| \\
	& \leq C\varepsilon \int_{0}^{t} \|\nabla\psi_{1x_1}\|^2 dt + C\varepsilon \int_{0}^{t} \|\bar{v}_{1x_1x_1x_1}\|_{L^\infty}^2\|\nabla\varphi\|^2 dt\qquad\qquad\\
	&
	\leq C_T \varepsilon^{\frac{5}{3}}|\ln\varepsilon|^{-8} + C_T \varepsilon^{\frac{5}{3}}|\ln\varepsilon|^{-14},
	\end{align*}
	and
	\begin{align*}
	&C\Big|\int_{0}^{t}\int_{\Omega} \frac{2(2\mu + \lam)\varepsilon}{\rho\theta}\bar{v}_{1x_1x_1}\varphi_{x_i}\nabla\varphi\cdot\nabla\psi_{1x_i} dxdt\Big| \\
&	 \leq C\varepsilon \int_{0}^{t} \|\bar{v}_{1x_1x_1}\|_{L^\infty}\|\nabla\varphi\|_{L^4}^2\|\nabla^2\psi_1\| dt \\
	&\leq C_T \frac{\varepsilon}{\delta^2} \int_{0}^{t} \|\nabla\varphi\|^{1/2}\|\nabla\varphi\|_1^{3/2}\|\nabla^2\psi_1\| dt
	\\&
	\leq \frac{1}{160} \sup_{0 \leq t \leq t_1(\varepsilon)}\|\nabla\varphi\|_1^2 + C_T \frac{\varepsilon^4}{\delta^{8}} \int_{0}^{t} \|\nabla\varphi\|^2\|\nabla^2\psi_1\|^4 dt \\
	&\leq \frac{1}{160} \sup_{0 \leq t \leq t_1(\varepsilon)}\|\nabla\varphi\|_1^2 + C_T \frac{\varepsilon^{6}}{\delta^{8}}|\ln\varepsilon|^{-4} \int_{0}^{t} \|\nabla^2\psi_1\|^2 dt\qquad\qquad\\
	&\leq \frac{1}{160} \sup_{0 \leq t \leq t_1(\varepsilon)}\|\nabla\varphi\|_1^2 + C_T \varepsilon^{\frac{16}{3}}|\ln\varepsilon|^{-16}.
	\end{align*}
	By using Young's inequality, Lemma \ref{lemma2.2} and Lemma \ref{lemma2.3}, we obtain	\begin{align*}
	&C\Big|\int_{0}^{t}\int_{\Omega} \frac{2\rho}{\theta}\Big(\frac{Q_1}{\tilde{\rho}}\Big)_{x_1}\nabla\varphi\cdot\nabla\psi_{1x_1} dxdt\Big| \\
	& \leq C \int_{0}^{t} \|\nabla\psi_{1x_1}\|^2 dt + C \int_{0}^{t} \Big\|\Big(\frac{1}{\tilde{\rho}}\Big(\frac{1}{\tilde{\rho}}(-\bar{v}_1z_1 + z_2)^2\Big)_{x_1}\Big)_{x_1}\Big\|_{L^\infty}^2\|\nabla\varphi\|^2 dt \\
	& \leq C_T \varepsilon^{\frac{2}{3}}|\ln\varepsilon|^{-8} + C_T \frac{\varepsilon^4}{\delta^{10}}\sup_{0 \leq t \leq t_1(\varepsilon)}\|\nabla\varphi\|^2
	\leq C_T \varepsilon^{\frac{2}{3}}|\ln\varepsilon|^{-8} + C_T \varepsilon^{4}|\ln\varepsilon|^{-18},\\[4mm]
	&C\Big|\int_{0}^{t}\int_{\Omega} \frac{2\rho}{\theta}\tilde{\rho}_{x_1}\tilde{v}_{1x_1}\Psi_{x_1}\cdot\Psi_{x_1x_1} dxdt\Big| \\
	& \leq C \int_{0}^{t} \|\Psi_{x_1x_1}\|^2 dt + C \int_{0}^{t} \|\tilde{\rho}_{x_1}\tilde{v}_{1x_1}\|_{L^\infty}^2\|\Psi_{x_1}\|^2 dt \\
	&
	 \leq C_T \varepsilon^{\frac{2}{3}}|\ln\varepsilon|^{-8} + C_T \varepsilon^{\frac{7}{6}}|\ln\varepsilon|^{-11},
	\end{align*}
	and
	\begin{align*}
	&C\Big|\int_{0}^{t}\int_{\Omega} \frac{2\rho}{\theta}\tilde{\rho}_{x_1}\tilde{v}_{1x_1x_1}\psi_1\psi_{1x_1x_1} dxdt\Big| \\
&	 \leq C \int_{0}^{t} \|\psi_{1x_1x_1}\|^2 dt + C \int_{0}^{t} \|\tilde{\rho}_{x_1}\tilde{v}_{1x_1x_1}\|_{L^\infty}^2\|\psi_1\|^2 dt \\
	&\leq C_T \varepsilon^{\frac{2}{3}}|\ln\varepsilon|^{-8} + C_T\Big(\frac{1}{\delta^2} + \frac{\varepsilon^2}{\delta^{5}}\Big)\Big(\frac{1}{\delta^4} + \frac{\varepsilon^2}{\delta^{7}}\Big) \sup_{0 \leq t \leq t_1(\varepsilon)}\|\psi_1\|^2\qquad\qquad\qquad \\
	&
	\leq C_T \varepsilon^{\frac{2}{3}}|\ln\varepsilon|^{-8} + C_T \varepsilon^{\frac{11}{6}}|\ln\varepsilon|^{-13}.
	\end{align*}
Moreover, it holds that
	\begin{align*}
	&C\Big|\int_{0}^{t}\int_{\Omega} \frac{2(2\mu + \lam)\varepsilon}{\theta}\tilde{\rho}_{x_1}\Big(\frac{-\bar{v}_1z_1 + z_2}{\tilde{\rho}}\Big)_{x_1x_1x_1}\psi_{1x_1x_1} dxdt\Big| \\
	& \leq C \int_{0}^{t} \|\psi_{1x_1x_1}\|^2 dt + C\varepsilon^2 \int_{0}^{t} \|\tilde{\rho}_{x_1}\|_{L^\infty}^2\Big\|\Big(\frac{-\bar{v}_1z_1 + z_2}{\tilde{\rho}}\Big)_{x_1x_1x_1}\Big\|^2 dt\qquad\\
	&
	\leq C_T \varepsilon^{\frac{2}{3}}|\ln\varepsilon|^{-8} + C_T \varepsilon^{\frac{7}{3}}|\ln\varepsilon|^{-10},\\[4mm]
	&C\Big|\int_{0}^{t}\int_{\Omega} \frac{2(2\mu + \lam)\varepsilon}{\tilde{\rho}\theta}\tilde{\rho}_{x_1}\bar{v}_{1x_1x_1x_1}\varphi\psi_{1x_1x_1} dxdt\Big| \\
	&\leq C \varepsilon \int_{0}^{t} \|\psi_{1x_1x_1}\|^2 dt + C \varepsilon \int_{0}^{t} \|\tilde{\rho}_{x_1}\|_{L^\infty}^2\|\bar{v}_{1x_1x_1x_1}\varphi\|^2 dt \\
	&\leq C_T\varepsilon^{\frac{5}{3}}|\ln\varepsilon|^{-8} + C_T\frac{\varepsilon}{\delta^5}\Big(\frac{1}{\delta^2} + \frac{\varepsilon^2}{\delta^{5}}\Big)\int_{0}^{t} \|\bar{v}_{1x_1}^{1/2}\varphi\|^2 dt\\
&\leq C_T \varepsilon^{\frac{5}{3}}|\ln\varepsilon|^{-8} + C_T \varepsilon^{\frac{8}{3}}|\ln\varepsilon|^{-14},
	\end{align*}
	and
	\begin{align*}
	&C\Big|\int_{0}^{t}\int_{\Omega} \frac{2\rho}{\theta}\tilde{\rho}_{x_1}\Big(\frac{Q_1}{\tilde{\rho}}\Big)_{x_1}\psi_{1x_1x_1} dxdt\Big| \\
	& \leq C \int_{0}^{t} \|\psi_{1x_1x_1}\|^2 dt + C \int_{0}^{t} \|\tilde{\rho}_{x_1}\|_{L^\infty}^2\Big\|\Big(\frac{1}{\tilde{\rho}}\Big(\frac{1}{\tilde{\rho}}(-\bar{v}_1z_1 + z_2)^2\Big)_{x_1}\Big)_{x_1}\Big\|^2 dt \qquad\qquad\qquad\\
	& \leq C_T \varepsilon^{\frac{2}{3}}|\ln\varepsilon|^{-8} + C_T \varepsilon^{\frac{13}{6}}|\ln\varepsilon|^{-11}.
	\end{align*}
	The remaining terms in $M_i(t,x) (i=1,2, \cdots 6)$ can be analyzed similarly. 
	
	
	Now we handle $\di  \int_{0}^{t}\int_{\Omega} \sum_{i=1}^6N_i(t, x)dxdt$ as follows. 	
	From H\"{o}lder's inequality, Sobolev's inequality and Young's inequality, we have
	\begin{align*}
	&C\Big|\int_{0}^{t}\int_{\Omega} \frac{\lambda\varepsilon}{\theta^2}(\div\Psi)^2\nabla\varphi\cdot\nabla\triangle\xi dxdt\Big|
	\leq C\varepsilon \int_{0}^{t} \|\nabla\Psi\|_{L^6}^2\|\nabla\varphi\|_{L^6}\|\nabla^3\xi\| dt \\
	& \leq C\varepsilon \int_{0}^{t} \|\nabla\Psi\|_1^2\|\nabla\varphi\|_1\|\nabla^3\xi\| dt
	 \leq \frac{\varepsilon}{160}\int_{0}^{t} \|\nabla^3\xi\|^2 dt + C\varepsilon \int_{0}^{t} \|\nabla\Psi\|_1^4\|\nabla\varphi\|_1^2 dt \qquad\quad\\
	& \leq \frac{\varepsilon}{160}\int_{0}^{t} \|\nabla^3\xi\|^2 dt + C_T(\varepsilon^4 + \varepsilon^2)|\ln\varepsilon|^{-4}\sup_{0 \leq t \leq t_1(\varepsilon)}\|\nabla\varphi\|_1^2.
	\end{align*}
	By Lemmas \ref{lemma2.2} and \ref{lemma2.3} and Young's inequality, it holds that
	\begin{align*}
	&C\Big|\int_{0}^{t}\int_{\Omega} \frac{2\rho}{\theta^2}\tilde{v}_{1x_1x_1}(2\mu\varepsilon\psi_{1x_1} + \lambda\varepsilon\div\Psi)\triangle\xi_{x_1} dxdt\Big| \\
	&\leq \frac{\varepsilon}{160} \int_{0}^{t} \|\nabla^2\xi_{x_1}\|^2 dt + C \varepsilon \int_{0}^{t} \|\tilde{v}_{1x_1x_1}\|_{L^\infty}^2\|\nabla\Psi\|^2 dt\qquad\qquad\qquad\qquad\qquad\qquad\qquad \\
	&
	\leq \frac{\varepsilon}{160} \int_{0}^{t} \|\nabla^2\xi_{x_1}\|^2 dt + C_T \varepsilon^{\frac{13}{6}}|\ln\varepsilon|^{-11}.
	\end{align*}
	By Sobolev's inequality, H\"{o}lder's inequality and Young's inequality, one can get
	\begin{align*}
	&C\Big|\int_{0}^{t}\int_{\Omega} \frac{2\lambda\varepsilon}{\rho\theta^2}\varphi_{x_i}(\div\Psi)^2\nabla\varphi\cdot\nabla\xi_{x_i} dxdt\Big| \\
    &\leq C\varepsilon \int_{0}^{t} \|\nabla\Psi\|_{L^\infty}\|\nabla\Psi\|_{L^6}\|\nabla\varphi\|_{L^6}^2\|\nabla^2\xi\| dt \\
	&\leq C\varepsilon \int_{0}^{t} \left(\|\nabla\Psi\|^{1/2}\|\nabla^2\Psi\|^{1/2} + \|\nabla^2\Psi\|^{1/2}\|\nabla^3\Psi\|^{1/2}\right)\|\nabla\Psi\|_1\|\nabla\varphi\|_1^2\|\nabla^2\xi\| dt \\
	&\leq C\varepsilon \int_{0}^{t} \|\nabla\Psi\|^{1/2}\|\nabla^2\Psi\|^{1/2}\|\nabla\Psi\|_1\|\nabla\varphi\|_1^2\|\nabla^2\xi\| dt + \frac{\varepsilon}{160} \int_{0}^{t} \|\nabla^3\Psi\|^2 dt \\
	&\quad + C\varepsilon \int_{0}^{t} \|\nabla^2\Psi\|^{2/3}\|\nabla\Psi\|_1^{4/3}\|\nabla\varphi\|_1^{8/3}\|\nabla^2\xi\|^{4/3} dt \\
	& \leq \frac{\varepsilon}{160} \int_{0}^{t} \|\nabla^3\Psi\|^2 dt + C_T\left[\left(\varepsilon^{\frac 52} + \varepsilon^2\right)|\ln\varepsilon|^{-3} + \left(\varepsilon^3 + \varepsilon^2\right)|\ln\varepsilon|^{-4}\right]\sup_{0 \leq t \leq t_1(\varepsilon)}\|\nabla\varphi\|_1^2.
	\end{align*}
The rest of the terms in $N_i(t,x)$   can be treated  in a similar fashion. 

With all these estimates for \eqref{24} in hand, together with the elliptic estimates $\|\triangle\Psi\| \sim\|\nabla^2\Psi\|$, $\|\triangle\xi\| \sim\|\nabla^2\xi\|, \|\nabla\triangle\Psi\| \sim \|\nabla^3\Psi\|$ and $\|\nabla\triangle\xi\| \sim \|\nabla^3\xi\|$, we can take $\varepsilon$ sufficiently small and   complete the proof of Lemma \ref{lemma4.3}.
\end{proof}

Finally, taking 
$\varepsilon$ sufficiently small,  e.g.,  $0<\varepsilon \leq \varepsilon_2$ for some $\varepsilon_2>0$, we can close the {\it a priori} assumptions \eqref{PA} and achieve the desired {\it a priori} estimates \eqref{PRO3.2}.  The proof of  the 
Proposition \ref{proposition3.2} is completed.

 \bigskip

 \section*{Acknowledgement}
 The research  of D. Wang was  partially supported by the
National Science Foundation under grant  DMS-1907519.
The research of Y. Wang was supported by NSFC grants No. 11671385 and 11688101 and CAS Interdisciplinary Innovation Team.

 \bigskip



\begin{thebibliography}{00}

\bibitem{BTW18} Bardos, C., Titi, E. and Wiedemann, E.: \textit{Onsager's conjecture with physical boundaries and an application to the viscosity limit.} Commun. Math. Phys. {\bf 370}, 291-310 (2019).

\bibitem{BB} Bianchini, S., Bressan, A.: \textit{Vanishing viscosity solutions of nonlinear hyperbolic systems.} Ann. of Math.(2) {\bf 161}, 223-342 (2005).

\bibitem{BCK} Brezina, J., Chiodaroli, E. and Kreml, O.: \textit{Contact discontinuities in multi-dimensional isentropic Euler equations,} Electron. J. Differential Equations {\bf 94}, 1-11 (2018).

\bibitem{CC} Chen, G.-Q., Chen, J.: \textit{Stability of rarefaction waves and vacuum states for the multidimensional Euler equations.} J. Hyperbolic Differ. Equ. {\bf 4} (1), 105-122 (2007).

\bibitem{CG18a} Chen, G.-Q., Glimm, J.: \textit{Kolmogorov-type theory of compressible turbulence and inviscid limit of the Navier-Stokes equations in ${\mathbb R}^3$.}  Phys. D {\bf 400}, 132138, 10 pp. (2019).

\bibitem{CLQ18} Chen, G.-Q., Li, S. and Qian, Z.: \textit{The inviscid limit of the Navier-Stokes equations with kinematic and Navier boundary conditions.} arXiv:1812.06565 [math.AP], 2018.

\bibitem{CP} Chen, G.-Q., Perepelitsa, M.: \textit{Vanishing viscosity limit of the Navier-Stokes equations to the Euler equations for compressible fluid flow.} Comm. Pure Appl. Math. {\bf 63}, 1469-1504 (2010).

\bibitem{CDK} Chiodaroli, E., De Lellis, C. and Kreml, O.: \textit{Global ill-posedness of the isentropic system of gas dynamics.} Comm. Pure Appl. Math. {\bf 68} (7), 1157-1190 (2015).

\bibitem{CK} Chiodaroli, E., Kreml, O.: \textit{Non-uniqueness of admissible weak solutions to the Riemann problem for isentropic Euler equations.} Nonlinearity {\bf 31}, 1441-1460 (2018).

\bibitem{CEIV} Constantin, P., Elgindi, T., Ignatova, M., and Vicol, V.: \textit{Remarks on the inviscid limit for the Navier-Stokes equations for uniformly bounded velocity fields.} SIAM J. Math. Anal. {\bf 49}, 1932-1946 (2017).

\bibitem{DS} De Lellis, C., Szekelyhidi Jr., L.: \textit{The Euler equations as a differential inclusion.} Ann. of Math. (2), {\bf 170} (3), 1417-1436 (2009).

\bibitem{DN2018} Drivas, T. D., Nguyen, H. Q.: \textit{Remarks on the emergence of weak Euler solutions in the vanishing viscosity limit.} J. Nonlinear Sci. {\bf 29}, 709-721 (2019).

\bibitem{FK} Feireisl, E., Kreml, O.: \textit{Uniqueness of rarefaction waves in multidimensional compressible Euler system.} J. Hyperbolic Differ. Equ. {\bf 12} (3), 489-499 (2015).

\bibitem{FKV} Feireisl, E., Kreml, O., Vasseur, A.: \textit{Stability of the isentropic Riemann solutions of the full multidimensional Euler system.} SIAM J. Math. Anal. {\bf 47} (3), 2416-2425 (2015).


\bibitem{GX} Goodman, J., Xin, Z.-P.: \textit{Viscous limits for piecewise smooth solutions to systems of conservation laws.} Arch. Rational Mech. Anal. {\bf 121}, 235-265 (1992).

\bibitem{HL} Hoff, D., Liu, T. P.: \textit{The inviscid limit for the Navier-Stokes equations of compressible, isentropic flow with shock data.} Indiana Univ. Math. J. {\bf 38}, 861-915 (1989).

\bibitem{HJW} Huang, F. M., Jiang, S. and Wang, Y.: \textit{Zero dissipation limit of full compressible Navier-Stokes equations with a Riemann initial data.} Commun. Inf. Syst. {\bf 13}, 211–246 (2013).

\bibitem{H-L-W} Huang, F. M., Li, M. J. and Wang, Y.: \textit{Zero dissipation limit to rarefaction wave with vacuum for one-dimensional compressible Navier-Stokes equations.} SIAM J. Math. Anal. {\bf 44}, 1742-1759 (2012).

\bibitem{H-L} Huang, F. M., Li, X.: \textit{Zero dissipation limit to rarefaction wave for the 1-D compressible Navier-Stokes equations.}  Chin. Ann. Math. {\bf 33}, 385-394 (2012).

\bibitem{HWY-1} Huang, F. M., Wang, Y. and Yang, T.: \textit{Fluid dynamic limit to the Riemann solutions of Euler equations: I. Superposition of rarefaction waves and contact discontinuity.} Kinet. Relat. Models {\bf 3}, 685-728 (2010).

\bibitem{HWY-2} Huang, F. M., Wang, Y. and Yang, T.: \textit{Vanishing viscosity limit of the compressible Navier-Stokes equations for solutions to a Riemann problem.} Arch. Rational Mech. Anal. {\bf 203}, 379-413 (2012).

\bibitem{HWWY} Huang, F. M., Wang, Y., Wang, Y. and Yang, T.: \textit{The limit of the Boltzmann equation to the Euler equations for Riemann problems.} SIAM J. Math. Anal. {\bf 45}, 1741-1811 (2013).

\bibitem{JNS} Jiang, S., Ni, G. X. and Sun, W. J.: \textit{Vanishing viscosity limit to rarefaction waves for the Navier-Stokes equations of one-dimensional compressible heat-conducting fluids.} SIAM J. Math. Anal. {\bf 38}, 368-384 (2006).

\bibitem{KV} Kang, M.-J., Vasseur A.: \textit{Uniqueness and stability of entropy shocks to the isentropic Euler system ina class of inviscid limits from a large family of Navier-Stokes systems.} Invent. Math., published online, https://doi.org/10.1007/s00222-020-01004-2, (2020).

\bibitem{KV1} Kang, M. J., Vasseur, A.: \textit{Well-posedness of the Riemann problem with two shocks for the isentropic Euler system in a class of vanishing physical viscosity limits.} https://arxiv.org/pdf/2010.16090.


\bibitem{Kato84} Kato, T.: \textit{Remarks on zero viscosity limit for non-stationary Navier-Stokes flows with boundary}. In Seminar on nonlinear partial differential equations (Berkeley, Calif., 1983), Vol. 2, Math. Sci. Res. Inst. Publ., 85-98. Springer, New York, 1984.

\bibitem{LXY} Lai, N. A., Xiang, W., Zhou, Y.: \textit{Global instability of the multi-dimensional plane shocks for the isothermal flow.} https://arxiv.org/abs/1807.07386.


\bibitem{L} Lax, P. D.: \textit{Hyperbolic systems of conservation laws, II.} Comm. Pure Appl. Math. {\bf 10}, 537-566 (1957).


\bibitem{LWW2} Li, L.-A., Wang, D.,  and Wang, Y.: \textit{Vanishing viscosity limit to the planar rarefaction wave for the two-dimensional compressible Navier-Stokes equations.} Commun. Math. Phys. 376, 353-384 (2020).

\bibitem{LWW} Li, L.-A., Wang, T. and Wang, Y.: \textit{Stability of planar rarefaction wave to 3D full compressible Navier-Stokes equations.} Arch. Rational Mech. Anal. {\bf 230}, 911-937 (2018).

\bibitem{LW} Li, L.-A., Wang, Y.: \textit{Stability of the planar rarefaction wave to the two-dimensional compressible Navier-Stokes equations.} SIAM J. Math. Anal. {\bf 50}, 4937-4963 (2018).



\bibitem{M} Ma, S.: \textit{Zero dissipation limit to strong contact discontinuity for the 1-D compressible Navier-Stokes equations.} J. Diff. Equa. {\bf 248}, 95-110 (2010).

\bibitem{KM} Markfelder, S., Klingenberg, C.: \textit{The Riemann problem for the multidimensional isentropic system of gas dynamics is ill-posed if it contains a shock.} Arch. Rational Mech. Anal. {\bf 227}, 967-994 (2018).

\bibitem{MN}   Matsumura, A., Nishida, T.: The initial value problem for the equations of motion of viscous and heat-conductive gases. J. Math. Kyoto Univ. {\bf 20}, 67-104 (1980).

\bibitem{Masmoudi2007} Masmoudi, N.: \textit{Remarks about the inviscid limit of the Navier-Stokes system.} Commun. Math. Phys. {\bf 270}, 777-788 (2007).

\bibitem{So} Solonnikov, V. A.: \textit{On solvability of an initial-boundary value problem for the equations of motion of a viscous compressible fluid.} in: Studies on Linear Operators and Function Theory. Vol. 6 [in Russain], Nauka, Leningrad, 1976,128-142.

\bibitem{WW} Wang, T., Wang, Y.: \textit{Nonlinear stability of planar rarefaction wave to the three-dimensional Boltzmann equation.} Kinet. Relat. Models {\bf 12}, no. 3, 637-679  (2019).

\bibitem{W-2} Wang, Y.: \textit{Zero dissipation limit of the compressible heat-conducting Navier-Stokes equations in the presence of the shock.}  Acta Math. Sci. Ser. B. {\bf 28}, 727-748 (2008).

\bibitem{X-1} Xin, Z.-P.: \textit{Zero dissipation limit to rarefaction waves for the one-dimensional Navier-Stokes equations of compressible isentropic gases.} Comm. Pure Appl. Math. {\bf 46}, 621-665 (1993).


\bibitem{Y} Yu, S.-H.: \textit{Zero-dissipation limit of solutions with shocks for systems of conservation laws.} Arch. Rational Mech. Anal. {\bf 146}, 275-370 (1999).
























\end{thebibliography}
\end{document}